\documentclass[reqno,12pt]{amsart}

\usepackage{geometry}    %12pt
\usepackage[hidelinks,colorlinks=false,linkcolor=false,urlcolor=false]{hyperref}
\makeatletter
\def\@cite#1#2{[\textbf{#1\if@tempswa , #2\fi}]}
\def\@biblabel#1{[\textbf{#1}]}
\makeatother
\usepackage[english]{babel}

\usepackage{blindtext}
\usepackage[utf8]{inputenc}

\usepackage{bbm}
\usepackage{tikz,tikz-3dplot }\usetikzlibrary{patterns}
\usepackage{graphicx}
\usepackage{wrapfig}

\usepackage{caption}
\usepackage{subcaption}

\usepackage{multicol}
\usepackage{xspace,amsmath,amsthm,amssymb,enumerate,color,esint,paralist}
\usepackage{datetime}

\usepackage[capitalise,nameinlink,noabbrev]{cleveref}
\usepackage{nicefrac}

\usepackage{appendix}
\newtheorem{theorem}{Theorem}[section]
\newtheorem*{theorem*}{Theorem}
\newtheorem{lemma}[theorem]{Lemma}

\newtheorem{corollary}[theorem]{Corollary}
\newtheorem*{claim*}{Claim}
\theoremstyle{definition}
\newtheorem{remark}[theorem]{Remark}
\newtheorem*{remark*}{Remark}

\newtheorem{setting}[theorem]{Setting}
\newtheorem{example}[theorem]{Example}
\newtheorem{model}[theorem]{Model}
\newtheorem*{examples*}{Examples}

\crefname{lemma}{Lemma}{Lemmas}
\crefname{proposition}{Proposition}{Propositions}
\crefname{theorem}{Theorem}{Theorems}
\crefname{claim}{Claim}{Claims}
\crefname{enumi}{Item}{Items}
\creflabelformat{enumi}{(#2\textup{#1}#3)}
\crefname{equation}{}{}
\crefname{assumption}{Setting}{Settings}
\crefname{setting}{Setting}{Settings}
\crefname{definition}{Definition}{Definitions}
\crefname{remark}{Remark}{Remarks}
\crefname{figure}{Figure}{Figures}
\crefname{model}{Model}{Models}

\crefname{chapter}{Chapter}{Chapters}
\Crefname{chapter}{chapter}{chapters}

\crefname{section}{Section}{Sections}
\Crefname{section}{section}{sections}

\crefname{subsection}{Subsection}{Subsections}
\Crefname{subsection}{subsection}{subsections}
\Crefname{theorem}{theorem}{theorems}

 \newcommand{\Laplace}{\mathop{}\!\mathbin\bigtriangleup}
\newcommand{\defeq}{\leftarrow}

\newcommand{\SUP}{\vee}

\renewcommand{\limsup}{\varlimsup}

\renewcommand{\epsilon}{\varepsilon}

\newcommand{\funcNu}[1]{{\mathbf{n}_{#1}}}
\newcommand{\R}{\mathbbm{R}}\newcommand{\N}{\mathbbm{N}}
\newcommand{\Z}{\mathbbm{Z}}\newcommand{\1}{\mathbbm{1}}
\renewcommand{\P}{\mathbbm{P}}
\newcommand{\E}{\mathbbm{E}}

\newcommand{\contSmooth}[1]{\mathcal{S}_{#1}}
\newcommand{\smooth}[2]{\mathbf{S}_{#1,#2}}
\newcommand{\dualsmooth}[2]{\mathbf{S}^*}

\newcommand{\DirichletExt}[2]{\mathbf{D}}
\newcommand{\NeumannExt}[2]{\mathbf{N}}
\newcommand{\nodal}[2]{\varphi^{(#1)}_{#2}}

\newcommand{\norm}[3]{\left\|#1\right\|_{#2,#3}}
\newcommand{\normB}[3]{\left\|#1\right\|_{L^{#2}(#3)}}
\newcommand{\avnorm}[3]{\left|#1\right|_{#2,#3}}

\newcommand{\modi}[1]{\mathbf{M}}

\newcommand{\Stat}{\mathsf{Stat}_\P}

\newcommand{\fluxQ}{\mathsf{q}}
\newcommand{\poincareConst}{C_{\mathrm{PI}}}

\newcommand{\matA}{\mathsf{a}}
\newcommand{\omegaH}{{\omega_\mathrm{h}}}
\newcommand{\matAhom}{{\mathsf{a}_\mathrm{h}}}

\newcommand{\nablaBar}{{\bar{\nabla}}}
\newcommand{\enor}[1]{E^\nu_{R}}
\newcommand{\etan}[1]{{E}^\tau_{R}}
\newcommand{\unit}[1]{\mathsf{e}_{#1}}
\newcommand{\bterm}[4]{ \mathfrak{B}\left(#3,#4\right) }

\title[Liouville principle for degenerate RCM]
{A Liouville principle for the random conductance model under degenerate conditions}
\author[Tuan Anh Nguyen]{Tuan Anh Nguyen\\
AG Stochastische Analysis,  Universit\"at Duisburg-Essen, \\45117 Essen, Germany, Email: \url{tuan.nguyen@uni-due.de}
}
\thanks{Supported by the DFG Research
Training Group (RTG 1845) ”Stochastic Analysis with Applications in Biology, Finance
and Physics” and the Berlin Mathematical School (BMS)}
\begin{document}
\begin{abstract}
We consider a random conductance model on the $d$-dimensional lattice, $d\in[2,\infty)\cap\mathbb{N}$, where the conductances take values in $(0,\infty)$ and
are however not assumed to be bounded from above and below. We assume that
the law of the random conductances is  stationary and ergodic with respect to translations on $\mathbb{Z}^d$ and invariant with respect to reflections on $\mathbb{Z}^d$ and satisfies a similar moment bound as that by Andres, Deuschel, and Slowik (2015), under which a quenched FCLT  holds. We prove a first-order Liouville theorem. In the proof we construct the sublinear correctors in the discrete and adapt boundary estimates for harmonic extensions from the work in the continuum done by Bella, Fehrman, and Otto (2018) to the discrete.
\end{abstract}
\maketitle
\tableofcontents
\listoffigures
\section{Introduction}
\subsection{Motivations}
The classical Liouville theorem, one of the most beautiful results in mathematics, states that for every $d\in [2,\infty)\cap\N$, $k\in \N_0$
the space of harmonic functions $u\in  C^2(\R^d, \R)$ which satisfy  that
$\sup_{x\in\R^d}\left[ |u(x)|/(1+|x|^k)\right]<\infty$
 contains only polynomials of degree $k.$ Here, a function $u\in C^2(\R^d,\R)$ is called harmonic if its Laplacian vanishes, i.e.,
for all
$x\in\R^d$ it holds that $\sum_{i=1}^{d}(\partial_{ii}^2u)(x)=0$.
One also has the same result in the discrete case where $u\colon \Z^d\to\R$ is called harmonic when it holds for all $x\in\Z^d$ that $2du(x)=\sum_{y\in\Z^d\colon |y-x|=1}^{}u(y)$.
It is well-known that for all $d\in[2,\infty)\cap\N$  a function $u\in C^2(\R^d,\R)$ (or $u\colon \Z^d\to\R$) is harmonic
if and only if
$(u(W_t))_{t\in[0,\infty)}$ (or $(u(S_n))_{n\in\N_0}$, respectively) is a martingale where $(W_t)_{t\in[0,\infty)}$ is a $d$-dimensional standard Brownian motion
and $(S_n)_{n\in\N_0}$ is a $d$-dimensional simple random walk. 
Studying harmonic functions
is a classical topic and appears
 in several mathematical areas, e.g., analysis of partial differential equations, probability, numerical analysis etc.
Especially, discrete harmonic functions, 
which have been studied for long (see, e.g., \cite{LFC28,Hei49,Duf53}),
are recently objects of several researches with quite surprising and beautiful results (see, e.g.,
\cite{BLMS17,LM15,LM15a,LM15b}, in particular, see \cite{LM15} for a new proof of the Liouville theorem in the discrete).

However, we  would like to go beyond this classical setting:
we are interested in  Liouville-type properties when the standard Brownian motion or the simple random walk is replaced by several types of \emph{random motions in continuum and discrete random media}, e.g.,  in \cref{m1,m2} below.
\begin{model}
[Random walks among random conductances]\label{m1}Let $d\in[2,\infty)\cap\N$.
For every $x,y\in \Z^d$ we call $x,y$ nearest neighbours
and write $x\sim y$ if $|x-y|=1$ where $|\cdot|$ denotes the Euclidean norm. Let $\E ^d$ be the set of
\emph{unoriented nearest neighbour edges} given by
$\E ^d=\{\{x,y\}\colon x,y\in\Z^ d, x\sim y\}$.
Let
$(\Omega,\mathcal{F})$ be the measurable space given by
$(\Omega,\mathcal{F})= ([0,\infty)^{\mathbbm E^d},\mathcal{B}([0,\infty))^{\otimes\mathbbm E^d})$. For every $e\in \E ^d$, $\omega\in\Omega$
we call $\omega$ an environment and  $\omega_e=\omega(e)\in [0,\infty)$ the \emph{conductance} of the edge $e$ in the environment $\omega$.
In order to consider conductances as random variables let $\mu_e\colon \Omega\to \R$, $e\in \E^d$, be the functions which satisfy for all  $e\in \E^d$, $\omega\in\Omega$ that $\mu_e(\omega)=\omega_e$, for simplicity we write 
$\mu_{xy}(\omega)=\mu_{\{x,y\}}(\omega)$,
and let
$m_x$, $x\in\Z^d$, be the functions which satisfy for all $x\in\Z^d$, $\omega\in\Omega$ that
$
m_x (\omega)=\sum_{y\in\Z^d\colon y\sim x}\omega_{xy}
$.
For each
$\omega\in\Omega$ we consider a discrete time random walk 
$(Z_n)_{n\in\N_0}$ which
jumps from $x$ to $y$, $x,y\in\Z^d$, $x\sim y$, with probability $\mu_{xy}(\omega)/
m_x (\omega)$ 
and a continuous time random walk $(X_t)_{t\in [0,\infty)}$ which waits at $x$ an exponential time with means $m_x (\omega)$ and jumps to $y\in\Z^d$ with $|x-y|=1$ with probability $\mu_{xy}(\omega)/m_x (\omega)$ and whose generator
$\mathcal{L}^\omega\colon \R^{\Z^d}\to\R^{\Z^d}$ satisfies for all $u\colon \Z^d\to\R$, $x\in\Z^d$ that
$
(\mathcal{L}^\omega u)(x)= \sum_{y\sim x} \mu_{xy}(\omega) \left(u(y)-u(x)\right).
$
For every $\omega\in\Omega$ we call  a function $u\colon \Z^d\to \mathbb{R}$  $\omega$-harmonic 
if $(\mathcal{L}^\omega u)(x)=0$. Then for all $u\colon\Z^d\to\R $ the following are equivalent: 
(a)
 $u$ is $\omega$-harmonic; (b)
$(u(X_t))_{t\in[0,\infty)}$ is a martingale; (c)
$(u(Z_n))_{n\in\N_0}$ is a martingale.
\end{model}
\begin{model}
[Diffusions on random media]\label{m2}
Let $d\in[2,\infty)\cap\N$, let $(\Omega,\mathcal{F},\P)$ be a probability space, 
let $\mathrm{Sym}(d,\R)$ be the space of $\R^{d\times d}$ symmetric matrices,
let
 $\matA\colon\Omega\times \R^d\to  \mathrm{Sym}(d,\R)$ be measurable (and we write $\matA^\omega(x)=\matA(\omega,x)$),
for every $\omega\in\Omega$ let
$\mathcal{L}^{\omega} \colon C^2(\R^d,\R)\to C(\R^d,\R)$ be the
 elliptic differential operator that satisfies for all $u\in C^2(\R^d,\R)$, $x\in\R^d$ that
$
(\mathcal{L}^\omega u)(x)=( \nabla\cdot \matA^\omega\nabla u)(x) = \sum_{i,j=1}^d
[ \partial_i(\matA^\omega_{ij}\partial_j u)](x)
$,
and let $(Y_t)_{t\in[0,\infty)}$ be the diffusion process generated by $\mathcal{L}^\omega$.
We call
a function $u\in C^2(\R^d,\R)$
$\omega$-harmonic (or $\matA^\omega$-harmonic) if for all $x\in\R^d$ it holds that $(\mathcal{L}^\omega u)(x)=0.$
Then a function $u$ is $\omega$-harmonic if and only if $(u(Y_t))_{t\in[0,\infty)}$ is a martingale.
\end{model}

This paper is motivated by several Liouville-type results for harmonic functions in 
\cref{m1,m2}.
First, Benjamini, Duminil-Copin, Kozma, and Yadin~\cite{BDKY15} 
consider the case of \emph{simple random walks on Bernoulli supercritical percolation clusters}, a special case of \cref{m1}. 
In this case we let $\Omega=\{0,1\}^{\E^d}$ and  for every  $\omega\in\Omega$, $e\in\E^d$ we call $e$ 
an $\omega$-open edge if $\mu_e(\omega)=\omega(e)=1$  and we call
$e$ 
an $\omega$-closed edge if $\mu_e(\omega)=\omega(e)=0$. 
For every $p\in [0,1]$ let $\P_p$ be the probability distribution on $(\Omega,\mathcal{F})$ under which the random variables $\mu_e$, $e\in\E^d$, are independent and $\mathrm{Bernoulli}(p)$-distributed, i.e., $\P_p(\{\omega\in\Omega\mid\mu_e(\omega)=1\})=p$ and $\P_p(\{\omega\in\Omega\mid\mu_e(\omega)=0\})=1-p $, i.e.,
$\P_p=\mathrm{Bernoulli}(p)^{\otimes \E^d}$. 
It is well-known that  there exists 
$\mathfrak p_\mathrm{c}\in (0,1)$ such that for every $p\in(\mathfrak{p}_c,1]$,
$\P_p$-a.e. $\omega\in\Omega$
 it holds   that
the $\omega$-open edges form a unique infinite connected component $\mathcal{C}^\infty(\omega)$, which is often called a percolation cluster. In 
\cite{BDKY15} it is shown that for every $p\in (\mathfrak{p}_c,1]$,  $\P_p$-a.e. $\omega\in\Omega$ 
the space of $\omega$-harmonic functions $u\colon \mathcal{C}^\infty(\omega)\to\R$ with linear growth has dimension $d+1$,  the same dimension as the space of harmonic functions  with linear growth in a $d$-dimensional space. 
Note that their work answers
Question 3 in Berger and Biskup~\cite{BB07} on uniqueness of the harmonic embedding.

Armstrong and Dario~\cite{AD17} consider
the random conductance model on a supercritical percolation cluster under the assumption that the conductances are uniformly elliptic. More precisely, let 
$c\in (1,\infty)$, $p\in (\mathfrak{p}_c,1]$, let
$\P$ be a probability measure on 
$(\Omega,\mathcal{F})$ in
\cref{m1} and assume that the random variables
$\mu_e$, $e\in\E^d$, are independent and identically distributed under $\P$ with
$\P(\{\omega\in\Omega\mid c^{-1}\leq \mu_e(\omega)\leq c\})=p$ and $\P(\{\omega\in\Omega\mid\mu_e(\omega)=0\})=1-p$.
 Under this condition they
extend the result by Benjamini, Duminil-Copin, Kozma, and Yadin~\cite{BDKY15} to 
higher order Liouville-type results: for $\P$-a.e. $\omega\in\Omega$ 
the edges $e\in\E^d$ with $\mu_e(\omega)=\omega(e)>0$ form a unique infinite cluster $\mathcal{C}^\infty(\omega)$ and the space of harmonic functions $u\colon \mathcal{C}^\infty(\omega)\to\R$ 
with $u(x)=o(|x|^{k+1})$, $|x|\to\infty$, in some sense,
has the same dimension of polynomial $p\colon\R^d\to\R$ with
$p(x)=o(|x|^{k+1})$, $|x|\to\infty$.

 Gloria, Neukamm, and Otto~\cite{GNO14} consider \cref{m2} under the assumption that 
the distribution of $\matA$ under 
$\P$ is stationary and ergodic, i.e.,
for all $x\in\R^d$ it holds that $\matA$ and $\matA(x+\cdot)$ have the same law under $\P$, and $\matA$ is bounded from above and below,
i.e., let $\lambda\in(1,\infty)$ be a fixed real number and assume for every $x\in\R^d$, 
$\P$-a.e. $\omega\in\Omega$
that
$\lambda^{-1}\leq \matA(\omega,x)\leq \lambda $. 
Corollary~1 in~\cite{GNO14}
 proves that for every  $\alpha\in (0,1)$, $\P$-a.e. $\omega\in\Omega$ the space of $\omega$-harmonic functions 
$u\colon \Z^d\to \R$ with $u(x)=o(|x|^{1+\alpha})$, $|x|\to\infty$, in some sense has dimension $d+1.$ 
Fischer and Raithel~\cite{FR16} prove a similar result in the case of the haft-space and Fischer and Otto~\cite{FO16} prove higher Liouville-type results, all results here in \cref{m2}. Last but not least,
Bella, Fehrman, and Otto~\cite{BFO18} continue with \cref{m2} and extend the result in~\cite{GNO14} to the case when
the distribution of
$\matA$ is  stationary and ergodic  and there exist $p,q\in(1,\infty)$ with $1/p+1/q<2/d$ such that
$
\E[| \mu(\matA)|^ p] +\E[ |\lambda(\matA)|^{-q} ]<\infty
$
where 
$
\lambda (\matA)= \inf_{\xi\in \mathbb{R}^d}{\xi\cdot \matA\xi}/{|\xi|^2}
$ and $
\mu (\matA)= \sup_{\xi\in \mathbb{R}^d}{|\matA\xi|^2}/(\xi\cdot\matA\xi)
$ (see Theorem 1 in~\cite{BFO18}). 

In \cref{m1} in the case $\P$-a.s. $\mu_e>0$ this moment condition is known as
$\E [|\mu_e|^p]+\E [|\mu_e|^{-q}]<\infty$ and $1/p+1/q<2/d$, or briefly
\emph{the $(p,q)$-moment condition}. Note that this (combined with ergodicity)
is a sufficient condition for a \emph{quenched invariance principle} to hold
(see Andres, Deuschel, and Slowik~\cite{ADS15}, Deuschel, Slowik, and the author~\cite{DNS17}, and Bella and Sch\"affner~\cite{BS19}). 
We therefore expect a Liouville-type result in \cref{m1} for \emph{purely ergodic environments under degenerate conditions}
 in at least two cases: 
\begin{enumerate}[(A)]
\item \label{c00} the conductances rely on a $d$-dimensional lattice and \item \label{c02} the conductances rely on a random graph, e.g., a $d$-dimensional percolation cluster. 
\end{enumerate}
Comparing those approaches we see that the assumption on the growth of the given harmonic functions
$u(x)=o(|x|^{1+\alpha})$ used by 
Gloria, Neukamm, and Otto~\cite{GNO14} and 
Bella, Fehrman, and Otto~\cite{BFO18}
is slightly stronger than the assumption 
$u(x)=o(|x|^{1+1})$ used by Armstrong and Dario~\cite{AD17}.
% However, this does not mean that the results by Gloria, Neukamm, and Otto~\cite{GNO14} and 
%Bella, Fehrman, and Otto~\cite{BFO18} are weaker than that by Armstrong and Dario~\cite{AD17}. Those two different approaches distinguish from each other essentially by assumptions on (i) the ergodicity of the environment and (ii) the growth of the given harmonic function. 
However, note that
Armstrong and Dario~\cite{AD17} work with independence, 
the strongest assumption on ergodicity,
while 
 Gloria, Neukamm, and Otto~\cite{GNO14} only assume pure ergodicity and later
Bella, Fehrman, and Otto~\cite{BFO18} even assume degenerate moment conditions. The stronger growth assumption here can be understood as a compensation for the purely ergodicity of the environment.
Especially, issues \eqref{c00} and \eqref{c02} above are still unsolved.
In fact, the main result of this paper, 
\cref{a12a} below, deals with case \eqref{c00} above, i.e., when
the conductances rely on a $d$-dimensional lattice, i.e., $\Omega= (0,\infty)^{\E^d}$.

Since the quenched invariance principle has been successfully extended from
case \eqref{c00} to case \eqref{c02} (see \cite{ADS15,DNS17}) we expect that a Liouville-type theorem is still true in case \eqref{c02} under the growth assumption 
$u(x)=o(|x|^{1+\alpha})$, $|x|\to\infty$, $\alpha\in (0,1)$. In this case we may need new ideas, since the approach adapted from Bella, Fehrman, and Otto \cite{BFO18} relies very much on the geometric structure of Euclidean lattices and may not work in case \eqref{c02}. In fact, to the best of the author's knowledge, there is no paper which constructs \emph{the second order corrector $\sigma$} in this case. Another challenge is to deal with $L^p$-estimates on boundaries, $p\in(1,\infty)$. 

So far,there are some first results which deal with case \eqref{c02}. 
Sapozhnikov \cite{Sap14} extends the result by
Benjamini, Duminil-Copin, Kozma, and Yadin~\cite{BDKY15} to 
percolation models with slowly decaying correlations (see the assumptions in 
Section 1.2.1 in  \cite{Sap14} and note that later it has been relaxed, see Corollary 6.3 in \cite{AS18}).
Examples for those models are 
random interlacements, level sets of Gaussian free fields, (cf. Section 1.2.2 in \cite{Sap14}) and level sets of some convex gradient fields (cf. \cite{Rod16}).
Theorem 1.18 in  \cite{Sap14}
  proves that for all those models the space of harmonic functions on the infinite cluster with  linear growth has dimension $d+1$ and
Theorem 1.6 and Corollary 1.14 in  \cite{Sap14}  prove that the space of harmonic functions with polynomial growth have bounded dimension. 

Besides the open question on a Liouville result under
the growth assumption 
$u(x)=o(|x|^{1+\alpha})$, $|x|\to\infty$, $\alpha\in (0,1)$,
it is also interesting to know whether the Liouville result by Armstrong and Dario \cite{AD17} under the growth assumption $u(x)=o(|x|^{k+1})$, $|x|\to\infty$, $k\in\N$, still holds for those percolation models.

\begin{table}\centering\small 
\begin{tabular}{|l|l|l|l|}
\hline 
\textbf{Year}&\textbf{Authors}						& \textbf{Environment} 		&\textbf{Growth } \\			 
\hline 
2011&Benjamini, Duminil-Copin,    			&simple 
												random walks	&$O(|x|)$\\ 
	&Kozma, and Yadin \cite{BDKY15}			&  on infinite 
												Bernoulli clusters	&\\ 
%\hline
%2015&Marahrens and Otto \cite{MO15}			&bounded from above and below 		&$O(|x|^\alpha)$\\
%	&										& and mixing 
%											on euclidean lattices	& \\
\hline
2014&Gloria, Neukamm, and Otto \cite{GNO14}	&bounded from above and below 		&$O(|x|^{1+\alpha})$\\
	&										& and stationary 
											on Euclidean spaces	&\\
\hline							
2014&Sapozhnikov \cite{Sap14} 				&simple random walks on infinite  & $O(|x|)$ \\
	&										& clusters in a class of percolations & \\
\hline
2015&Fischer and Otto \cite{FO16}	&bounded from above and below 		&$O(|x|^{k+\alpha})$\\
	&										& and mixing
											on Euclidean spaces	&\\
\hline
2015&Fischer and Raithel \cite{FR16}	&bounded from above and below 		&$O(|x|^{1+\alpha})$\\
	&			(haft space)			& and mixing
											on Euclidean spaces	&\\
\hline
2016&Bella, Fehrman, and Otto \cite{BFO18}	&unbounded with $(p,q)$-condition 		&$O(|x|^{1+\alpha})$\\
	&										& and stationary 
											on euclidean spaces	&\\
\hline 
2016& Armstrong and Dario \cite{AD17} 		& bounded from above and below  
												 		&$o(|x|^{k+1})$ \\
	&										& on infinite Bernoulli
												 clusters		&\\ 	
\hline
\end{tabular}\caption{Some known Liouville-type results}\label{i10}
\end{table}

\subsection{Main result}In this subsection we formulate our main result.
\begin{setting}[Random conductances on a $d$-dimensional lattice]\label{a01}Let $d\in [2,\infty)\cap\N$.
 %\\
For every $x,y\in \Z^d$ we call $x,y$ nearest neighbours
and write $x\sim y$ if $|x-y|=1$ where $|\cdot|$ is the Euclidean norm. Let $\E ^d$ be the set 
given by
$\E ^d=\{\{x,y\}\colon x,y\in\Z^ d, x\sim y\}$. %\\
Let
$(\Omega,\mathcal{F})$
be the   measurable space given by
$(\Omega,\mathcal{F})= ((0,\infty)^{\E^d},\mathcal{B}((0,\infty))^{\otimes\E^d})$ and
for every $e\in \E ^d$, $\omega\in\Omega$ we write 
$\omega_e=\omega(e)\in (0,\infty)$.
Let $\mu_e\colon \Omega\to(0,\infty)$, $e\in\E^d$, be the functions which satisfy for all $e\in \E^d$ that $ \mu_e(\omega)=\omega_e$ and for simplicity we write 
$\mu_{xy}(\omega)=\mu_{\{x,y\}}(\omega)$.
%To lighten the notation for every $(x,y)\in \E^d_\pm $ we also write   $\omega_{xy}=\omega_{yx}=\omega(\{x,y\})$.
Let $\tau_a\colon \Omega\to\Omega$, $a\in \Z^d$,
be the operators which satisfy for all $a\in \Z^d$, $\omega\in\Omega$, 
$x,y\in \Z^d$ with $\{x,y\}\in \E ^d$ that
$(\tau_a\omega)(\{x,y\})= \omega(\{a+x,a+y\})$.  %\\
A probability measure $\mathbb{P}$ on $(\Omega,\mathcal{F})$ is called \emph{stationary and ergodic} if
it holds for all $A\in\mathcal{F}$, $x\in\Z^d$ that $\mathbb{P}(A)= \mathbb{P}(\tau_x A)$ and 
it holds for all $A\in\mathcal{F}$ with $\mathbb{P}(\forall x\in \Z^d\colon  \tau_x(A)=A)=1$  that $\P(A)\in \{0,1\}$.

\end{setting}
\begin{theorem}[First-order Liouville principle]
\label{a12a}
Assume \cref{a01}.
Let
$\alpha\in (0,1)$,
$p,q\in (1,\infty]$ satisfy that $1/p+1/q\leq 2/d$, 
we write $p/(p-1)=1$ if $p=\infty$,
let $\mathbb{P}$ be a stationary and ergodic probability measure on $(\Omega,\mathcal{F})$, 
assume that $\P$ is invariant under reflections on $\Z^d$,
assume for all $e\in\E^d$ that $\|\mu_e\|_{L^p(\Omega,\R)}+\|\mu^{-1}_e\|_{L^q(\Omega,\R)}<\infty$, and
for every $\omega\in \Omega$
let $\mathcal{L}^\omega\colon \R^{\Z^d}\to\R^{\Z^d}$ be the operator which satisfies for all $u\colon \Z^d\to \R$, $x\in\Z^d$ that
$(\mathcal{L}^\omega u)(x)=\sum_{y\in\Z^d\colon y\sim x} \omega_{xy}(u(y)-u(x))$ 
and let $\mathbf{S}(\omega) $ be the set of all functions $u\colon \Z^d\to\R$ with the properties that for all $x\in\Z^d$ it holds that 
\begin{align}
(\mathcal{L}^\omega u)(x)=0\quad\text{and}\quad\lim_{R\to\infty}
\frac{1}{
R^{1+\alpha} }
\left[
\sum_{y\in\Z^d\colon |y|_\infty< R}^{}
|u(y)|^{\frac{2p}{p-1}}\right]^{\frac{p-1}{2p}}=0. 
\end{align}
Then it holds for $\P$-a.e. $\omega\in\Omega$ that 
$\mathbf{S}(\omega)$ is a linear space and
the dimension of $\mathbf{S}(\omega)$ is $d+1$.
\end{theorem}
In the case when
$\mu_e$, $e\in\E^d$, are i.i.d.,  it is easy to construct several examples that satisfy the assumptions of \cref{a12a}. 
A more interesting example is given in
\cref{x30} below.

\begin{example}[Discrete massless Gaussian free field 
\cite{BS12}]\label{x30}Assume \cref{a01}, assume that  $d\geq 3$, 
let $(\varphi_x)_{x\in \Z^d} $ be a discrete massless Gaussian free field defined on some probability space $(\Omega',\mathcal{F}',\P') $,
let $\P$ be the probability measure on $(\Omega,\mathcal{F})$ which
satisfies for all $x,y\in\Z^d$, $x\sim y$   that 
$\P(\mu_{xy}\in \cdot)=\P'(\exp(\varphi_x+\varphi_y)\in\cdot)$. Then  $\P$ is invariant under reflections on $\Z^d$ and it holds for all $p,q\in (1,\infty)$ that $\|\mu_e\|_{L^p(\Omega,\R)}+\|\mu^{-1}_e\|_{L^q(\Omega,\R)}<\infty$.
\end{example}
 \cref{a12a} can be viewed as a discrete analogue of 
the result by Bella, Fehrman, and Otto (Theorem 1 in~\cite{BFO18}). 
As in the continuum case
\cref{a12a} is a direct consequence of 
\begin{enumerate}[(i)]
\item 
 \emph{the existence and sublinearity of the correctors}~$\phi,\sigma$, 
\item a purely deterministic result on discrete PDEs called
\emph{the excess decay},
\end{enumerate}
and  the fact that when the distribution of the conductances is invariant under reflections, then the homogenized matrix $\matAhom$ (i.e., the covariance matrix of the limiting Brownian motion in the quenched invariance principle \cite{ADS15,DNS17}) is a diagonal matrix (cf. Items~(iii) in Theorem~4.6 in De Masi et. al.~\cite{DMFGW89}). 

The main novelty of this paper is that 
we combine tools from numerical analysis and a duality argument to adapt boundary estimates   from the continuum to the discrete.
 Furthermore, using a density argument we slightly relax the condition $1/p+1/q< 2/d$ in \cite{BFO18} to $1/p+1/q\leq  2/d$.
Note that using estimates on boundaries may reduce the dimension and therefore relax some moment assumptions, which may lead to quite surprising results. E.g., Bella and Sch\"affner \cite{BS19} prove the quenched invariance principle for the random conductance model under $(p,q)$-condition with $1/p+1/q<2/(d-1)$. We therefore believe that the condition we use here $1/p+1/q\leq 2/d$ and $p,q\in (1,\infty]$ may still be relaxed.
Moreover, we expect that, up to some technicalities, this paper might be extended to the case when $\matAhom$ is not a diagonal matrix, in which tools in the discrete, e.g., borrowed from numerical analysis, are usually not readily developed. 
Unfortunately, a  challenge here is to deal with $L^p$-estimates on boundaries, $p\in(1,\infty)$, which is, in the non-diagonal case, very hard to formulate and prove.
In the diagonal case it is proved in a paper by the author \cite{Ngu19}.
\medskip
\paragraph*{\it\bf Structure of this paper}
 \cref{y14c}  sketches the proof of \cref{a12a} and discusses the main difficulties. \cref{a07c} discusses the construction of the first and second order correctors in the discrete case.
The remaining technical details are worked out in the rest of the paper.  \cref{x15b} provides formal calculations for the energy of the homogenization error in the discrete case. \cref{b17}, an interesting excursion to numerical analysis, discusses how to smooth functions on a discrete surface, which is an important ingredient to adapt the proof by Bella, Fehrman, and Otto~\cite{BFO18}. \cref{b47c} verifies in details boundary estimates  which are sketched in \cref{y14c}.

\subsection*{Acknowledgement}This paper is based on a part of the author's dissertation~\cite{Ngu17} written  under supervision of Jean-Dominique Deuschel at Technische Universit\"at Berlin. The author thanks Benjamin Fehrman and Felix Otto for useful discussions and for sending him the manuscript of their paper~\cite{BFO18}.

\section{Sketch of the proof of the main theorem}\label{y14c}
\subsection{Settings and main ingredients of the proof}
Together with  \cref{a01}, \cref{c10} below are used throughout this section and the rest of this paper.
\begin{setting}[Basic notation]\label{c10}
Denote by $(\unit{i})_{i\in [1,d]\cap\N}$ (as column vectors) the standard basis of $\R^d$. 
For every $x\in\R^d$, $i\in [1,d]\cap\N$ let $x_i$ be the $i$-th coordinate of $x$ if  no confusion arises.
For every $x,y\in\R^d$
let
  $x\cdot y$ be
the standard scalar product of $x$ and $y$, i.e., $x\cdot y= \sum_{i=1}^d x_iy_i$,
let $|x|$ be the Euclidean norm of $x$, i.e., $|x|^2= \sum_{i=1}^{d}|x_i|^2$, and let $|x|_\infty$ be the maximum norm of $x$, i.e., $|x|_\infty=\max_{i=1}^d |x_i|$.

For every $R\in \N$
let $C_R$, $\overline{C}_R$, $\partial C_R$,
$D_R$, $\overline{D}_R$, and $\partial D_R $
be the sets given by
$C_R=\{x\in \R^d\colon |x|_\infty< R\} $, $ \overline{C}_R=\{x\in \R^d\colon |x|_\infty\leq  R\} $, $\partial C_R=\{x\in\R^d\colon |x|_\infty= R\}$,
$D_R=C_R\cap \Z^d $, $  \overline{D}_R=\overline{C}_R\cap \Z^d$, and $ \partial D_R=\partial C_R\cap \Z^d$.

For every \emph{finite} non-empty set $A $ and every function $u$ defined on $A$ let $|A|$ be the cardinality of $A$,
let $\avnorm{u}{\infty}{A}$ and
$\norm{u}{\infty}{A}$ be the real  numbers given by 
$
\avnorm{u}{\infty}{A}=\norm{u}{\infty}{A}= \sup_{x\in A}|u(x)|
$,
and
let $\avnorm{u}{p}{A}$ and
$\norm{u}{p}{A}$, $p\in [1,\infty)$, be the real  numbers
 which satisfy for all $p\in [1,\infty)$ that
$|A|\avnorm{u}{p}{A}^p=\norm{u}{p}{A}^p= \sum_{x\in A}|u(x)| ^p.$

For every 
$A\subseteq \Z^d$,
 $u\colon A\to\R$, $i\in[1,d]\cap\N$ let 
$\nabla_iu\colon A\cap (A-\unit{i})\to\R$ be the function which satisfies for all $x\in A\cap (A-\unit{i})$  that
\begin{align}
(\nabla_i u)(x)= u(x+\unit{i})-u(x)\end{align} and let 
$\nabla_i^*u\colon A\cap (A+\unit{i})\to\R$ be the function
which satisfies for all $x\in A\cap (A+\unit{i})$  that
\begin{align}
(\nabla_i^* u)(x)= u(x-\unit{i})-u(x).\end{align}
For every 
$A\subseteq \Z^d$,
 $u\colon A\to\R^d$ let 
$\nabla u\colon A\cap(\cap_{i=1}^d (A-\unit{i}))\to\R^d$ be the function which satisfies for all $x\in A\cap (\cap_{i=1}^d(A-\unit{i})$
that 
\begin{align}
(\nabla u)(x)=\sum_{i=1}^{d} (\nabla _iu)(x)\unit{i} \in\R^d \end{align} and let
$\nabla^* u\colon A\cap(\cap_{i=1}^d (A+\unit{i}))\to\R^d$ be the function which satisfies for all $x\in A\cap(\cap_{i=1}^d (A+\unit{i})))$
that 
\begin{align}
(\nabla^*u)(x)=\sum_{i=1}^{d}(\nabla _i^*u)(x)\unit{i}\in\R^d .
\end{align}
For every 
$A\subseteq \Z^d$,
 $u\colon A\to\R$ let $\Laplace u \colon \cap_{i=1}^d(A\cap (A+\unit{i})\cap (A-\unit{i}))\to\R$
be the function which satisfies for all $x\in\cap_{i=1}^d(A\cap (A+\unit{i})\cap (A-\unit{i})) $ that 
\begin{align}
(\Laplace u)(x)= -(\nabla^*\cdot \nabla u)(x)=-\sum_{i=1}^d (\nabla_i^*\nabla_iu)(x).
\end{align}

For every $m\in \N$ and every measurable functions $f,g\colon \Omega\times \Z^d\to \R^m$ we say that $\P$-a.s. $f=g$ if for $\P$-a.e. $\omega\in\Omega$ and for every $x\in\Z^d$ it holds that $f(\omega,x)=g(\omega,x)$. %\\

For every probability measure $\P$ on $(\Omega,\mathcal{F})$ and
every $m\in\N$
let $\Stat(\R^m)$ be the set of all measurable functions 
 $ f\colon \Omega\times \Z^d\to \R^m$ which satisfy for 
$\P$-a.e. $\omega\in\Omega$ and every $x\in\R^d $ that $f(\omega,x)= f(\tau_x\omega,0)$. Functions in $\cup_{m\in\N}\Stat(\R^m)$ are often called \emph{stationary functions}.
\end{setting}

 \cref{a12,y14} below are the main ingredients of the proof of \cref{a12a}.
 \cref{a12}  is a discrete analogue of Lemma 1 in~\cite{BFO18} on
 \emph{existence and sublinearity of the correctors}~$\phi,\sigma$ and
\cref{y14} is adapted from
Theorem 2 in~\cite{BFO18}, the so-called
\emph{the excess decay}.
%Its proof is represented in \cref{a07c}.

\begin{corollary}[Existence and sublinearity of the correctors]\label{a12}	
Let $\mathbb{P}$ be a stationary and ergodic probability measure on $(\Omega,\mathcal{F})$, 
let $\mu_e\colon \Omega\to(0,\infty)$, $e\in\E^d$, be the functions which satisfy for all $e\in \E^d$ that $ \mu_e(\omega)=\omega(e)$,
let
$p,q\in (1,\infty]$ satisfy for all $e\in\E^d$ that
$1/p+1/q\leq 2/d$ and $\|\mu_e\|_{L^p(\Omega,\R)}+\|\mu^{-1}_e\|_{L^q(\Omega,\R)}<\infty$,
let $\matA\colon \Omega\times \Z^d\to \R^{d\times d}$ be the  function whose values are diagonal $\R^{d\times d}$ matrices and which satisfies for all $x\in\Z^d$, $\omega\in\Omega$, $i,j\in [1,d]\cap\N$ that
\begin{align}\label{a12b}
\matA_{ij}^\omega(x)=\omega(\{x,x+\unit{i}\})\1_{\{i\}}(j), 
\end{align}
and for every $r\in\{p,q\}$ we write $r/(r+1)=r/(r-1)=1$ if $r=\infty$.
Then there exist uniquely measurable functions 
$\phi_i,\fluxQ_{ij},\sigma_{ijk}\colon \Omega\times\Z^d\to \R$, $i,j,k\in [1,d]\cap\N$, and a matrix 
$\matAhom\in \R^ {d\times d}$
such that $\P$-a.s. for all $  i,j,k\in [1,d]\cap\N$
it holds that 
$\nabla\phi_i,\nabla\sigma_{ijk}\in\Stat(\R^d) $,
$\phi_i(\cdot,0)=\sigma_{ijk}(\cdot,0)=0$,
$
-\nabla^*\cdot\left( \matA\left(\unit{i}+\nabla\phi_i\right)\right)=0$,
$-\nabla^*\cdot \sigma_{i}=\fluxQ_i$,
$-\Laplace\sigma_{ijk}= \nabla_j\fluxQ_{ik}- \nabla _k\fluxQ_{ij}$, 
$\fluxQ_i=\matA\left(\nabla\phi_i+\unit{i}\right)-\matAhom\unit{i}$,
$
  \matAhom \unit{i}=\E[ \left(\matA\left(\nabla\phi_i+\unit{i}\right)\right)(\cdot,0) ],
$
$
\sigma_{ijk}=-\sigma_{ikj},
$
$
\E\!\left[ \left(\nabla\phi_i\cdot \matA\nabla\phi_i\right)(\cdot,0)\right]<\infty$,
$\E\!\left[ |(\nabla\phi_i)(\cdot,0)|^{2q/(q+1)}\right] <\infty$, 
$\E\!\left[ |(\nabla\sigma_{ijk})(\cdot,0)|^{2p/(p+1)}\right]<\infty,
$ and
\begin{align}
\lim_{R\to\infty} \left[\frac{1}{R}
\avnorm{\phi_i(\omega,\cdot )}{2p/(p-1)}{{D}_R}\right] =
\lim_{R\to\infty}\left[ \frac{1}{R}\avnorm{\sigma_{ijk}(\omega,\cdot)}{2q/(q-1)}{{D}_R} \right]=0.\label{d12}
\end{align}
\end{corollary}
.%
\begin{corollary}[The excess decay - deterministic version]\label{y14}
Let $\alpha\in (0,1)$, $K\in (1,\infty)$, 
$p,q\in (1,\infty]$,
$\Lambda\in (0,\infty)$ satisfy that $1/p+1/q\leq 2/(d-1)$.
For every $r\in\{p,q\}$ we write $r/(r+1)=r/(r-1)=1$ if $r=\infty$.
For every
$\omega\in \Omega$
let $\matA^\omega\colon \Z^d \to \R^{d\times d}$ be the function whose values are $\R^{d\times d}$ diagonal matrix and which satisfies
for all $x\in\Z^d$ that  $ \matA_{ij}(x)= \omega(\{x,x+\unit{i}\}) \1_{\{i\}}(j)$.
For every $\delta\in (0,1)$, $\omega\in \Omega$, $r\in \N$, $R\in [r,\infty)\cap\N$  let $\mathbf{C}(\omega,\delta,r,R)$ be the set of all couples $(\phi,\sigma)$, $\phi \colon \Z^d\to \R^d$,
$\sigma \colon \Z^d\to \R^{d\times d\times d}$,
with the properties that
\begin{enumerate}[i)]
\item 
 there exist a diagonal matrix $\matAhom\in \R^{d\times d}$ and a function
$\fluxQ\colon \Z^d\to \R^{d\times d}$ such that
it holds  for all $i,j,k\in [1,d]\cap\N$ that
$\fluxQ_i=\matA^\omega\left(\nabla\phi_i+\unit{i}\right)-\matAhom
\unit{i}$,
$-\nabla^*\cdot \fluxQ_i=0$,
$-\nabla^*\cdot \sigma_{i}=\fluxQ_i$,
$K^{-1}\leq (\matAhom)_{ii}\leq K $
and $\sigma_{ijk}=-\sigma_{ikj}$ and
\item 
it holds for all $\rho \in[r,R]\cap\N  $ that
$\max\left\{\avnorm{\phi}{\frac{2p}{p-1}}{D_\rho}, \avnorm{\sigma}{\frac{2q}{q-1}}{D_\rho}\right\}\leq \delta\rho $.

\end{enumerate}
For every $\omega\in \Omega$, $R\in\N$ let $\mathbf{H}(\omega,R)$ be the set of all functions $u\colon \overline{D}_R\to\R$ which satisfy for all $x\in D_R$ that $(\nabla^*\cdot \matA^\omega \nabla u)(x)=0$. 
For every $r\in\N$, $R\in [r,\infty)\cap\N$ let $\Omega(r,R)\subseteq\Omega$ be the set given by $\Omega(r,R)= \{\omega\in\Omega\mid\forall \rho\in [r,R]\cap\N\colon  \avnorm{\omega}{p}{E_\rho}+\avnorm{\omega^{-1}}{q}{E_\rho}<\Lambda\} $.
For every $R\in\N$, $\omega\in\Omega$, $u\colon \overline{D}_R\to\R$, $\phi\colon \Z^d\to \R$ let $\mathrm{Exc}(\omega,R,u,\phi)$ be the real number given by 
\begin{align}
\mathrm{Exc}(\omega,R,u,\phi)= \inf_{\xi\in\R^d}\avnorm{\matA^\omega (\nabla u -\xi_i (\unit{i}+\nabla\phi_i))\cdot (\nabla u -\xi_i (\unit{i}+\nabla\phi_i)) }{1}{D_R} 
\end{align}
where Eintein's notation was used.
Then there exist $C_0,C_1\in (0,\infty)$
such that for all $r\in\N$, $R\in [r,\infty)\cap\N$,
$\omega\in \Omega(r,R)$,
$(\phi,\sigma)\in \mathbf{C}(\omega,1/C_0,r,R) $, $u\in \mathbf{H}(\omega,R)$ it holds that
$\mathrm{Exc}(\omega,R,u, \phi)\leq C_1(r/R)^\alpha\mathrm{Exc}(\omega, R,u,\phi)$.
\end{corollary}
In \cref{c10,a12,y14} above we first consider the discrete derivatives  as \emph{functions acting on vertices}. This notation is inspired by some authors who study homogenization for discrete elliptic equations on $\Z^d$, e.g., 
\cite{GO11,GO12}. Note that in the notation of 
\cref{a01,c10} and \cref{a12b} it holds
for every $u\colon\Z^d\to\R$, $x\in\Z^d\to\R $, $\omega\in\Omega$  that 
\begin{align}
(\mathcal{L}^\omega u)(x)= -(\nabla^*\cdot \matA^\omega \nabla u)(x) ,\end{align}
which is a good way to adapt the notation 
in the continuum case (\cref{m2}) to the discrete.

\renewcommand{\smooth}[2]{\mathbf{S}}
We will essentially adapt the proofs in Bella, Fehrman, and Otto~\cite{BFO18}  in the continuum to the discrete.
Since discrete and continuum objects often have many similar properties, we only  work out in details arguments which really make issues in the discrete.

The construction of the correctors $\phi$ and $\sigma$ (i.e., the proof of \cref{a12}) is discussed
carefully in \cref{a07c}. Note that for sublinearity we do not use Sobolev's embedding and weak convergence as done in p. 1394 in~\cite{BFO18} which might be hard to understand in the discrete. Our argument relies on a density argument combined with Sobolev's inequality in the discrete (see \cref{b01}). Therefore, we only need to assume $1/p+1/q\leq 2/d$ rather than the strict inequality 
$1/p+1/q< 2/d$.

The proof of \cref{a12a} from \cref{a12,y14} is straightforward: One can easily adapt the proof of Theorem~1 in p. 1388 in~\cite{BFO18} into the discrete. Note that 
Lemma~3 in~p.~1385 in~\cite{BFO18}, an important argument in this proof, is just a simple combination of H\"older's inequality and
Cacciopolli's inequality, which is also well-known for discrete models (see, e.g., Proposition~4.1 in~\cite{DD05}, Section~2 in~\cite{Del98}, or Section~4.3 in~\cite{CG98}).

\subsection{Sketch of the proof of the excess decay}

In the rest of this section  we sketch the proof of \cref{y14} (and use therefore the notation given in there). It is now convenient to introduce
\cref{b66b} below that will be used throughout the rest of this paper.

\begin{setting}[Sets and functions of edges]\label{b66b}
Let $\E ^d_{\pm}$
be the set of \emph{oriented nearest neighbour edges} which satisfies that
$
\E ^d_{\pm}=\{(x,y)\mid x,y\in\Z^ d, x\sim y\}.$
 For every 
$A\subseteq \Z^d$,
 $u\colon A\to\R^d$ let $ \nablaBar _{\cdot }u\colon\{(x,y)\in\E^d_\pm \colon x,y\in A\}\to\R  $ be the function which satisfies
for all $ (x,y)\in\E^d_\pm $ with $x,y\in A$
 that
$\nablaBar_{(x,y)} u= u(y)-u(x)$ and we write $\nablaBar_{(x,y)} u= \nablaBar_{xy} u$ to lighten the notation.
For every $R\in \N$ let
$E_R$,
$\etan{R}$, $\enor{R}$, and  $\tilde \partial D_R$ be the sets given by
$E_R=\{(x,y)\in \mathbb E^d_\pm\mid  \tfrac12|(x+y)|_\infty<R\}$,
$
\etan{R}=\{(x,y)\in \E ^d_\pm \mid  x,y\in \partial D_R, x\sim y\}$, $\enor{R}=\{(x,y)\in\E^d_{\pm}\mid x\in  D_R,y\in \partial D_R\}$, and
$
\tilde \partial D_R=\{x\in\partial D_R\mid \exists y\in D_R\colon  x\sim y\}.$
For every $R\in\N$ let  $\funcNu{R}\colon \tilde{\partial}D_R\to\enor{R}$ be the mapping which satisfies for all $y\in \tilde{\partial}D_R$ that
$\funcNu{R}(y)\in \enor{R}$ is the unique element in $\enor{R}$ such that
$y$ is an endpoint of $\funcNu{R}(y)$
 and we call $\funcNu{R}$, $R\in\N$, \emph{the normal mappings}. 
 See \cref{f01} for an illustration of $\etan{R}, \enor{R}, \tilde{\partial}D_R,\funcNu{R} $.
 %\\
\end{setting}
\begin{figure}\center
\begin{tikzpicture}[scale=1.5]
\draw[very thin, dashed, gray,step =0.4] (-2,-2) grid (2,2);
\draw[very thin, dashed, gray,step =1] (3.6,0) grid (6.4,2);
\draw (3.6,2)--(6.4,2)[line width=2, color=red];
\draw [fill,red] (4,2) circle(.1);
\draw [fill,red] (5,2) circle(.1);
\draw [fill,red] (6,2) circle(.1);
\draw (4,1)--(4,1.8)[very thick, blue,->];
\draw (5,1)--(5,1.8)[very thick, blue,->];
\draw (6,1)--(6,1.8)[very thick, blue,->];
\node[below right] at (5,1) {$x$};
\node[below right] at (5,2) {$y$};
\node[left] at (5,1.5) {$\funcNu{R}(y)$};
\draw (-2,-2)--(-2,2)--(2,2) -- (2,-2)--(-2,-2)[very thick, color=red];
\foreach \x in {-1.6,-1.2,...,2} 
	\draw[very thick,->,color=blue] (\x,1.65)--(\x,1.90);
\foreach \x in {-1.6,-1.2,...,2} 	
	\draw[very thick,->,color=blue] (\x,-1.65)--(\x,-1.90) ;
\foreach \x in {-1.6,-1.2,...,2} 
	\draw[very  thick,->,color=blue] (1.65,\x)--(1.90,\x) ;
\foreach \x in {-1.6,-1.2,...,2} 	
	\draw[very thick,->,color=blue] (-1.65,\x)--(-1.90,\x);
\foreach \x in {-1.6,-1.2,...,2} 	
	\draw [red](2,\x) circle(.05)[fill];	
\foreach \x in {-1.6,-1.2,...,2} 	
	\draw [red](-2,\x) circle(.05)[fill];	
\foreach \x in {-1.6,-1.2,...,2} 	
	\draw[red] (\x,2) circle(.05)[fill];	
\foreach \x in {-1.6,-1.2,...,2} 	
	\draw [red](\x,-2) circle(.05)[fill];
\draw [red] (-2,-2) circle(.05)[fill,white];
\draw [red] (-2,2) circle(.05)[fill,white];
\draw [red] (2,-2) circle(.05)[fill,white];
\draw [red] (2,2) circle(.05)[fill,white];
\draw [red] (-2,-2) circle(.05);
\draw [red] (-2,2) circle(.05);
\draw [red] (2,-2) circle(.05);
\draw [red] (2,2) circle(.05);
\end{tikzpicture}
\caption{Normal and tangential edges in the two-dimensional case}\label{f01}

\begin{flushleft}
\emph{Left.} The sets $\etan{N}$, $\enor{d,N}$, and $\tilde{\partial}D_R$ in \cref{b66b}  for $d=2$, $N=10$:
$\etan{d,N}$ contains all oriented red edges going in both direction, $\enor{d,N}$ contains all \emph{exterior} blue normal edges, and $\tilde{\partial}D_R$ contains all red points without four corners $\circ$. \emph{Right.} The definition of $\funcNu{R}$ in \cref{b66b}:  for
every $y\in\tilde\partial D_R$ there exists a unique
 $x\in D_R$ such that $(x,y)\in\enor{R}$ and we set $\funcNu{R}(y)=(x,y).$

\end{flushleft}\end{figure}
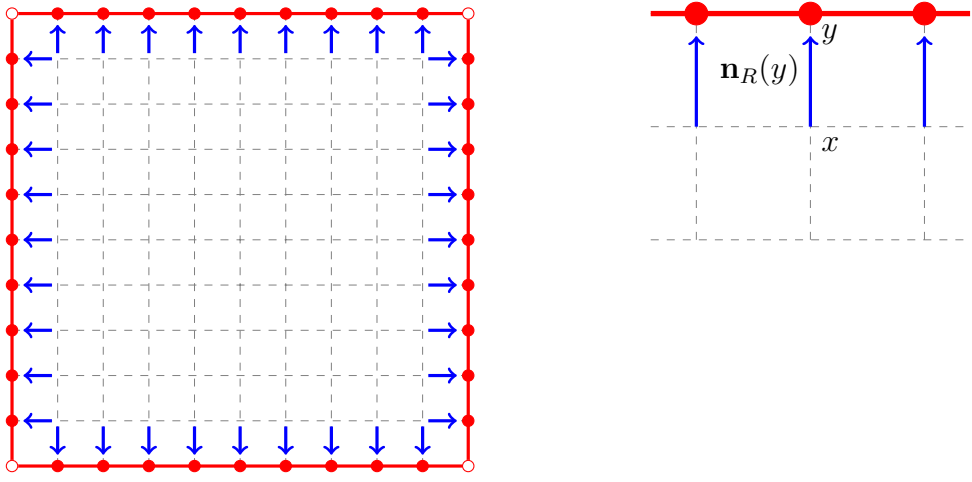

We will adapt the 5-step proof of 
Theorem~2 (the excess decay) in~\cite{BFO18} (see this proof in Section~5 in~\cite{BFO18}) to the discrete.
Let 
$\omegaH\in\Omega$ be the environment which satisfies for all 
$i\in [1,d]\cap\N $, $x\in\Z^d$
that $\omegaH(\{x,x+\unit{i}\})=(\matAhom)_{ii}$ and we consider
$u\in \mathbf{H}(\omega,R)$, $v\in\mathbf{H}(\omegaH,R) $, and the homogenization error $w=u-v-\phi_i\nabla_i v$, Einstein's notation used, as discrete counterparts of $u,v,w$ in the proof in the continuum.
Note that for our result we are only interested in boxes of large sizes $R\gg1$.

First, the calculations in Step 1 (p. 1395-6 in~\cite{BFO18}) are adapted to the discrete
by \cref{a73} where 
a new notation is chosen to simplify the product rule in the discrete case. 

To adapt 
Step 2 (p. 1396--1403 in~\cite{BFO18}) to the discrete requires more technicalities. 
First, 
for convenience we use  a similar notation as (30) and (33) in~\cite{BFO18}, i.e.,
let $\Lambda, \overline{\Lambda}\in [0,\infty)$ be given by
\begin{align}
\Lambda=\avnorm{\omega}{p}{E_R}+
\avnorm{\omega^{-1}}{q}{E_R}\quad\text{and}\quad
\overline{\Lambda}= \Lambda \avnorm{\omega(\nablaBar u)^2}{1}{E_R}.
\end{align}
Observe that
the sentence containing formula (39) in~\cite{BFO18} (about a suitable radius)  still works in the discrete case. 
So, we can assume that
\begin{align}
\avnorm{\nablaBar u}{\frac{2q}{q+1}}{\etan{R}}+
\avnorm{\omega\nablaBar u}{\frac{2q}{q+1}}{\enor{R}}\leq \overline{\Lambda}^{\nicefrac{1}{2}}.
\end{align}
Next, let $\epsilon\in (0,1/2)$, $\gamma\in (0,1/8)\cap(1/\N)$ be fixed. 
As in the continuum case (cf. (34) in~\cite{BFO18}) we will  construct $v$ such that
the homogenization error $w=u-v-\phi_i\nabla_i v$ satisfies that
\begin{align}\begin{split}
\avnorm{\omega (\nablaBar w)^2}{1}{E_R}\lesssim \epsilon^{\theta(p,q)}\overline{\Lambda}&+
\gamma
^{\min\left\{\frac{p-1}{2p},
\frac{q-1}{2q}\right\}}
\epsilon^{-(d-1)\min \left\{\frac{p+1}{p}, \frac{q+1}{q}\right\}} \Lambda\overline{\Lambda}\\
&+\left[\frac{1}{R^2}\left(
\avnorm{\phi}{\frac{2p}{p-1}}{D_R}+
\avnorm{\sigma}{\frac{2q}{q-1}}{D_R}\right)^2\right] \gamma^{-(d+2)}\Lambda\overline{\Lambda}\label{t04}
\end{split}\end{align}
where $\theta(p,q)\in [0,1]$ is given by
\begin{align}
\theta(p,q)= 
\left[1-(d-1)\left(\frac{1}{2p}+\frac{1}{2q}\right)\right]
\1_{d>2}+
\frac{(q-1)p}{q(p+1)}\1_{d=2}.\label{t04b}
\end{align}
We also start with the ansatz $w=u-v-\eta\phi_i \nabla_i v $ where $\eta \colon \Z^d\to\R$ is a cutoff function (like (43) in~\cite{BFO18}) which satisfies for all
$x\in D_{R-2\rho} $ that $\eta(x)=1$,
which satisfies for all
$x\in \Z^d\setminus D_{R-\rho} $ that $\eta(x)=0$, and 
which satisfies for all
$e\in \E^d_\pm $ that $|\nabla_e \eta|\leq \rho^{-1}$ where $\rho,R\in \N$ with $1\ll\rho\ll R$.
Our next aim is to adapt p. 1399 in~\cite{BFO18}, which is 
an a priori estimate on the energy of the homogenization error $w$,
 to the discrete case. To this end we do some detailed calculations in \cref{b66,x25}.  
Observe that \cref{x25}, Young's inequality, and H\"older's inequality yield that
\begin{align}\label{t01}
\begin{split}
&\avnorm{\omega (\nablaBar w)^2}{1}{E_R}
\lesssim  \left[
\avnorm{\omega(\nablaBar v)^2\1_{E_R\setminus E_{R-\rho}}}{1}{E_R}
+
\avnorm{\omega^{-1}(\nablaBar v)^2\1_{E_R\setminus E_{R-\rho}}}{1}{E_R}\right]\\
&+\frac{1}{|E_R|}\left|\sum_{(x,y)\in   \enor{R}}\left(u(x)-v(x)\right)
\left(\omega\nabla u-\omega_\mathrm{h} \nabla v\right)_{xy}\right|\\
&+
\Lambda \left[
\avnorm{|\phi|^2} {\frac{p}{p-1}}{D_R}+
\avnorm{|\sigma|^2}{\frac{q}{q-1}}{D_R}\right]
\left[
\avnorm{\nabla^*\nabla v}{\infty}{D_{R-\rho}}^2
+
\avnorm{\nabla\nabla v}{\infty}{D_{R-\rho}}^2+
\rho^{-2}\avnorm{\nabla v}{\infty}{D_{R-\rho}}^2\right].\end{split}
\end{align}

To construct $v$, as in the continuum case, we also distinguish the Dirichlet case ($q\geq p$) and the Neumann case ($p\geq q$). 
In the Dirichlet case we use a Dirichlet boundary condition  and in the Neumann case we use a Neumann boundary condition. While discrete boundary problems are widely studied in numerical analysis, the first challenge here is to find a way to \emph{smooth functions on the discrete surface of a discrete box}, i.e., to answer what
formulas (37) and (38) in~\cite{BFO18} look like in the discrete case. The solution to this issue is discussed in \cref{b17} whose main result  \cref{a48a} proves the existence of a family of good smoothing operators. Our construction relies on an elegant idea by Scott and Zhang~\cite{SZ90}. Since for our purpose we do not need the whole generality of~\cite{SZ90},  we choose an important feature of this paper to represent from scratch for convenience of the reader (see \cref{b39c}). 

Another difficulty is to adapt some regularity estimates (estimates (40) and (41) in~\cite{BFO18} comparing the tangential and normal component on a surface of the gradient of a harmonic function) to the discrete case. This contains many technicalities and is written in another paper (see~\cite{Ngu19}).
Here, we consider it as  assumptions (see \eqref{y01c} and \eqref{y01d} in \cref{b47b}).  

Although 
 \cref{a48a} is very promising, we still have  some minor issues: 
\begin{enumerate}[(a)]
\item \label{s01a}
 in contrast to the continuum case 
a discrete Neumann condition is only defined on $\tilde{\partial} D_R$, which is a proper subset of $\partial D_R$ (see \cref{f01}), and we therefore cannot directly use the sequence of smoothing operators in \cref{a48a} to smooth a Neumann condition and
\item \label{s01b} the proof in the continuum case  requires the symmetry of the convolution operator (see the sentence containing (49) in~\cite{BFO18}) which we have not adapted into the discrete yet. 
\end{enumerate}
A simple way to resolve issue \eqref{s01a} above is to identify functions on 
$\partial D_R$ as functions on $\tilde{\partial}D_R$, i.e.,
we replace a function $f\colon {\partial} D_R\to\R$ 
by $Zf\colon\partial D_R\to\R  $ which satisfies for all
$x\in\tilde\partial D_R$
that $(Zf)(x)=f(x)$ and which satisfies for all $x\in\partial D_R\setminus\tilde\partial D_R$ that $(Zf)(x)=0$. However, this ''modifying by zeros'' has the disadvantage that we cannot control the gradient on the boundary. A better way to resolve (a) is to copy the values of $f$ on $\tilde{\partial} D_R$
to $\partial D_R\setminus\tilde{\partial} D_R$: 
\begin{remark}[Modifying operator]\label{a04}
Let
$h\colon (\partial D_R\setminus\tilde{\partial} D_R)\to \tilde{\partial}D_R $ satisfy for all $x\in (\partial D_R\setminus\tilde{\partial} D_R)$
that $|h(x)-x|_\infty\leq 1$ and let $\modi{R}\colon \R^{{\partial}D_R}\to \R^{{\partial}D_R} $ be the \emph{modifying operator} which satisfies for all $x\in\tilde\partial D_R$, $f\colon \tilde{\partial} D_R\to\R$ that $(\modi{R}f)(x)= f(x)$ and
which satisfies for all $x\in\partial D_R\setminus\tilde\partial D_R$, $f\colon \tilde{\partial} D_R\to\R$ that $(\modi{R}f)(x)= f(h(x))$. 
For $d=2$ (see \cref{f01}) 
each corner, i.e., each element of $\partial D_R\setminus\tilde{\partial}D_R$,
has two neighbour and there are $4$ corners. Hence, for $d=2$
there are $2^4=16$ ways to choose $h$ and $\modi{M}$. However, it is easy to see that the modifying operator $\modi{R}$ satisfies for all $u\colon\partial D_R\to\R$, $r\in[1,\infty]$ that
$
\avnorm{\nablaBar (\modi{R}u)}{r}{\etan{R}}\leq c
\avnorm{\nablaBar u}{r}{\etan{R}}
$ and $
\avnorm{\modi{R}u}{r}{\partial D_{R}}\leq c \avnorm{u}{r}{\tilde \partial D_{R}}
$
where $c$ only depends on the dimension, e.g., choose $c=4^d$, and do not depend on the choice of $h$.
\end{remark}
In order to resolve issue \eqref{s01b} above we define the operator $\dualsmooth{R}{\epsilon}\colon \R^{\tilde{\partial}D_R}\to\R^{\tilde{\partial}D_R}$, which, up to the modifying operator $\modi{M}$, can be viewed as the dual operator of $\smooth{R}{\epsilon}$ in \cref{a48a} (here, the dependency on $R,\epsilon$ are dropped everywhere to lighten the notation). Here, 
$\dualsmooth{R}{\epsilon}$ is defined as the operator which satisfies for all 
$h\colon \tilde{\partial} D_R\to\R$, $g\colon \partial D_R\to\R$ that
\begin{align}
\sum_{x\in \tilde \partial D_{R}} (\dualsmooth{R}{\epsilon}h)(x) g(x)=
\sum_{x\in \tilde \partial D_{R}} h(x) \left(\smooth{R}{\epsilon}\modi{R}g\right)(x).
\end{align}
With $g\defeq (\1_{\{x\}}\upharpoonright_{\partial D_R})$ 
for
$x\in\tilde{\partial}D_R$ 
we obtain the canonical definition that 
$\dualsmooth{R}{\epsilon}\colon \R^{\tilde{\partial}D_R}\to\R^{\tilde{\partial}D_R}$ is the operator which satisfies for all $h\colon \tilde{\partial} D_R\to\R$, $x\in\tilde{\partial}D_R$ that
\begin{align}
 (\dualsmooth{R}{\epsilon}h)(x) =
\sum_{y\in \tilde \partial D_{R}} h(y) \left(\smooth{R}{\epsilon}\modi{R} \left(\1_{\{x\}}\upharpoonright_{\partial D_R}\right) \right)(y)
\end{align}
where for every
$x\in\tilde{\partial}D_R$ we denote by $\1_{\{x\}}\upharpoonright_{\partial D_R}\colon\partial D_R\to\R $  the function which satisfies for all 
$y\in\partial D_R\setminus\{x\}$ that $(\1_{\{x\}}\upharpoonright_{\partial D_R}) (y)=0$ and 
$(\1_{\{x\}}\upharpoonright_{\partial D_R}) (x)=1$.
Further technicalities are worked out in \cref{b47c}.
In \cref{b47b} we will list all properties of 
$\smooth{R}{\epsilon}$, $\modi{R}$, and $\dualsmooth{R}{ \epsilon}$ needed for the next step and we define in the Neumann case
 $v\defeq \NeumannExt{R}{\epsilon}u$ with $ \NeumannExt{R}{\epsilon}u$ in \eqref{y01} and in the Dirichlet case $ v\defeq\DirichletExt{R}{\epsilon}u$ with $\DirichletExt{R}{\epsilon}u$ in \eqref{y01b}.
Although we do not have the symmetry as in the continuum, $\dualsmooth{R}{\epsilon}$
has similar properties as $\smooth{R}{\epsilon}$ (see \cref{y27,d09}).
With other techniques from~\cite{BFO18} (\cref{x33,d06,y03,y10,y11,d11})
we  obtain 
\cref{y26}, which implies
an estimate on
the boundary term  in \eqref{t01}   and
 which adapts (53) and (54) in~\cite{BFO18} to the discrete case:
\begin{align}
\frac{1}{|E_R|}
\left|\sum_{x\in\tilde{\partial} D_R}\left(u(x)-v(x)\right)
\left(\omega\nabla u-\omega_\mathrm{h} \nabla v\right)_{\funcNu{R}(x)}\right|\lesssim
\epsilon^{\theta(p,q)}\overline{\Lambda}\label{t03a}
\end{align}
where $\theta(p,q)$ is given in \eqref{t04b} and $\funcNu{R}$ is the normal mapping in \cref{b66b}.
The rest of Step 2 is now much simpler. The annulus term  in \eqref{t01} is estimated by using \cref{d16}:
\begin{align}\begin{split}
&\avnorm{\omega(\nablaBar v)^2\1_{E_R\setminus E_{R-\rho}}}{1}{E_R}
+
\avnorm{\omega^{-1}(\nablaBar v)^2\1_{E_R\setminus E_{R-\rho}}}{1}{E_R}\\&\lesssim    \left(\frac{\rho}{R}\right)
^{\min\left\{\frac{p-1}{2p},
\frac{q-1}{2q}\right\}}
\epsilon^{-(d-1)\min \left\{\frac{p+1}{p}, \frac{q+1}{q}\right\}} \Lambda\overline{\Lambda}.\label{t03b}
\end{split}\end{align}
Next, note that the mean value inequality and
Cacciopoli's inequality for harmonic functions also hold in the discrete
(for the mean value inequality see, e.g., Lemma 3.4 in~\cite{Del98} or
Section 4 in~\cite{CG98} and for Cacciopoli's inequality see, e.g., Proposition 4.1 in~\cite{DD05}, Section 2 in~\cite{Del98}, or Section 4.3 in~\cite{CG98}). In addition, 
in our situation, interior estimates  are much easier than that on boundary terms since we can always adapt the sizes of balls and annuli. Here, the mean value inequality,
Cacciopoli's inequality,  the fact that
$\forall x\in D_{R-\rho}\colon D_{\rho/2}(x)\subseteq D_R $, and 
the fact that 
$\avnorm{\nabla v}{2}{D_{R}}^2\lesssim \overline{\Lambda}$ (see \cref{y15})
 prove for all $x\in D_{R-\rho}$ that
\begin{align}
|\nabla\nabla v(x)|^2 \lesssim 
\avnorm{\nabla\nabla v}{2}{D_{\rho/8}(x)}
\lesssim \rho^{-2}
\avnorm{\nabla v}{2}{D_{\rho/2}(x)}^2
\lesssim R^d\rho^{-(d+2)} \avnorm{\nabla v}{2}{D_{R}}^2
\lesssim \frac{1}{R^2}\left(\frac{R}{\rho}\right)^{d+2}\overline{\Lambda}
,
\end{align}
\begin{align}
|\nabla^*\nabla v(x)|^2 \lesssim 
\avnorm{\nabla^*\nabla v}{2}{D_{\rho/8}(x)}^2
\lesssim \rho^{-2}
\avnorm{\nabla v}{2}{D_{\rho/2}(x)}^2
\lesssim R^d\rho^{-(d+2)} \avnorm{\nabla v}{2}{D_{R}}^2
\lesssim \frac{1}{R^2}\left(\frac{R}{\rho}\right)^{d+2}\overline{\Lambda}
,
\end{align}
and
\begin{align}
{\rho^{-2}}|\nabla v(x)|^2\lesssim 
{\rho^{-2}} 
\avnorm{\nabla v}{2}{D_{\rho/2}(x)}^2
\lesssim  R^d\rho^{-(d+2)} \avnorm{\nabla v}{2}{D_{R}}^2
\lesssim \frac{1}{R^2}\left(\frac{R}{\rho}\right)^{d+2}\overline{\Lambda}
.
\end{align}
This 
  implies an estimate for the last term in \eqref{t01}:
\begin{align}\begin{split}
&\Lambda \left[
\avnorm{|\phi|^2} {\frac{p}{p-1}}{D_R}+
\avnorm{|\sigma|^2}{\frac{q}{q-1}}{D_R}\right]
\left[
\avnorm{\nabla^*\nabla v}{\infty}{D_{R-\rho}}^2
+
\avnorm{\nabla\nabla v}{\infty}{D_{R-\rho}}^2+
\rho^{-2}\avnorm{\nabla v}{\infty}{D_{R-\rho}}^2\right]\\
&\lesssim 
\Lambda \left[\frac{1}{R^2}\left(
\avnorm{\phi} {\frac{2p}{p-1}}{D_R}+
\avnorm{\sigma}{\frac{2q}{q-1}}{D_R}\right)^2\right] \left(\frac{R}{\rho}\right)^{d+2}\overline{\Lambda}.\label{t03c}
\end{split}\end{align}
This, \eqref{t03a}, \eqref{t03b}, and \eqref{t01} imply that
\begin{align}\begin{split}
\avnorm{\omega (\nablaBar w)^2}{1}{E_R}\lesssim \epsilon^{\theta(p,q)}\overline{\Lambda}&+
\left(\frac{\rho}{R}\right)
^{\min\left\{\frac{p-1}{2p},
\frac{q-1}{2q}\right\}}
\epsilon^{-(d-1)\min \left\{\frac{p+1}{p}, \frac{q+1}{q}\right\}} \Lambda\overline{\Lambda}\\&+\Lambda\left[\frac{1}{R^2}\left(
\avnorm{\phi} {\frac{2p}{p-1}}{D_R}+
\avnorm{\sigma}{\frac{2q}{q-1}}{D_R}\right)^2\right] \left(\frac{R}{\rho}\right)^{d+2}\overline{\Lambda}.
\end{split}\end{align}
Choosing $\rho\geq 20$ and $R\geq 64\rho$ such that $ \gamma \leq \rho/R \leq 2\gamma$ and noting
that for our main result we are only interested in $R\gg 1$
we obtain \eqref{t04} 
and therefore finish this step (cf. (56) in~\cite{BFO18}). Steps 3,4, and 5 can be easily adapted to the discrete. Indeed, we only have to adapt interior estimates which are consequences of the discrete mean value inequality (see, e.g., Lemma 3.4 in~\cite{Del98} or
Section 4 in~\cite{CG98}) and the discrete Cacciopolli inequality
(see, e.g., Proposition 4.1 in~\cite{DD05}, Section 2 in~\cite{Del98}, or Section 4.3 in~\cite{CG98}) and we can always adapt the size of boxes.
\medskip

\renewcommand{\smooth}[2]{\mathbf{S}_{#1,#2}}

\section{Construction of the correctors}\label{a07c}
\subsection{Preliminary}
Throughout this section we always use the notation in \cref{m03} below.
\begin{setting}\label{m03}
Let $\mathbb{P}$ be a stationary and ergodic probability measure on $(\Omega,\mathcal{F})$ and let $\matA\colon \Omega\times \Z^d\to \R^{d\times d}$ be the function whose values are diagonal $\R^{d\times d}$ matrices and which satisfies for all $x\in\Z^d$, $\omega\in\Omega$, $i,j\in [1,d]\cap\N$ that
$\matA_{ij}(\omega,x)=\omega(\{x,x+\unit{i}\})\1_{\{i\}}(j)$. 
For every measurable function $\zeta\colon\Omega\to\R $ let 
$\mathrm{D}_i\zeta,\mathrm{D}_i^*\zeta\colon \Omega\to\R $, $i\in [1,d]\cap\N$,
be the functions which satisfy for all 
$\omega\in \Omega$,
$i\in [1,d]\cap\N$ that
\begin{align}
(\mathrm{D}_i\zeta)(\omega)= \zeta(\tau_{\unit{i}}\omega)-\zeta(\omega)\quad\text{and}\quad(\mathrm{D}_i^*\zeta)(\omega)= \zeta(\tau_{-\unit{i}}\omega)-\zeta(\omega)
\end{align}
and let 
$\mathrm{D}\zeta,\mathrm{D}^*\zeta\colon \Omega\to\R^d $ be the functions which satisfy for all $\omega\in \Omega$
that
\begin{align}
(\mathrm{D}\zeta)(\omega)=\sum_{i=1}^{d}(\mathrm{D}_i\zeta)(\omega)\unit{i}\quad\text{and}\quad(\mathrm{D}^*\zeta)(\omega)=\sum_{i=1}^{d}(\mathrm{D}_i^*\zeta)(\omega)\unit{i}.
\end{align} %\\
 Since 
the $x$-derivatives are denoted by $\nabla, \nabla^*$ and the $\omega$-derivative are denoted by $\mathrm{D}, \mathrm{D}^*$,
no confusion can arise and we extend the notation   above to functions $u$ that depend on $\omega\in\Omega$ and $x\in \Z^d$. E.g., we write $(\nabla u)(\omega,x)= (\nabla u(\omega,\cdot))(x)$ and 
$(\mathrm{D} u)(\omega,x)= (\mathrm{D}u (\cdot,x))(\omega)$.
\end{setting}
\cref{a07b} below is a simple observation. However, we will use it several times.
\begin{remark}\label{a07b}
For all
$v\in\Stat(\R)$ it holds that  $\nabla v,\mathrm{D}v,\nabla^* v,\mathrm{D}^* v\in \Stat(\R^d)$, 
$\P$-a.s. 
$\nabla v=\mathrm{D}v$, and $\P$-a.s.  $\nabla^* v=\mathrm{D}^* v$.
Roughtly speaking for stationary functions we can replace a $x$-derivative by the corresponding $\omega$-derivative and vice verse.
\end{remark}

%Next, as usually done for the first order corrector $\phi_i$, we fix $\sigma_{ijk}(\omega,0)=0$. For all $e\in\Z^d$ with $|e|=1$, $ \sigma_{ijk}(\cdot,e)$ is thus well-defined. The requirement that its gradient $ \nabla^*\sigma_{ijk}$ is stationary, i.e.  $\sigma_{ijk}$ is co-cycle, then extends $ \sigma_{ijk}$ to $\Omega\times\Z^d$ 

\cref{a31c} below  is a well-known result on  
\emph{the first order corrector}
$\phi$ (see, e.g.,~\cite{Kue83,BB07,Bis11,GNO14,BFO18}, especially Section 3.2 in~\cite{Bis11} for a detailed discussion).
\cref{z03} is  a typical result on weak convergence and the idea of its proof might be found in, e.g., \cite{Kue83}.
Since the notation is quite different in the literature we include this result here for convenience of the reader.
\begin{lemma}[Construction of $\phi$]\label{a31c}	
Let $q\in [1,\infty]$ 
and assume  that 
\begin{align}
\left\|\matA(\cdot,0)\right\|_{L^1(\Omega,\R^{d\times d})}+
\left\|\matA(\cdot,0)^{-1}\right\|_{L^q(\Omega,\R^{d\times d})}<\infty .
\end{align}
 Then there exist measurable functions 
$\phi_i\colon \Omega\times\Z^d\to \R$, $i\in [1,d]\cap\N$, and a diagonal matrix 
$\matAhom\in \R^ {d\times d}$
such that $\P$-a.s. for all $i\in [1,d]\cap\N$ it holds that
$
\nabla\phi_i \in \Stat(\R^d)$, $
-\nabla^*\cdot\left( \matA\left(\unit{i}+\nabla\phi_i\right)\right)=0$,
$
\E\!\left[ \left(\nabla\phi_i\cdot \matA\nabla\phi_i\right)(\cdot,0)\right]<\infty$, and
$\E\!\left[ |(\nabla\phi_i)(\cdot,0)|^{2q/(q+1)}\right] <\infty$.

\end{lemma}

\begin{proof}[Sketch of the proof of \cref{a31c}]
Typically (see, e.g., Section 3.2 in~\cite{Bis11}) 
in order to construct
the corrector $\phi$ 
we introduce  the Hilbert subspace $H^2_\nabla$ defined as the closure of $\{\mathrm{D}\zeta\mid \zeta\in L^\infty(\Omega,\R)\}$ with respect to 
the scalar product $(u,v)_{L^2_\mathrm{cov}}:= \E[
\matA(\cdot,0)
u \cdot v]$. Next, we
 define $ (\nabla \phi_i+\unit{i})(\cdot,0)\colon \Omega\to \R^d$, $i\in [1,d]\cap\N$ as orthogonal projections of $\unit{i}$, $i\in [1,d]\cap\N$, onto 
$H^2_\nabla$ and define the homogenized matrix $\matAhom$ with the property  that 
$\forall\, i\in [1,d]\cap\N\colon 
  \matAhom \unit{i}=\E[ \left(\matA\left(\nabla\phi_i+\unit{i}\right)\right)(\cdot,0) ].
$
The requirement that $\forall\, i\in[1,d]\cap\N\colon \nabla\phi_i\in \Stat(\R^d)$ then allows
to define the so-called harmonic coordinates
$(\nabla \phi_i+\unit{i})\colon \Omega\times\Z^d\to \R^d $, $i\in[1,d]\cap\N$,
as extensions of those orthogonal projections, i.e.,
as the functions which satisfy for every $x\in\Z^d$, $i\in[1,d]\cap\N$, and $\P$-a.e. $\omega\in\Omega$ that
$(\nabla \phi_i+\unit{i})(\omega,x)= (\nabla \phi_i+\unit{i})(\tau_x\omega,0)$.
\cref{a07b}, partial integration, the projection  property, and the stationary property ensure for all 
$x\in\Z^d$, $i\in[1,d]\cap\N$, $\zeta\in L^\infty(\Omega,\R)$
that $
\E[(\nabla^*\cdot \matA(\nabla \phi_i+\unit{i}))(\cdot,x)\zeta]=
\E[(\mathrm{D}^*\cdot \matA(\nabla \phi_i+\unit{i}))(\cdot,x)\zeta]=
\E[(\matA(\nabla \phi_i+\unit{i}))(\cdot,x)\cdot\mathrm{D}\zeta]=0$ and hence $\P$-a.s. $\nabla^*\cdot \matA(\nabla \phi_i+\unit{i})=0$. After having constructed $\nabla \phi_i$, $i\in [1,d]\cap\N $, we define its primitive $\phi_i$, $i\in [1,d]\cap\N $, 
(with $\phi(\cdot,0)=0$ fixed)
by discrete contour integrals, which does not depend on the choice of paths (due to the stationary property). The integrability of $\nabla\phi$, i.e., the fact that 
$\forall\,i\in[1,d]\cap\N\colon \E\!\left[ |(\nabla\phi_i)(\cdot,0)|^{2q/(q+1)}\right] <\infty$ easily follows from H\"older's inequality.
\end{proof}

\begin{lemma}[Weak convergence]\label{z03}Let
$L^2_{\nabla^*}$ be the closure of 
$\{\mathrm{D}^*\zeta\mid\zeta\in L^\infty(\Omega)\}$ in  $L^2(\Omega,\R ^d)$,
 let $f\in L^2(\Omega,\R^d)$, and let $ \xi$ be the orthogonal projection of $f$ onto $L^2_{\nabla^*}$.
Then there exists a unique family $(u_\epsilon)_{\epsilon\in (0,1)}\subseteq L^2(\Omega,\R)$ 
and there exists a sequence $(\epsilon_n)_{n\in\N}\subseteq(0,1)$
such that
 for all
$\epsilon\in (0,1)$ it holds that $
\left(\epsilon + \mathrm{D}\cdot \mathrm{D}^* \right)u_\epsilon=\mathrm{D} \cdot f$
and such that for all $G\in L^2(\Omega,\R^d) $ it holds that
 $\lim_{n\to\infty}
\E[ \mathrm{D}^* u_{\epsilon_n}\cdot G]=\E[ \xi\cdot G]$.
\end{lemma}
\begin{proof}[Proof of \cref{z03}]For every 
$\epsilon\in (0,1)$ let  $(\cdot,\cdot )_\epsilon\colon L^2(\Omega,\R)\to\R$ 
be the function which satisfies
for all $u,v\in L^2(\Omega,\R)$  that
$(u,v)_{\epsilon}= \E[\epsilon uv +
\mathrm{D}^* u\cdot \mathrm{D}^*v].$ Then for all $\epsilon\in (0,1)$ it holds that
$L^2(\Omega,\R) $ equipped with $(\cdot,\cdot)_\epsilon$ is a Hilbert space.
Furthermore, the Cauchy-Schwarz inequality shows that
for all $\epsilon\in (0,1)$, $v\in L^2(\Omega,\R)$ it holds that
$
 \left|\E[  f\cdot \mathrm{D}^*v]\right|\leq \left(\E[ |f|^2] \right)^{\!\nicefrac{1}{2}}
\left(\E[ |\mathrm{D}^*v|^2] \right)^{\!\nicefrac{1}{2}}
\leq \left(\E[ |f|^2] \right)^{\!\nicefrac{1}{2}}\left(v,v\right)_\epsilon^{\!\nicefrac{1}{2}}.
$
This implies for all $\epsilon\in (0,1)$ that  $L^2(\Omega,\R)\ni v\mapsto \E [f\cdot \mathrm{D}^*v]\in\R  $ is  continuous with respect to $(\cdot,\cdot )_\epsilon$. 
For every  $\epsilon\in (0,1)$ let $u_\epsilon \in L^2(\Omega,\R)$ be the function
which exists by Riesz' lemma with the property that $(u_\epsilon,v)_\epsilon=\E[f\cdot \mathrm{D}^*v]$.
This, the partial integration, and the definition of $(\cdot,\cdot )_\epsilon$, $\epsilon\in (0,1)$, 
show that for all $ v\in L^2(\Omega)$, $\epsilon\in (0,1) $ it holds that 
$
\E[\epsilon  u_\epsilon v]+ \E[\mathrm{D}\cdot \mathrm{D}^*u_\epsilon)  v]=
\E[(\epsilon  u_\epsilon v]+ \E[ (\mathrm{D}^*u_\epsilon)\cdot (\mathrm{D}^*  v)]= 
(u_\epsilon,v)_\epsilon=\E[f\cdot \mathrm{D}^*v]
= \E[ (\mathrm{D}\cdot f) v]
$ and hence $\left(\epsilon + \mathrm{D}\cdot \mathrm{D}^* \right)u_\epsilon=\mathrm{D} \cdot f$. To show uniqueness let $\tilde
{u}_\epsilon \in L^2(\Omega,\R)$, $\epsilon\in (0,1)$, satisfy
for all 
  $\epsilon\in (0,1)$ that $\left(\epsilon + \mathrm{D}\cdot \mathrm{D}^* \right)\tilde{u}_\epsilon=\mathrm{D} \cdot f$. 
The partial integration and the fact that $\forall\,\epsilon\in (0,1)\colon \left(\epsilon + \mathrm{D}\cdot \mathrm{D}^* \right)(u_\epsilon-\tilde{u}_\epsilon)=0$    show 
for all $\epsilon\in (0,1)$ that
$
\E [\epsilon(u_\epsilon-\tilde{u}_\epsilon)^2 +|\mathrm{D}^*u_\epsilon-\mathrm{D}^*\tilde{u}_\epsilon|^2]=
\E[((\epsilon + \mathrm{D}\cdot \mathrm{D}^* )(u_\epsilon-\tilde{u}_\epsilon))(u_\epsilon-\tilde{u}_\epsilon)]=0$ and hence
$\P$-a.s. $u_\epsilon=\tilde{u}_\epsilon$.
Next, the fact that $\forall\,\epsilon\in (0,1)\colon (u_\epsilon,v)_\epsilon=\E[f\cdot \mathrm{D}^*v]$, the Cauchy-Schwarz inequality, and the definition of $(\cdot,\cdot )_\epsilon$, $\epsilon\in (0,1)$, imply for all $\epsilon\in (0,1)$ that
$(u_\epsilon,u_\epsilon)_\epsilon=
\left|\E[ f\cdot \mathrm{D}^*u_\epsilon]\right|
\leq (\E[|f|^2] )^{\!\nicefrac{1}{2}}(\E [|\mathrm{D}^*u_\epsilon|^2])^{\!\nicefrac{1}{2}}
\leq (\E[|f|^2] )^{\!\nicefrac{1}{2}}\left(u_\epsilon,u_\epsilon\right)_\epsilon^{\!\nicefrac{1}{2}}
$, therefore $ \E[\epsilon|u_\epsilon|^2+ |\mathrm{D}^*u_\epsilon|^2]=  (u_\epsilon,u_\epsilon)_\epsilon \leq \E[|f|^2]$. 
This implies that there exist $(\epsilon_n)_{n\in\N}\subseteq(0,1)$
and $\eta \in L^2(\Omega,\R^d)$ which satisfy for all $G\in L^2(\Omega,\R^d) $, $v\in L^2(\Omega,\R)$ that
 $\lim_{n\to\infty}
\E[ \mathrm{D}^* u_{\epsilon_n}\cdot G]=\E[ \eta\cdot G]$ and 
$\lim_{n\to\infty}
\epsilon_n \E [u_{\epsilon_n} v]=0$.
Hence,
$\E[ \eta\cdot \mathrm{D}^*v]=
\lim_{n\to\infty}(
\epsilon_n \E [u_{\epsilon_n} v]+ \E[\mathrm{D}^*u_{\epsilon_n}\cdot \mathrm{D}^*v] )= \E [f\cdot \mathrm{D}^*v]
$. Thus, $\eta$ is the projection of $f$ onto $L^2_{\nabla^*}$ and  $\xi=\eta$.
This completes the proof of \cref{z03}.
\end{proof}
\subsection{A Meyer-type estimate}
\cref{b08} below adapts the $L^r$-estimate, $r\in (1,\infty)$, in (21) into the discrete. Although it is not really a difficult issue, we still need to find some suitable results in the literature.
\begin{lemma}[A Meyer-type estimate on the probability space]
\label{b08}Let $r\in (1,\infty)$, let 
$L^2_{\nabla^*}$ be the closure of 
$\{\mathrm{D}^*\zeta\mid \zeta\in L^\infty(\Omega)\}$ in  $L^2(\Omega,\R ^d)$, and for every $f\in L^\infty(\Omega,\R ^d)$ let $\mathcal{H}f  $ be  the projection of $f$ onto $L^2_{\nabla^*}$.
Then 
there exists $c\in (0,\infty)$ such that for all
$f\in L^\infty(\Omega,\R ^d)$ it holds that
$
\|\mathcal{H}f\|_{L^r(\Omega,\R^d)}\leq c \|f\|_{L^r(\Omega,\R^d)}.
$
\end{lemma}
\begin{proof}[Proof of \cref{b08}]
Let 
  $ G_T\colon \Z^d\times\Z^d\to\R$, $T\in(0,\infty)$, be the so-called Green functions
which satisfy  for all $T\in(0,\infty)$ and for all bounded functions $f\colon \Z^d\to\R$ that
the unique solution $u$ to $\forall\,x\in\Z^d\colon  ((T^{-1}-\Laplace) u)(x)=f(x) $ satisfies for all $x\in \Z^d$ that $ u(x)=\sum_{y\in\Z^d}G_T(x,y)f(y) $ (for a detailed result on existence and uniqueness see, e.g., Theorem~4.1 in~\cite{Mou13}).
For every 
$T\in (0,\infty)$, 
$f\in\Stat(\R^d)$ with $f(\cdot,0)\in L^\infty(\Omega,\R^d)$ let 
$u_T^f\colon \Omega\times \Z^d\to \R $ be the function which satisfies 
for all $\omega\in\Omega$, $x\in\Z^d$ 
 that  $u^f_T(\omega,x)=\sum_{y\in\Z^d}G_T(x,y)(\nabla\cdot f)(\omega,y) $. 
Then for all
$T\in (0,\infty)$, 
$f\in\Stat(\R^d)$ with $f(\cdot,0)\in L^\infty(\Omega,\R^d)$ it holds that
$u^f_T\in \Stat(\R)$ and
 $\P$-a.s.
$(T^{-1}-\Laplace) u^f_T= \nabla \cdot f$ and therefore (by \cref{a07b}) it holds  $\P$-a.s.
$(T^{-1}-\mathrm{D}\cdot \mathrm{D}^*) u^f_T= \mathrm{D} \cdot f$. Next, observe that
the first and second derivatives of the Green functions (see, e.g., Proposition 3.7 in~\cite{Mou13}) satisfy that 
there exist $c_1,c_2\in (0,\infty)$ such that  for all $x,y\in\Z^d$  it holds that
$
|\nabla_x G_T(x,y)|\leq c_1(1\SUP |x-y|)^{-(d-1)} \exp(-\frac{1}{\sqrt{T}}c_2|x-y|)
$
and
$
|\nabla_x\nabla_ y G_T(x,y)|\leq c_1(1\SUP |x-y|)^{-d} \exp(-\frac{1}{\sqrt{T}}c_2|x-y|)
.$ 
In addition, partial integration shows  for all $f\colon \Z^d \to\R^d$ with finite support and for all $x\in\Z^d$  that
$u^f_T(x)=\sum_{y\in\Z^d}G_T(x,y)(\nabla\cdot f)(\omega,y) $
and $(\nabla^*u^f_T)(x)=\sum_{y\in\Z^d}\nabla^*\nabla^*G_T(x,y) f(y) $.
Using the standard argument with good and bad boxes to prove a Meyer-type estimate in Section 4 in~\cite{BSW14} (especially the estimate on the third derivatives of the Green functions) we  easily adapt (22) in~\cite{BFO18} into the discrete case. 
Continuing with a purely deterministic argument similar to that in the continuum case  (the paragraph below (22) in~\cite{BFO18})  
then  yields
that there exists $c\in (0,\infty)$ such that for every $R\in\N$, $T\in (0,\infty)$, 
$f\in\Stat(\R^d)$ with $f(\cdot,0)\in L^\infty(\Omega,\R^d)$, and $\P$-a.e. $\omega\in \Omega$  
it holds that
\begin{align}
\avnorm{(\nabla^* u_T^f)(\omega,\cdot)}{r}{D_{R/2}}^r\leq ce^{-crR/\sqrt{T}} \left[\sup_{x\in\Z^d} |f(\omega,x)|\right]+c\avnorm{f(\omega,\cdot)}{r}{D_{2R}}.\end{align}
Letting $R$ tend to infinity, the fact that $\forall \,T\in (0,T),f\in \Stat(\R^d)\colon u_T^f\in \Stat(\R)$, and the ergodic theorem show that there exists $c\in (0,\infty)$
such that
for all $T\in (0,\infty)$, 
$f\in\Stat(\R^d)$  with $f(\cdot,0)\in L^\infty(\Omega,\R^d)$ it holds that 
\begin{align}\label{k124}
\|(\mathrm{D}^* u^f_T) (\cdot,0)\|_{L^r(\Omega,\R^d)}=
\|(\nabla^* u^f_T) (\cdot,0)\|_{L^r(\Omega,\R^d)}\leq c\|f(\cdot,0)\|_{L^r(\Omega,\R^d)}.
\end{align}
 Now, for every $f\in L^\infty(\Omega,\R^d)$ let $\hat{f}\in \Stat(\R^d)$ be the function which satisfies for all $\omega\in\Omega$, $x\in \R^d$ that $\hat{f}(\omega,x)=f(\tau_x\omega)$.
 \cref{z03} and the fact that $\forall\,T\in (0,T)\colon (T^{-1}-\mathrm{D}\cdot \mathrm{D}^*) u^f_T= \mathrm{D} \cdot f$ then imply
 for all $f\in L^\infty(\Omega,\R^d)$  that there exists a sequence $(T_n)_{n\in \N}\subseteq(0,\infty)$ such that for all 
$G\in L^2(\Omega,\R^d) $ it holds that
 $\lim_{n\to\infty}
\E[ \mathrm{D}^* u_{T_n}^{\hat{f}}\cdot G]=\E[ (\mathcal{H}f)\cdot G]$. This, \eqref{k124}, and a simple duality argument complete the proof of \cref{b08}.
\end{proof}
\subsection{Sublinearity via an approximating argument}
In \cref{b01} below we use an approximating argument rather than Sobolev's embedding and weak convergence as done in p.~1394 in~\cite{BFO18}. Therefore, our result includes the case when $1/p+1/q\leq 2/d$.
\begin{lemma}[A density argument]\label{b01}
Let $p,q\in (1,\infty]$ satisfy that $1/p+1/q\leq 2/d$,  let
$L^{2p/(p+1)}_{\nabla^*}$ be the closure of 
$\{\mathrm{D}^*\zeta\mid \zeta\in L^\infty(\Omega)\}$ in  $L^{2p/(p+1)}(\Omega,\R ^d)$,
and let
$u\colon \Omega\times\Z^d\to\R$ be measurable and satisfy that
$\nabla^* u \in\Stat(\R^d)$ and $(\nabla^* u)(\cdot,0)\in L^{2p/(p+1)}_{\nabla^*} $. Then it holds that
\begin{align}
\limsup_{R\to\infty}\left[ \frac1R\inf_{a\in \R }
\avnorm{u-a}{2q/(q-1)}{D_R}\right]=0
\end{align}
where $2p/(p+1):= 2$ if $p=\infty$ and $2q/(q-1):=2$ if $q=\infty$.
\end{lemma}
\begin{remark}
\cref{b01} still holds if we drop $*$ everywhere in its statement.
\label{b01b}
\end{remark}
\begin{proof}
[Proof of \cref{b01}]
Throughout this proof 
for every  measurable function 
$v\colon \Omega\times\Z^d\to\R$ with
$\nabla^*v \in\Stat(\R^d)$ and $(\nabla^* v)(\cdot,0)\in L^{2p/(p+1)}_{\nabla^*} $
let $F(v)\in[0,\infty] $ be the real extended number which satisfies that
$
F(v)= \E\!\left[ \limsup_{R\to\infty} \tfrac1R\inf_{a\in \R }
\avnorm{v-a}{2q/(q-1)}{D_R}\right].
$
Sobolev's inequality (combined with the assumption that $1/p+1/q\leq 2/d$) and the ergodic theorem then show that there exists $c\in (0,\infty)$ which satisfies  for all  measurable functions 
$v\colon \Omega\times\Z^d\to\R$ with
$\nabla^*v \in\Stat(\R^d)$ and $(\nabla^* v)(\cdot,0)\in L^{2p/(p+1)}_{\nabla^*} $  that
\begin{align}\label{b03}
F(v)\leq c
\E\!\left[\limsup_{R\to\infty} 
\avnorm{\nabla^* v}{2p/(p+1)}{D_{2R}} \right]=c\|(\nabla^* v)(\cdot,0)\|_{L^{2p/(p+1)}(\Omega,\R^d)}.
\end{align} 
Next, the assumption that $(\nabla^* u)(\cdot,0)\in L^{2p/(p+1)}_{\nabla^*} $
implies that there exists a sequence $(\zeta_n)_{n\in\N}\subseteq L^\infty(\Omega,\R) $ such that 
$
\lim_{n\to\infty}\|(\nabla^* u)(\cdot,0)-\mathrm{D}^*\zeta_n\|_{L^{2p/(p+1)}(\Omega,\R^d)}=0.$
For every  $n\in\N$ let $u_n\colon \Omega\times\Z^d\to\R$ be the measurable function which satisfies for $\P$-a.e. $\omega\in\Omega$ and $x\in\Z^d$ that $u_n(\omega,x)=\zeta_n(\tau_x\omega)-\zeta(\omega)$. Then it holds for all $n\in\N$ that
$u_n\in\Stat(\R)$ and hence $\P$-a.s. $(\nabla^* u_n) (\cdot,0)=(\mathrm{D}^*u_n)(\cdot,0)=\mathrm{D}^*\zeta_n$. This and the fact that $
\lim_{n\to\infty}\|(\nabla^* u)(\cdot,0)-\mathrm{D}^*\zeta_n\|_{L^{2p/(p+1)}(\Omega,\R^d)}=0$ imply that $\lim_{n\to\infty}\|(\nabla^* u)(\cdot,0)-(\nabla^*u_n)(\cdot,0)\|_{L^{2p/(p+1)}(\Omega,\R^d)}=0$. Hence, 
\eqref{b03} (with $v\defeq u-u_n$ for $n\in\N$) shows that $\lim_{n\to\infty}F(u-u_n)=0$.
Furthermore, the definition of $u_n$, $n\in\N$, the triangle inequality, and the fact that $\forall\,n\in\N\colon \zeta_n\in L^\infty(\Omega,\R)$ imply for all $n\in \N$   that $\P$-a.s.
$\limsup_{R\to\infty}\tfrac{1}{R}\avnorm{u_n}{2q/(q-1)}{D_R}\leq\limsup_{R\to\infty}\tfrac{2}{R} \|\zeta_n\|_{L^\infty(\Omega,\R)}=0$. This shows for all $n\in\N$ that $F(u_n)=0$. The triangle inequality and the fact that $\lim_{n\to\infty}F(u-u_n)=0$ then prove that $F(u)=0$.
This completes the proof of \cref{b01}.
\end{proof}
\cref{b05} below can be proved similarly to the argument from (25) to (26) in~\cite{BFO18}. We include the proof here for convenience.
\begin{lemma}[A dyadic argument]\label{b05}Let $u\colon \Z^d\to\R$ be a function, let $r\in [1,\infty]$ be a real extended number, and assume that
\begin{align}
\limsup_{R\to\infty} \left[\frac1R\inf_{a\in \R }
\avnorm{u-a}{r}{D_R}\right]=0.\end{align} 
Then it holds that
\begin{align}
\limsup_{R\to\infty} \left[\frac1R
\avnorm{u}{r}{D_R}\right]=0.
\end{align}
\end{lemma}
\begin{proof}[Proof of \cref{b05}]Let $\delta\in (0,\infty)$ be arbitrary and fixed, let $R_0\in\N$ satisfy that for all $R\in [R_0,\infty)\cap\N$ it holds that $\inf_{a\in \R }
\avnorm{u-a}{r}{D_R}<\delta R$, let $a\colon [R_0,\infty)\cap\N\to \R $ be a function which satisfies for all $R\in  [R_0,\infty)\cap\N$ that $\avnorm{u-a(R)}{r}{D_R}<\delta R$,
and let $m\in [1,\infty]$ be the real extended number given by
$m=\sup_{R\in\N}|D_{2R}|/|D_{R}|$. The triangle inequality then implies for all $R,R'\in  [R_0,\infty)\cap\N$ with $R\leq R'\leq 2R$ that 
\begin{align}\begin{split}
&|a(R')-a(R)|\leq \avnorm{u-a(R)}{r}{D_R}+\avnorm{u-a(R')}{r}{D_R}\\&\leq 
\avnorm{u-a(R)}{r}{D_R}+\left[\frac{|D_{R'}|}{|D_{R}|}\right]^{\!1/r}\avnorm{u-a(R')}{r}{D_{R'}}\leq \delta R + m^{1/r}\delta 2R\leq 3m^{1/r}\delta R.
\end{split}\end{align}
This and the triangle inequality show for all $n\in \N_0$, $R\in [2^nR_0,2^{n+1}R_0)$ that
\begin{align}\begin{split}
|a(R)-a(R_0)|&\leq |a(R)-a(2^{n}R_0)|+\sum_{j=0}^{n-1}| a(2^{j+1}R_0)-a(2^jR_0)|\\
&\leq 3m^{1/r}\delta 2^nR_0+
\left[\sum_{j=0}^{n-1}
3m^{1/r}\delta 2^jR_0\right]\leq 3m^{1/r}\delta2^{n+1}R_0\leq 6m^{1/r}\delta R.
\end{split}\end{align}
We therefore obtain for all $R\in [R_0,\infty)\cap\N$ that $|a(R)-a(R_0)|\leq  6m^{1/r}\delta R$. The triangle inequality and the fact that 
$\forall\,R\in  [R_0,\infty)\cap\N\colon \avnorm{u-a(R)}{r}{D_R}<\delta R$ establish that
\begin{align}
\begin{split}
\left[
\limsup_{R\to\infty}
\tfrac{1}{R}
\avnorm{u}{r}{D_R}\right]&\leq 
\left[
\limsup_{R\to\infty}
\tfrac{1}{R}
\avnorm{u-a(R)}{r}{D_R}\right]+ \left[\limsup_{R\to\infty}
\tfrac{1}{R}
|a(R)-a(R_0)|\right]+\left[\limsup_{R\to\infty}
\tfrac{1}{R}
a(R_0)\right]
\\
&\leq \delta+ 6m^{1/r}\delta .
\end{split}
\end{align}
This, the fact that $\delta\in (0,\infty)$ was arbitrarily chosen, and the fact that $m<\infty$ complete the proof of \cref{b05}.
\end{proof}
\subsection{Conclusion}
\begin{lemma}\label{d20}
Let $r\in [1,\infty)$ and let $u\colon \Omega\times \Z^d\to\R $ be a measurable function which satisfies that $u\in \Stat(\R)$, $u(\cdot,0)\in L^r(\Omega,\R ) $, and $\P$-a.s. $\Laplace u =0 $. Then $\P$-a.s. $u=\E[u(\cdot,0)]$.
\end{lemma}
\begin{proof}[Proof of \cref{d20}]
The maximal inequality for harmonic functions 
(see, e.g., Corollary 3.9 in~\cite{ADS16}) and the fact that
$\P$-a.s.  $\Laplace u=0$
 show that there exist $\gamma \in (0,\infty)$ such that for every $R\in\N$,
$\P$-a.s. $\omega\in\Omega$
 it holds that $|u(\omega,0)|\leq\avnorm{u(\omega,\cdot)}{\infty}{D_{R}}  \leq \avnorm{u(\omega,\cdot)}{r}{D_{2R}}^\gamma$.
The assumption that $u\in \Stat(\R)$ and the ergodic theorem  then show that $u(\cdot,0)\leq \|u(\cdot,0)\|_{L^r(\Omega,\R)}^\gamma$. Hence, $u(\cdot,0)\in L^\infty(\Omega,\R)$. Furthermore, the assumption that $u\in \Stat(\R)$ and $\P$-a.s.
$\Laplace u=0$ and \cref{a07b} imply that $\P$-a.s. $\mathrm{D}^*\cdot\mathrm{D} u= \nabla^*\cdot \nabla u=\Laplace u=0$.
This and partial integration (combined with the fact that $u(\cdot,0)\in L^\infty(\Omega,\R)$) establish that
$\E [|(\mathrm{D}u)(\cdot,0 )|^2 ] =\E [(\mathrm{D}^*\cdot\mathrm{D} u)(\cdot,0) u(\cdot,0) ]=0$. Hence, 
$\P$-a.s.
$(\mathrm{D} u)(\cdot,0)=0$. Ergodicity then implies that $\P$-a.s.
$u(\cdot,0)=\E[u(\cdot,0)]$.
This and the fact that $u\in\Stat (\R)$ complete the proof of \cref{d20}.
\end{proof}
\cref{d21} below is useful since  only the gradients of the correctors  are stationary.
\begin{lemma}\label{d21}
Let $r\in [1,\infty)$ and let $u\colon \Omega\times \Z^d\to\R $ be a measurable function which satisfies that $\nabla^*u\in \Stat(\R)$, $(\nabla^*u)(\cdot,0)\in L^r(\Omega,\R ) $, 
$\E[(\nabla^*u)(\cdot,0)]=0$,
and $\P$-a.s. $\Laplace u =0 $. Then it holds for every $x\in\Z^d$, $\P$-a.e. $\omega\in\Omega$ that $u(\omega,x)=u(\omega,0)$.
\end{lemma}
\begin{proof}[Proof of \cref{d21}]
\cref{d20} (with $u\defeq \nabla^*_\ell u$ for $\ell\in [1,d]\cap\N$) shows for all $\ell\in [1,d]\cap\N$ that $\P$-a.s. $\nabla^*_\ell u=0 $. This completes the proof of \cref{d21}.
\end{proof}
\begin{proof}[\bf Proof of \cref{a12}]The construction of $\phi$ is well-known, see \cref{a31c}, and therefore omitted here. 
For the rest of the proof we often use Eintein's notation.
For every $r\in (1,2]$ let $L^r_{\nabla^*}$ be the closure of 
$\{\mathrm{D}^*\zeta\mid \zeta\in L^\infty(\Omega)\}$ in $L^r(\Omega,\R^d)$ and for all $i,j,k\in [1,d]\cap\N$,  $n\in \N$
let
$S_{ijk}^{(n)}
\in L^2_{\nabla^*}  $ be the projection onto
$L^2_{\nabla^*} $ of $(\fluxQ_{ik}\unit{j}-\fluxQ_{ij}\unit{k})\1_{|\fluxQ|\leq n}$, i.e., it holds for all 
$\,i,j,k\in [1,d]\cap\N$, $n\in \N$, $\zeta\in L^\infty(\Omega,\R)$  that $\E[ S_{ijk}^{(n)} 
\cdot \mathrm{D}^*\zeta]=\E[(\fluxQ_{ik}\unit{j}-\fluxQ_{ij}\unit{k})\1_{|\fluxQ|\leq n}\cdot \mathrm{D}^*\zeta]$.
Jensen's inequality and a simple approximation argument prove for all $i,j,k\in[1,d]\cap\N$, $n\in\N$ that 
$S_{ijk}^{(n)}\in L^{2p/(p+1)}_{\nabla^*}$.
\cref{b08} then shows for all $i,j,k\in [1,d]\cap\N$ that $(S_{ijk}^{(n)})_{n\in\N}$ is a Cauchy sequence in $L^{2p/(p+1)}(\Omega,\R^d)$ 
 and converges to $S_{ijk}\in  L^{2p/(p+1)}_{\nabla^*} $ which satisfies for all $\zeta\in L^\infty(\Omega,\R)$ that 
$\E[ S_{ijk}
\cdot \mathrm{D}^*\zeta]=\E[(\fluxQ_{ik}\unit{j}-\fluxQ_{ij}\unit k)\cdot \mathrm{D}^*\zeta]$. This and partial integration show   for all   $i,j,k\in [1,d]\cap\N$, $\zeta\in L^\infty(\Omega,\R)$  that
$\E[(\mathrm{D}\cdot S_{ijk})
\zeta]=\E[(\mathrm{D}\cdot(\fluxQ_{ik}\unit{j}-\fluxQ_{ij}\unit{k}))\zeta]$. Hence, it holds for all 
 $i,j,k\in [1,d]\cap\N$ that
$\mathrm{D}\cdot S_{ijk}=\mathrm{D}\cdot(\fluxQ_{ik}\unit{j}-\fluxQ_{ij}\unit{k})$ and
$\mathrm{D}\cdot (S_{ijk}+S_{ikj})=\mathrm{D}\cdot(\fluxQ_{ik}\unit{j}-\fluxQ_{ij}\unit{k}+\fluxQ_{ij}\unit{k}-\fluxQ_{ik}\unit {j})=0$.
Now, for every  $i,j,k\in [1,d]\cap\N$ let $\sigma_{ijk}\colon \Omega\times\Z^d\to\R$ be the unique measurable function which satisfies that $\sigma_{ijk}(\cdot,0)=0$, 
$\nabla^*\sigma_{ijk}\in\Stat(\R^d) $,
 and $(\nabla^*\sigma_{ijk})(\cdot,0)=S_{ijk}$ (the way to extend  $S_{ijk}^{(n)}=(\nabla^*\sigma_{ijk})(\cdot,0) $ to 
$\sigma_{ijk}$ easily done by using discrete contour integrals as done for $\phi$).
The fact that 
 $\forall\,i,j,k\in [1,d]\cap\N\colon \mathrm{D}\cdot S_{ijk}=\mathrm{D}\cdot(\fluxQ_{ik}\unit{j}-\fluxQ_{ij}\unit {k})$ and
$\nabla\cdot (S_{ijk}+S_{ikj})=0$ and \cref{a07b} then imply
for all $i,j,k\in [1,d]\cap\N$ that 
$-\Laplace \sigma_{ijk}=\nabla\cdot \nabla^*\sigma_{ijk}=\nabla\cdot(\fluxQ_{ik}\unit{j}-\fluxQ_{ij}\unit{k})$ and $-\Laplace (\sigma_{ijk}+\sigma_{ikj})=\nabla\cdot \nabla^*(\sigma_{ijk}+\sigma_{ikj})=0$. 
This, the fact that $\forall\,i,j,k\in [1,d]\cap\N\colon
\nabla^*\sigma_{ijk}\in\Stat(\R^d), 
 \nabla^*\sigma_{ijk}(\cdot,0)=S_{ijk}\in L^{2p/(p+1)}_{\nabla^*} $, and $\P$-a.s. $\sigma_{ijk}(\cdot,0)=0$, a simple approximation argument, and \cref{d21} imply for all 
$i,j,k\in [1,d]\cap\N$
that $ \E[(\nabla^*\sigma_{ijk})(\cdot,0)]=0$ and $\P$-a.s. $\sigma_{ijk}=-\sigma_{ikj}$.
Next, the fact that
$\forall\,i,j,k\in [1,d]\cap\N\colon- \Laplace \sigma_{ijk}=\nabla\cdot(\fluxQ_{ik}\unit{j}-\fluxQ_{ij}\unit{k})$
implies
  for all $i,j\in [1,d]\in\N$ that
$\P$-a.s. 
$
\Laplace (\nabla^*\cdot \sigma_{ij}+\fluxQ_{ij}) =\nabla^*_k \Laplace\sigma_{ijk}
+\Laplace \fluxQ_{ij}
= -\nabla^*_k \nabla_j\fluxQ_{ik}-\nabla_k^*\nabla_k \fluxQ_{ij} + \Laplace \fluxQ_{ij}=0.
$ This and \cref{d20} (combined with the fact that $\forall\,i,j,k\in[1,d]\cap\N\colon \E[\fluxQ_{ij}]=0=\E[\nabla^*\sigma_{ijk}]=0$) imply for all $i,j\in[1,d]\cap\N$ that $\P$-a.s. $\nabla^*\cdot \sigma_{ij}+\fluxQ_{ij}=0$. Next, \cref{b01,b05} (combined with the assumption that 
$1/p+1/q\leq 2/d$) and the fact that $\forall\,i,j,k\in [1,d]\cap\N\colon\nabla^*\sigma_{ijk}\in\Stat(\R^d) $ and $ \nabla^*\sigma_{ijk}(\cdot,0)=S_{ijk}\in L^{2p/(p+1)}_{\nabla^*} $
imply the sublinearity of $\sigma$ in \eqref{d12}. The sublinearity of $\phi$ in \cref{d12} follows similarly (we use  \cref{b01b,b05}, the assumption that
$1/p+1/q\leq 2/d$,
and a simple density argument). Finally, uniqueness of $\phi$ is clear and uniqueness of $\sigma$ follows from \cref{d21} combined with the fact that $\forall\,i,j,k\in[1,d]\cap\N \colon -\Laplace\sigma_{ijk}= \nabla_j\fluxQ_{ik}- \nabla _k\fluxQ_{ij}$, 
the fact that
$\forall\,i,j,k\in[1,d]\cap\N \colon \nabla\sigma_{ijk} \in \Stat(\R^d)$, and the fact that
$\forall\,i,j,k\in[1,d]\cap\N \colon$ $\P$-a.s. $\sigma_{ijk}(\cdot,0)=0$. The proof of \cref{a12} is thus completed.
\end{proof}

\section{Calculating the energy of the homogenization error}\label{x15b}
The main result of this section, \cref{x25}, calculates the energy of the homogenization error and therefore adapts the argument in 
p. 1399 in~\cite{BFO18} to the discrete. First, 
\cref{a73} prepares the equation for the homogenization error, i.e, adapts Step 1 of the proof of Theorem~2 in~~\cite{BFO18} into the discrete (p. 1395--6 in~\cite{BFO18}).
\cref{x25} is proved by testing the equation   of the homogenization error (\cref{a73}) with itself -- we will explain in \cref{b66} how to do this in the discrete. 

Throughout this section we use \cref{x12} and \emph{Einstein's notation} summing over repeated indices.
Here, the bracket notation is useful to write the product rule (see \cref{x12b}).
%Since we have already constructed the correctors

\begin{setting}\label{x12} 
Let
$\omega\in \Omega$,
let $\matA\colon \Z^d \to \R^{d\times d}$ be the function whose values are $\R^{d\times d}$ diagonal matrix and which satisfies
for all $x\in\Z^d$ that  $ \matA_{ij}(x)= \omega(\{x,x+\unit{i}\}) \1_{\{i\}}(j)$,
let 
$\matAhom\in \R^{d\times d}$, 
let $\phi_{i}\colon \Z^d \to\R$,  $\sigma_{ijk}\colon \Z^d \to\R$, 
$\fluxQ_{ij}\colon \Z^d \to\R$, $i,j,k\in [1,d]\cap\N$, be functions which
satisfy 
for all $i,j,k\in [1,d]\cap\N$
that
$\fluxQ_i=
\matA\left(\nabla\phi_i+\unit{i}\right)-\matAhom\unit{i}$,
$-\nabla^*\cdot \fluxQ_i=0$,
$-\nabla^*\cdot \sigma_{i}=\fluxQ_i$,
and $\sigma_{ijk}=-\sigma_{ikj}$.
For every  $f,g\colon \Z^d\to \R  $, $F\colon \Z^d\to\R^d$ let $(f \nabla_i) g ,(f\nabla_i^*)g\colon \Z^d\to\R$, $i\in[1,d]\cap \N$, be the functions which satisfy for all  $x\in\Z^d$, $i\in[1,d]\cap \N$ that
\begin{align}
\left[(f \nabla_i) g \right](x)= f(x+\unit{i}) (\nabla_i g) (x)\quad\text{and}\quad
[(f\nabla_i^*)g](x)= f(x-\unit{i})\nabla^*_i g(x),
\end{align}
let $(f\nabla)g,(f\nabla^*)g\colon \Z^d \to\R^d $ be the vector fields which
satisfy for all $x\in\Z^d$ that
\begin{align}
[(f\nabla)g] (x)=\left[\sum_{i=1}^d [(f \nabla_i) g](x)  \unit{i}\right]
\quad\text{and}\quad[(f\nabla^*)g] (x)=\left[\sum_{i=1}^d [(f \nabla_i^*) g](x) \unit{i}\right],
\end{align}
and
let 
$(F\cdot \nabla)g, (F\cdot \nabla^*)g\colon \Z^d\to \R$ be the functions which satisfy for all $x\in\Z^d$ that
\begin{align}
[(F\cdot \nabla)g ](x) = \sum_{i=1}^d [(F_i \nabla_i)g](x)
\quad\text{and}\quad
[(F\cdot \nabla^*)g](x) = \sum_{i=1}^d [(F_i \nabla_i^*)g](x).
\end{align}
\end{setting}%
\begin{remark}\label{x12b}
Using the notation in \cref{x12} we obtain for all $f,g\colon \Z^d\to\R$, $x\in\Z^d$ that
\begin{align}\begin{split}
[\nabla_i(fg)] (x) 
&=f(x+\unit{i})[g(x+\unit{i})-g(x)]+ f(x)[g(x+\unit{i})-g(x)]\\ &=[(f\nabla_i) g](x)+[g\nabla_i f](x).
\end{split}\end{align}
This shows for all $f,g\colon \Z^d\to\R$ that $\nabla_i(fg)=(f\nabla_i) g + g\nabla_i f$, which, up to the brackets, looks the same as its continuum counterpart, $\forall\, f,g\in C^1(\R^d,\R)\colon \partial_i(fg)=f\partial_i g + g\partial_i f$.
%This lightens the notation very well when we  adapt the computations on the homogenization error (formula (29) in~\cite{BFO18}) into the discrete case as seen in \cref{a73}. 
\end{remark}

\begin{lemma}[Equation for the homogenization error]\label{a73}
Let
$x\in \Z^d $,
 $u,v,w,\eta\colon \Z^d\to\R $  satisfy that 
$
w(x)=u(x)-v(x)-\eta(x)\phi_i(x)(\nabla_i v) (x)$ and $(\nabla^*\cdot \matA\nabla u)(x)= (\nabla^ * \cdot \matAhom \nabla v)(x)=0$.
Then it holds that
\begin{align}\label{a74}\begin{split}
(\nabla^*\cdot \matA \nabla w)(x) = &-\left( \nabla^*\cdot ((1-\eta)(\matA-\matAhom)\nabla v)\Big.\right)(x)\\&- \left(\nabla^*\cdot [(\sigma_i\cdot \nabla^*)(\eta\nabla_iv)]\Big.\right)(x) -\left( \nabla^*\cdot \matA[(\phi_i\nabla)(\eta\nabla_i v)]\Big.\right)(x).
\end{split}\end{align}
\end{lemma}
\begin{proof}[Proof of \cref{a73}]
The product rule shows that
\begin{align}
( \nabla w)(x)= \left(\Big.\nabla u- \nabla v- \eta\nabla_iv\nabla\phi_i - (\phi_i\nabla)(\eta\nabla_i v)\right)(x).
\end{align}
This and the fact that $(\nabla^*\cdot \matA\nabla u)(x)=0 $ imply that
\begin{align}
(\nabla^*\cdot \matA\nabla w)(x) = -\left(\nabla^* \cdot(\matA\nabla v +	\eta \nabla_iv \matA\nabla\phi_i)\Big.\!\right)(x)	-\left(\nabla^ *\cdot \matA\left((\phi_i\nabla) (\eta \nabla_iv\big.)\right)\Big.\!\right)(x).\label{x20}
\end{align}
This together with the fact that
\begin{equation}
\begin{aligned}
(\matA\nabla v +	\eta \nabla_iv \matA\nabla\phi_i)
&=(1-\eta)\matA\nabla v+\eta\nabla_i v \matA(\unit{i}+\nabla \phi_i)
\\
&= \left((1-\eta)(\matA-\matAhom)\nabla v\Big. \right) +\left( (1-\eta)\matAhom\nabla v+\eta\nabla_i v \matA(\unit{i}+\nabla \phi_i)\Big. \right)
\end{aligned}
\end{equation}
shows that
\begin{equation}\label{a77}
\begin{aligned}
&(\nabla^*\cdot \matA\nabla w)(x)=-\left(\Big.\!\nabla^*\cdot \left(\big.(1-\eta)(\matA-\matAhom)\nabla v\right)\right)(x)\\ &\qquad-\left(\nabla^*\cdot  \left((1-\eta)\matAhom\nabla v+\eta\nabla_i v \matA(\unit{i}+\nabla \phi_i)\big.\right)\Big.\right)(x)-\left(\nabla^ *\cdot \matA \left((\phi_i\nabla) (\eta \nabla_iv\big.)\right)\Big.\!\right)(x).
\end{aligned}
\end{equation}
Furthermore, the product rule and the fact that $(\nabla^ * \cdot \matAhom \nabla v)(x)=0$ imply that
\begin{align}\begin{split}
&-\left(\nabla^*\cdot \left((1-\eta)\matAhom\nabla v\right)\Big.\!\right)(x)=\left(\nabla^*\cdot\left( \eta\nabla_i v \matAhom\unit{i}\right)\Big.\!\right)(x)\\&=\left(\eta\nabla_iv \nabla^*\cdot \matAhom\unit{i}\Big.\!\right)(x)+\left(\Big. (\matAhom\unit{i}\cdot\nabla^*)(\eta\nabla_iv)\right) (x) =\left(\Big.(\matAhom\unit{i}\cdot\nabla^*)(\eta\nabla_iv)\right)(x).
\end{split}\label{x13}
\end{align}
In addition, the product rule  and the assumption that  $\nabla^ *\cdot(\matA(\unit{i}+\nabla\phi_i))=0$ show that
\begin{align}\begin{split}
-\nabla^*\cdot\left(\eta\nabla_i v \matA(\unit{i}+\nabla \phi_i) \right)
&=- \eta\nabla_i v \nabla^* \cdot \matA(\nabla\phi_i+ \unit{i})-(\matA(\unit{i}+\nabla\phi_i)\cdot \nabla^ *)(\eta\nabla_i v)\\
&=-(\matA(\unit{i}+\nabla\phi_i)\cdot \nabla^ *)(\eta\nabla_i v).
\end{split}\label{x14}
\end{align}
Next, the product rule and the fact $\forall i,j,k\in [1,d]\cap\N \colon \sigma_{ijk}=-\sigma_{ikj}$ imply  for all 
$\xi\colon \Z^d\to \R^d$ that 
\begin{align}
\begin{aligned}
\left(\left(\nabla^*\cdot \sigma_i\right) \cdot \nabla^*\right)\xi_i 
&= \left(\nabla^*_k\sigma_{ijk}) \cdot \nabla^*_j\right)\xi_i
= \nabla^*_j(\xi_i\nabla^*_k\sigma_{ijk}) - \xi_i \nabla^*_j\nabla^*_k \sigma_{ijk}
=\nabla^*_j(\xi_i\nabla^*_k\sigma_{ijk})\\
&=\nabla^*_j\left(\nabla^*_k (\xi_i \sigma_{ijk})-(\sigma_{ijk}\nabla^*_k)\xi_i\Big.\right)
=\nabla^*_j \nabla^*_k (\big.\xi_i \sigma_{ijk})-\nabla^*_j(\big.(\sigma_{ijk}\nabla^*_k)\xi_i)
\\
&=
- \nabla_j^* ((\sigma_{ij}\cdot\nabla^*)\xi_i)
=-\nabla^*\cdot((\sigma_i\cdot\nabla ^*)\xi_i).
\end{aligned}
 \label{a79}.
\end{align}
This (with $\xi\defeq\eta\nabla v$), \eqref{x13},  \eqref{x14}, and the assumption that $-\nabla^*\cdot\sigma_i=\fluxQ_i$ yield that
\begin{align}\begin{split}
&-\left(\nabla^ *\cdot( (1-\eta)\matAhom\nabla v+\eta\nabla_i v \matA(\unit{i}+\nabla \phi_i))\Big.\right)(x)\\&=-\left(\left(\Big. [\matA(\unit{i}+\nabla\phi_i)-\matAhom\unit{i}]\cdot \nabla^*\right)(\eta\nabla_i v) \right)(x)\\
&= \left(\Big.\!\left(\big.(\nabla^*\cdot \sigma_i)\cdot \nabla^*\right)(\eta\nabla_iv)\right)(x)
=-\left(\nabla^*\cdot\left((\sigma_i\cdot\nabla ^*)(\eta\nabla_iv)\big.\right)\Big.\!\right)(x)\label{a78}.
\end{split}
\end{align}
Combining this with \eqref{a77} we obtain that
\begin{align}\begin{split}
(\nabla^*\cdot \matA \nabla w)(x) = &-\left( \nabla^*\cdot ((1-\eta)(\matA-\matAhom)\nabla v)\Big.\right)(x)\\&- \left(\nabla^*\cdot [(\sigma_i\cdot \nabla^*)(\eta\nabla_iv)]\Big.\right)(x) -\left( \nabla^*\cdot \matA[(\phi_i\nabla)(\eta\nabla_i v)]\Big.\right)(x).
\end{split}\end{align}
The proof of \cref{a73} is thus completed.
\end{proof}
For the next step recall  \cref{b66b}. The reader should keep in mind \cref{f01}.
\begin{lemma}\label{b66}Let $R\in \N$,  $g\colon \E^d_{\pm}\to\R $, $w\colon \Z^d\to\R $, $h,f\colon\Z^d\to\R^d$,
assume  for all $x,y\in\Z^d$ with $(x,y)\in \E^d_{\pm}$ that $ g(x,y)=-g(y,x)$,
assume for all $x\in\Z^d$, $i\in [1,d]\cap\N$ that
$ h_i(x)=g(x,x+\unit{i})$,
assume for all $x\in D_R$ that $ (\nabla^*\cdot h)(x)=(\nabla^ *\cdot f)(x)$,
and assume for all 
$ x\in \Z^d\setminus D_{R-1}$ that $ f(x)=0$.
Then it holds that
\begin{align}\label{b71}
\left[\sum_{\substack{x,y\in \Z^d\colon (x,y)\in E_R}}
g(x,y)\nablaBar_{xy}w\right] =
2\left[
\sum_{y\in\tilde{\partial} D_R}
g(\funcNu{R}(y))
w(y) \right]
+2\left[
\sum_{x\in D_R}\nabla w(x)\cdot  f(x)\right]
 .
\end{align}
\end{lemma}
%%%%%%%%%%%%%%%%%%%%%%%
\begin{proof}[Proof of \cref{b66}]
Let $w_0\colon \Z^d\to\R$ be the function which satisfies for all $x\in \Z^d$ that 
$w_0(x)=w(x) \1_{\Z^d\setminus D_R}(x) $. Then 
for all $x\in  D_{R-1}$ it holds that $(\nabla w_0)(x)=0$. This and the assumption that
$\forall\, x\in \Z^d\setminus D_{R-1}\colon  f(x)=0$ imply for all $x\in\Z^d$ that $(\nabla w_0) (x)\cdot f(x)=0$.
Next, the fact that $\forall\,x\in \Z^d\setminus D_R\colon (w-w_0)(x)=0$  and the assumption that
$\forall\,x\in D_R\colon  (\nabla^*\cdot h)(x)=(\nabla^ *\cdot f)(x)$ imply for all $x\in\Z^d$ that
$
\sum_{x\in \Z^d}(w-w_0)(x)(\nabla^ *\cdot f)(x)
=\sum_{x\in \Z^d}(w-w_0)(x)(\nabla^ *\cdot h)(x).
$
The fact that $\forall\,x\in\Z^d\colon (\nabla w_0) (x)\cdot f(x)=0$ and discrete partial integration then show that
\begin{align}\label{b72b}
\sum_{x\in \Z^d}(\nabla w)(x)\cdot  f(x)=
\sum_{x\in \Z^d}(\nabla w-\nabla w_0)(x)\cdot  f(x)=
\sum_{x\in \Z^d}(\nabla w-\nabla w_0)(x)\cdot  h(x).
\end{align}
Furthermore, the fact that $\forall\,x\in \Z^d\setminus D_R\colon (w-w_0)(x)=0$ implies for all $i\in [1,d]\cap \N$, $x\in \Z^d$ that
$\left(\nabla_i w-\nabla_i w_0\right)(x)=(\nabla_iw-\nabla_iw_0)(x)\1_{(x,x+\unit{i})\in E_R} .$ This and the assumption that $\forall\,x\in\Z^d,i\in [1,d]\cap\N\colon h_i(x)=g(x,x+\unit{i})$ imply that
\begin{align}
\left[\sum_{x\in \Z^d} h(x)\cdot (\nabla w-\nabla w_0)(x)\right]=
\sum_{x\in \Z^d}\left[\sum_{i=1}^d \1_{(x,x+\unit{i})\in E_R}g(x,x+\unit{i})(\nabla_iw-\nabla_iw_0)(x)\right].
\end{align}
This and \eqref{b72b} demonstrate that
\begin{align}
\left[\sum_{x\in D_R}\nabla w(x)\cdot  f(x)\right]=\sum_{x,y\in \Z^d}\left[\sum_{i=1}^d\1_{(x,y)\in E_R}\1_{y=x+\unit{i}} g(x,y)\nablaBar_{xy}(w-w_0)\right].\label{x21}
\end{align}
Swapping $x$ and $y$ and using the assumption that $\forall x,y\in \E _\pm^d\colon g(x,y)=-g(y,x)$ yield that
\begin{align}\begin{split}
\left[\sum_{x\in D_R}\nabla w(x)\cdot  f(x)\right]
&=\sum_{x,y\in \Z^d}\left[\sum_{i=1}^d\1_{(y,x)\in E_R}\1_{x=y+\unit{i}} g(y,x)\nablaBar_{yx}(w-w_0)\right]\\
&=\sum_{x,y\in \Z^d}\left[\sum_{i=1}^d\1_{(y,x)\in E_R}\1_{x=y+\unit{i}} g(x,y)\nablaBar_{xy}(w-w_0)\right].\end{split}\label{x22}
\end{align}
Adding this and \eqref{x21}  
we obtain  that
\begin{align}\begin{split}
&2\left[\sum_{x\in D_R}\nabla w(x)\cdot  f(x)\right]\\
&=
\left[\sum_{x,y\in \Z^d}
\left[\sum_{i=1}^d\left(
\1_{(x,y)\in E_R}\1_{y=x+\unit{i}}+\1_{(y,x)\in E_R}\1_{x=y+\unit{i}}\right)\right] g(x,y)\nablaBar_{xy}(w-w_0)\right] \\
&=\left[
\sum_{x,y\in \Z^d }\1_{(x,y)\in E_R}
 g(x,y)\nablaBar_{xy}(w-w_0)\right]=\left[
\sum_{x,y\in \Z^d\colon (x,y)\in E_R }
 g(x,y)\nablaBar_{xy}(w-w_0)\right]
 .\end{split}
\end{align}
This shows that
\begin{align}
\left[\sum_{\substack{x,y\in \Z^d\colon  (x,y)\in E_R}}
g(x,y)\nablaBar_{xy}w\right] =
\left[
\sum_{\substack{x,y\in \Z^d\colon (x,y)\in E_R}}
g(x,y)\nablaBar_{xy}w_0\right]
+2\!\!\left[
\sum_{x\in D_R}\nabla w(x)\cdot  f(x)\right]
 .
\end{align}
Next, the fact that
$\forall\,x\in D_R \colon w_0(x)=0$
 shows  for all $x,y\in\Z^d$ with $(x,y)\in E_R$  that
\begin{align}\label{x23}\begin{split}
\nablaBar_{xy}w_0=w_0(y)-w_0(x)
&= 
\left(w_0(y)-w_0(x)\right)\left(\1_{x\in D_R,y\in \partial D_R}+\1_{y\in D_R,x\in \partial D_R}\right)\\
&=w_0(y)\1_{x\in D_R,y\in \partial D_R}-w_0(x)\1_{y\in D_R,x\in \partial D_R}.\end{split}
\end{align}
This, the assumption  that $\forall\,x,y\in\Z^d,\, (x,y)\in \E _\pm^d\colon g(x,y)=-g(y,x)$,  the fact that $\forall x\in \partial D_R\colon w_0(x)=w(x)$, and the definitions of
$\enor{R}$ and $\funcNu{R}$ in \cref{b66b}  prove that
\begin{align} \begin{split}
&\left[\sum_{\substack{x,y\in\Z^d\colon  (x,y)\in E_R}}
g(x,y)\nablaBar_{xy}w_0\right]\\&=\left[
\sum_{\substack{ x,y\in\Z^d\colon(x,y)\in E_R}}
g(x,y)\left[
w_0(y)\1_{x\in D_R,y\in \partial D_R}-w_0(x)\1_{y\in D_R,x\in \partial D_R}\Big.\right] \right]\\
&=\left[
\sum_{\substack{ x,y\in\Z^d\colon(x,y)\in E_R}}
g(x,y)
w_0(y)\1_{x\in D_R,y\in \partial D_R}+g(y,x)w_0(x)\1_{y\in D_R,x\in \partial D_R}\right] \\
&=2\left[
\sum_{\substack{ x,y\in\Z^d \colon (x,y)\in  \enor{R}}}
g(x,y)
w(y)\right]= 2\left[
\sum_{y\in\tilde{\partial} D_R}
g(\funcNu{R}(y))
w(y) \right].\end{split}\label{x24}
\end{align} 
This and \eqref{x23} complete the proof of \cref{b66}.
\end{proof}

\begin{corollary}[Energy of the homogenization error]\label{x25}
Let $R\in [8,\infty)\cap\N,$ 
$u,v,w\colon \overline{D}_R\to\R$, $\eta\colon \Z^d\to\R $,
assume that $\matAhom$ is a diagonal matrix, 
let $\omega_\mathrm{h}\in\Omega$ satisfy that 
for all $ x\in\Z^d$, $i\in [1,d]\cap\N$ that $ \omegaH(\{x,x+\unit{i}\}) =(\matAhom)_{ii}$, assume for all $x\in\overline{D}_R$ that
 $w(x)=u(x)-v(x)-\eta(x)\phi_i(x)(\nabla_i v)(x)$, assume for all
$x\in \Z^d\setminus D_{R-4}$ that $ \eta(x)=0$, and assume for all $ x\in D_R$ that $
(\nabla^*\cdot\matA\nabla u)(x)= (\nabla^ * \cdot \matAhom \nabla v)(x)=0$. Then it holds that
\begin{align}\label{x26}
\begin{split}
&\left[\sum_{x,y\in\Z^d\colon (x,y)\in E_{R}} \omega_{xy}\left(\nablaBar_{xy} w\right)^2\right]
= -\left[\sum_{x,y\in\Z^d\colon(x,y)\in E_{R}}(1-\eta(x))\left((\omega-\omegaH)\nablaBar v\nablaBar w\right)_{xy}\right]\\
&+2\left[\sum_{y\in\tilde{\partial}D_R}\left(u(y)-v(y)\right)
\left(\omega\nabla u-\omega_\mathrm{h} \nabla v\right)_{\funcNu{R}(y)}\right]\\
&
-2\left[\sum_{x\in {D}_{R}} \left(\sigma_i\cdot \nabla^*)(\eta\nabla_iv) +  \matA\left((\phi_i\nabla)(\eta\nabla_i v)\right)\right)(x)\cdot (\nabla w)(x)\right].\end{split}
\end{align}
\end{corollary}
\begin{proof}[Proof of \cref{x25}]Throughout the proof we  assume that
$u,v,w\colon \Z^d\to\R$ by using any extensions which satisfies  for all $x\in\Z^d\setminus\overline{D}_R$ that $u(x)=v(x)$ and $w(x)=0$. The assumption that
$\forall\, x\in\overline{D}_R \colon w(x)=u(x)-v(x)-\eta(x)\phi_i(x)(\nabla_i v)(x)$
and the assumption that
$\forall\,x\in \Z^d\setminus D_{R-4}\colon  \eta(x)=0$ then
imply  for all $x\in\Z^d$ that $w(x)=u(x)-v(x)-\eta(x)\phi_i(x)(\nabla_i v)(x)$.
Next, let $f,h\colon \Z^d\to\R $ be the functions given by
\begin{align}\label{x27}
h= \matA\nabla w+ (1-\eta)(\matA-\matAhom)\nabla v\quad\text{and}\quad 
f=-  (\sigma_i\cdot \nabla^*)(\eta\nabla_iv) -  \matA\left(\big.(\phi_i\nabla)(\eta\nabla_i v)\right),
\end{align}
and let $g\colon\E ^d_\pm\to\R  $ be the function which satisfies that for all $(x,y)\in  \E ^d_\pm$ it holds that
\begin{align}\label{x28}
g(x,y)=(\omega \nablaBar w)_{xy} +(1-\eta(x))\left((\omega-\omega_\mathrm{h})\nablaBar v \right)_{xy}.
\end{align}
The fact that $\forall\,x\in \Z^d\setminus D_{R-4}\colon  \eta(x)=0$ and the fact that
$\forall\, x\in\Z^d\colon w(x)=u(x)-v(x)-\eta(x)\phi_i(x)(\nabla_i v)(x)$ imply 
 for all $y\in \tilde{\partial}D_R$ that $w(y)= u(y)-v(y)$ and
\begin{align}
\begin{aligned}
g(\funcNu{R}(y))
&=\left(\omega\nablaBar w +(\omega-\omega_\mathrm{h})\nablaBar v  \right)_{\funcNu{R}(y)}\\
&
=\left(\omega(\nablaBar u-\nablaBar v)+(\omega-\omega_\mathrm{h})\nablaBar v\right)_{\funcNu{R}(y)} 
= (\omega\nablaBar{u}-\omega_\mathrm{h}\nablaBar{v})_{\funcNu{R}(y)}.
\end{aligned}
\end{align}
Therefore, it holds that
\begin{align}
\begin{split}
&\left[\sum_{y\in \tilde{\partial}D_R}
w(y)g(\funcNu{R}(y))\right]=\left[
\sum_{y\in \tilde{\partial}D_R}
(u(y)-v(y))
\left(\omega\nablaBar{u}-\omega_\mathrm{h}\nablaBar{v}\right)_{\funcNu{R}(y)}\right].
\end{split}
\end{align}
Furthermore, 
\cref{a73} and \eqref{x27} imply that $\nabla^*\cdot h=\nabla^*\cdot f$. This,
\eqref{x27}, \eqref{x28} and
\cref{b66} (with  $f\defeq f$, $g\defeq g$, $h\defeq h$, $w\defeq w$)
ensure that 
\begin{align}
\begin{split}
&\left[\sum_{\substack{ x,y\in\Z^d\colon  (x,y)\in E_{R}}} \left[(\omega \nablaBar w)_{xy} +(1-\eta(x))\left((\omega-\omega_\mathrm{h})\nablaBar v \right)_{xy}\right]\nablaBar_{xy} w\right]\\
&=\left[ \sum_{\substack{  x,y\in\Z^d\colon (x,y)\in E_{R}}}g(x,y)\nablaBar_{xy}w\right]=
2\left[\sum_{y\in \tilde{\partial}D_R}
w(y)g(\funcNu{R}(y))\right]+2\left[\sum_{x \in D_R}\nabla w(x)\cdot f(x)\right]
\\&=2\left[
\sum_{y\in \tilde{\partial}D_R}
(u(y)-v(y))
\left(\omega\nablaBar{u}-\omega_\mathrm{h}\nablaBar{v}\right)_{\funcNu{R}(y)}\right]
\\&\qquad-2\left[\sum_{x\in {D}_{R}} \left((\sigma_i\cdot \nabla^*)(\eta\nabla_iv) +  \matA\left(\big.\!\left(\phi_i\nabla)(\eta\nabla_i v\right)\right)\right)(x)\cdot \nabla w(x)\right].\end{split}
\end{align}
This implies \eqref{x26}. The proof of \cref{x25} is thus completed.
\end{proof}

\section{Smoothing functions in the discrete case}\label{b17}
It is known  that continuous functions on the $d$-dimensional unit sphere can be smoothed since in this case there is an explicit formula for the convolution kernel (see formulas (50) and (51) in Bella, Fehrman, and Otto~\cite{BFO18}). In fact, we can even smooth functions on boundaries of Lipschitz domains, e.g., on the surface of the unit box $[0,1]^d$, as in \cref{b34} below. Roughly speaking, 
up to some technicality, all results in~\cite{BFO18}  are still true for continuum boxes instead of balls.

For \cref{b34} and also for the rest of this section we  use the  notation in \cref{x01} below.
\begin{setting}\label{x01}
For every metric space $X$ 
let $\mathrm{Lip}(X)$ denote the set of Lipschitz continuous functions from $X$ to $\R$.
For every bounded Lipschitz domain $\Omega\subseteq \R^d $,  i.e.,  the boundary $\partial\Omega$ is, roughly speaking, locally the graph of a Lipschitz function,  and for every $f\in \mathrm{Lip}(\partial\Omega)$
 denote by 
$\int_{\partial\Omega} f \,d\sigma$ 
the surface integral of $f$ on $\partial\Omega$
(see Section 2.1 in~\cite{MM13}) and denote by $\nabla^\mathrm{tan}f$ 
the tangential gradient of $f$ on $\partial\Omega$, which is well-defined $d\sigma$-almost everywhere on $\partial\Omega$
(see Section 2.2 in~\cite{MM13}).  
For
every bounded Lipschitz domain $\Omega\subseteq \R^d $ and every $p\in [1,\infty]$
the Lebesgue space $L^p(\partial\Omega)$ with the norm $\|\cdot \|_{L^p(\partial\Omega)}$
and the Sobolev space $L^p_1(\partial\Omega)$ with the norm $\|\cdot \|_{L^p_1(\partial\Omega)}$ are given in a usual way, e.g.,
$\forall\, f\in L^p_1(\partial \Omega)\colon \|f \|_{L^p_1(\partial\Omega)}\colon = \|f \|_{L^p(\partial\Omega)} +
\|\nabla^\mathrm{tan}f \|_{L^p(\partial\Omega,\R^{d-1})}
$ (see Section 2.3 in~\cite{MM13}). 
\end{setting}
\begin{lemma}[Smoothing operators on  boundaries of Lipschitz domains]\label{b34}Let $\Omega$ be a bounded Lipschitz domain. Then there exist 
$\epsilon_0\in (0,1)$, $C_1\colon  (1,\infty)\times(1,\infty)\to (0,\infty)$, $C_2\colon [1,\infty]\to (0,\infty)$
and there exist linear operators $\contSmooth{\epsilon}\colon \mathrm{Lip}(\partial\Omega)\to \mathrm{Lip}(\partial\Omega)$, $\epsilon\in (0,\epsilon_0)$,
such that  for all $u\in \mathrm{Lip}(\partial\Omega)$, 
$r,s\in [1,\infty]$ with $1\leq  s\leq r\leq  \infty$,  $\epsilon\in (0,\epsilon_0)$
it holds that
\begin{align}
\label{b35}
\|\contSmooth{\epsilon}u\|_{L^r(\partial\Omega)}
&\leq C_1(r,s) \epsilon^{-(d-1)(\frac1s-\frac1r)}\|u\|_{L^s(\partial\Omega)},\\
\label{b37}
\|\nabla^\mathrm{tan} (\contSmooth{\epsilon}u)\|_{L^r(\partial\Omega)}
&\leq C_1(r,s) \epsilon^{-(d-1)(\frac1s-\frac1r)}\|\nabla^\mathrm{tan}u\|_{L^s(\partial\Omega)},\\
\label{b36}
\|u-\contSmooth{\epsilon}u\|_{L^s(\partial\Omega)}&\leq C_2(s) \epsilon\|\nabla^\mathrm{tan}u\|_{L^s(\partial\Omega)}.
\end{align}
\end{lemma}
 A typical construction for the sequence of smoothing operators 
$(\mathcal{S}_\epsilon)$ in \cref{b34} 
 can be sketched as follows: 
\begin{inparaenum}[(i)]
\item construct a partition of unity, \item 
decompose functions into ''small pieces'' compactly supported on local charts, \item lift them to the Euclidean space to smooth them using a usual sequences of mollifiers (the scale of the mollifiers should be small enough so that the outcome functions are still supported in the corresponding local charts), 
and \item add those outcome functions together. 
\end{inparaenum}
Since it is quite technical (but more or less straightforward) and it is not the novelty of this paper, the proof of \cref{b34} is omitted.

%\subsection{A linear interpolation of discrete functions to the continuum}\label{b18}
A more important issue is to find an idea to adapt \cref{b34} to the discrete case.
To the best  of the author's knowledge, it is difficult to directly smooth functions in the discrete case where the scale of a mollifier must be at least $1$.
Here, we choose a hybrid solution: we will interpolate discrete functions to the continuum and smooth them (as continuum funtions) using \cref{b34}. First,
there are several ways to interpolate discrete functions. For convenience we choose the polilinear interpolation already used by Deuschel, Giacomin, and Ioffe (see formula (1.17) in~\cite{DGI00}), which allows to start quickly from scratch. 
In \cref{b19b} the continuum interpolation 
of a given discrete function is constructed on each $(d-1)$-dimensional unit box by moving this unit box to the reference box $[0,1]^{d-1}$ and
using \eqref{b19}. Note that \cref{b19c} ensures  the consistency of this construction.
\begin{setting}[Interpolation for the reference element]\label{b19a}Let $L\colon\R^{ \{0,1\}^{d-1}}\to C([0,1]^{d-1},\R ) $ be the operator which satisfies for all $u\colon \{0,1\}^{d-1}\to \R $, $x\in [0,1]^{d-1}$ that
\begin{align}\label{b19}
(L u)(x)= \sum_{a\in \{0,1\}^{d-1}} \left[\left(\prod_{i=1}^{d-1} \left(a_ix_i +(1-a_i)(1-x_i)\right)\right) u(a)\right].
\end{align}
\end{setting}
\begin{lemma}
\label{b19c}
Assume \cref{b19a}, let $k\in [1,d-1]$, let $j_1,\ldots,j_k\in [1,d-1]\cap\N$ satisfy that $j_1\leq \ldots\leq j_k$, let
$a_{j_1},\ldots,a_{j_k}\in \{0,1\}$, and let $u\colon \{0,1\}^{d-1}\to\R$.
Then the values of $Lu$ on the $(d-1-k)$-dimensional box
$\left\{x\in [0,1]^{d-1}\colon x_{j_1}=a_{j_1},\ldots,x_{j_k}=a_{j_k} \right\}$
are uniquely determined by the values of $u$ on
the set of vertices 
$\left\{x\in \{0,1\}^{d-1}\mid x_{j_1}=a_{j_1},\ldots,x_{j_k}=a_{j_k} \right\}$.
\end{lemma}
\begin{setting}[Global interpolation and nodal functions]\label{b19b}Assume \cref{b19a}. For every $R\in\N$  we chop 
$\partial C_R$ into $(d-1)$-dimensional unit boxes and denote this set of boxes by $\mathcal{S}_R$ which is formally given by
$
\mathcal{S}_R=\{ 
\Gamma = a+ ([0,1]^{j-1}\times \{0\}\times [0,1]^{d-j}\mid 
(\Gamma\subseteq \partial C_R)\wedge  (a\in\Z^d) \wedge (j\in [1,d]\cap \N )\} $.
For every $R\in\N$, $S\in \mathcal{S}_R$
let $A_S^{(R)}\colon \R^{d-1}\to\R^d$  be the bijective affine linear transformation which satisfies for all $S\in \mathcal{S}_R$  that $S= A_S^{(R)}(\{0,1\}^{d-1})$.
For every $R\in\N$
let $T_R\colon \R^{\partial D_R}\to C(\partial C_R,\R)$ be the mapping which satisfies that for all  $u\colon\partial D_R\to\R$,
$S\in \mathcal{S}_R$, $x\in S$ it holds that
$(T_R (u))(x)= (L (u\circ A_S^{-1})) (A_S^{-1}x)$.
For every $R\in\N$, $x\in\partial D_R$ 
let
$\nodal{R}{x}\in C(\partial C_R,\R) $ be the function which satisfies that
 $\nodal{R}{x}= T_R(\1_{\{x\}}\upharpoonright_{\partial D_R})$ where we write $\1_{\{x\}}\upharpoonright_{\partial D_R}\colon\partial D_R\to\R $ to denote the function which satisfies for all 
$y\in\partial D_R\setminus\{x\}$ that $(\1_{\{x\}}\upharpoonright_{\partial D_R}) (y)=0$ and 
$(\1_{\{x\}}\upharpoonright_{\partial D_R}) (x)=1$.
\end{setting}
In finite element the functions $\nodal{R}{x}$, $R\in\N$, $x\in\partial D_R$, are often called nodal functions. \cref{b20} is straightforward and its proof is therefore omitted.
\begin{lemma}\label{b20}Assume \cref{b19b}. Then 
\begin{enumerate}[i)]
%\item 
%it holds for all $R\in\N$ that
%$T_R $ is  linear  and injective,
\item \label{b20b}
it holds for all $R\in\N$ that
$\{\nodal{R}{x}\mid x\in \partial D_R\}$ is linearly independent, 
\item \label{b20a}
it holds for all $R\in\N$, $\Gamma\in\mathcal{S}_R$, $x\in \partial D_R\setminus \Gamma$, $y\in \Gamma$ that $\varphi_x^{(R)}(y)=0$,
\item \label{b20d}
it holds for all $R\in\N$, $y\in \partial C_R$ that
$(T_R(f))(y)= \sum_{x\in\partial D_R}f(x)\varphi_x^{(R)}(y)$,
and
\item \label{b20e}
there exists $C\colon [1,\infty]\to (0,\infty)$ which satisfies  
for all $R\in \N$, $u\colon \partial D_R\to \R $  that 
\begin{align}\label{b26}
C(p)^{-1} \normB{T_R u}{p}{\partial{C}_R}\leq 
\norm{u}{p}{\partial{D}_R} \leq C(p) \normB{T_R u}{p}{\partial{C}_R}
\end{align}
and
\begin{align}
\label{b27}
C(p)^{-1}\normB{\nabla^\mathrm{tan}(T_R u)}{p}{\partial{C}_R}\leq 
\norm{\overline\nabla u}{p}{\etan{R}} \leq C(p)\normB{\nabla^\mathrm{tan}(T_R u)}{p}{\partial{C}_R}.\end{align}
\end{enumerate}
\end{lemma}
After being interpolated and smoothed a function $u\colon \partial D_R\to\R$ has already been very much deformed. In general, the outcome continuum function does not represent a discrete function, i.e., it is not in
$\mathrm{span}\{\phi_x^{(R)}\mid x\in\partial D_R\}$  and we therefore have to make it discrete again.
A natural way to define a discrete function $Af\colon\partial D_R\to\R$ from a given continuum function ${f}\in C(\partial C_R,\R)$ is, e.g.,
to assign each $x\in\partial D_R$ a
$(d-1)$-dimensional unit box 
$\Gamma_x\in \mathcal{S}_R$  such that $x\in\Gamma_x$ and
to let $Af$ be given for every $x\in\partial D_R$ by
$(Af)(x)= \int_{\Gamma_x} fd\sigma$. However, this kind of averaging is not a projection, i.e., $AT_R \neq \mathrm{id}$. In \cref{b39c} below we use the idea by Scott and Zhang~\cite{SZ90} to overcome this problem. Intuitively, this idea is to insert
$\psi^{(\Gamma,R)}_x$ in the integral in \eqref{c03} and
require \eqref{c01} to make $\Pi^{(R,\Gamma)}$ a projection.
\begin{lemma}[Scott-Zhang projection]\label{b39c}
Assume \cref{b19b}. 
For every $R\in \N$ let $ \mathbf{A}_R$ be the set 
given by
$\mathbf{A}_R=\{\Gamma=(\Gamma_x)_{ x\in \partial D_R}\subseteq \mathcal{S}_R\mid 
\forall x\in \partial D_R\colon 
x\in\Gamma_x\cap \Z^d\}$, i.e., $\mathbf{A}_R$ is, roughly speaking, the set of all possibilities to assign each $x\in \partial D_R$ to a unique $(d-1)$-dimensional unit box  $\Gamma_x\in \mathcal{S}_R$ with $x\in \Gamma_x$.
For every $R\in\N$, $\Gamma\in \mathbf{A}_R$,
$x\in \partial D_R$ let $\psi_x^{(R,\Gamma)}\in  C(\Gamma_x,\R)$  be the linear combination of the functions
$\varphi_z^{(R)}$, $z\in \Gamma_x\cap\Z^d$, with the property that for all $z\in \Gamma_x\cap\Z^d$ it holds that
\begin{align}\label{c01}
\int_{\Gamma_x}\psi_{x}^{(R,\Gamma)}\varphi^{(R)}_{z}\,d\sigma=\1_{\{x\}}(z).\end{align}
For every $R\in\N$, $\Gamma\in \mathbf{A}_R$ let $\Pi^{(R,\Gamma)}\colon \mathrm{Lip}(\partial C_R)\to \R^{\partial D_R}$ be the operator which satisfies for all  
$f\in \mathrm{Lip}(\partial C_R)$,
$x\in \partial D_R$ that
\begin{align}
(\Pi^{(\Gamma,R)} f) (x)= \int_{\Gamma_x}\psi_{x}^{(\Gamma,R)} f\,d\sigma.\label{c03}
\end{align}
Let $B\in[0,\infty]$ be the real extended number given by
\begin{align}\label{b47}
B=
\sup\left\{\|\psi_{x}^{R,\Gamma}\|_{L^\infty(\Gamma_x)}\,\middle| R\in\N,\Gamma \in \mathbf{A}_R,x\in \partial D_R\right\}.\end{align}
 Then 
\begin{enumerate}[i)]
\item\label{b40a} it holds 
that $B<\infty$,
\item\label{b40b} it holds 
for all $R\in \N$, $\Gamma\in \mathbf{A}_R$  that $\Pi^{(R,\Gamma)} T_R=\mathrm{id}$, and
\item\label{b40c} there exists $C\colon [1,\infty]\to(0,\infty)$ such that for all
$p\in [1,\infty]$,
$R\in \N$, $\Gamma\in \mathbf{A}_R$, $f\in\mathrm{Lip}(\partial C_R)$
it holds that
\begin{align}
\norm{\Pi^{(R,\Gamma)}f}{p}{\partial D_R} \leq C(p) \normB{f}{p}{\partial  C_R}
\quad\text{and}\quad
\norm{\nablaBar (\Pi^{(R,\Gamma)}f)}{p}{\etan{R}}\leq C(p) \normB{\nabla^\mathrm{tan} f}{p}{\partial{C}_R}
.
\end{align}
\end{enumerate}
\end{lemma}
\tdplotsetmaincoords{70}{30}
\begin{figure}\centering
\begin{subfigure}{.45\textwidth}
\begin{tikzpicture}[scale=1.5,tdplot_main_coords]
\draw (-0.5,0,0)--(3.5,0,0);
\draw (-0.5,0,1)--(3.5,0,1);
\draw (-0.5,1,1)--(3.5,1,1);
\draw (-0.5,0,-1)--(3.5,0,-1);
\draw (0,0,-1.5)--(0,0,1)--(0,1.5,1);
\draw (1,0,-1.5)--(1,0,1)--(1,1.5,1);
\draw (2,0,-1.5)--(2,0,1)--(2,1.5,1);
\draw (3,0,-1.5)--(3,0,1)--(3,1.5,1);
\draw plot [mark=*, mark size=1] coordinates{(2,0,0)}; 
\draw plot [mark=*, mark size=1] coordinates{(1,0,0)}; 
%\node [below left ] at (0,0,0){$y$};
\node [below left ] at (0,0,1){$\Gamma_x$};
\node [below left ] at (1,0,0){$x$};
\node [below left ] at (2,0,0){$y$};
\node [below right ] at (2,0,1){$\Gamma_y$};

\fill[pattern=north east lines, pattern color=red] 
(0,0,0)--(1,0,0)--(1,0,1)--(0,0,1) --(0,0,0);		
\fill[pattern=north west lines, pattern color=blue] 
(1,0,0)--(2,0,0)--(2,0,1)--(1,0,1) --(1,0,0);	
\draw [line width=1.5](0,0,-1)--(0,0,1)--(3,0,1)--(3,0,-1)--(0,0,-1); 		 
\end{tikzpicture}\caption{}\label{b42}
\end{subfigure}
\begin{subfigure}{.45\textwidth}
\begin{tikzpicture}[scale=1.5,tdplot_main_coords]
\draw (-0.5,0,0)--(3.5,0,0);
\draw (-0.5,0,1)--(3.5,0,1);
\draw (-0.5,0,-1)--(3.5,0,-1);
\draw (-0.5,1,1)--(3.5,1,1);
\draw (0,0,-1.5)--(0,0,1)--(0,1.5,1);
\draw (1,0,-1.5)--(1,0,1)--(1,1.5,1);
\draw (2,0,-1.5)--(2,0,1)--(2,1.5,1);
\draw (3,0,-1.5)--(3,0,1)--(3,1.5,1);
\fill[pattern=north west lines, pattern color=red] 
(0,0,0)--(1,0,0)--(1,0,1)--(0,0,1) --(0,0,0);		
\fill[pattern=north west lines, pattern color=blue] 
(2,0,0)--(3,0,0)--(3,0,-1)--(2,0,-1);
\draw plot [mark=*, mark size=1] coordinates{(1,0,0)}; 
\node [above right] at (1,0,0) {$x$};
\node [below left] at (0,0,.7){$\Gamma_x$};
\node [above right] at (2.5,1,1){$\Gamma_y$};
\draw plot [mark=*, mark size=1] coordinates{(2,0,0)}; 
\node [above right] at (2,0,0) {$y$};
\draw [line width=1.5](0,0,-1)--(0,0,1)--(3,0,1)--(3,0,-1)--(0,0,-1); 
\end{tikzpicture}\caption{}\label{b44}
\end{subfigure}
\begin{subfigure}{.45\textwidth}
\begin{tikzpicture}[scale=1.5,tdplot_main_coords]
\draw (-0.5,0,0)--(3.5,0,0);
\draw (-0.5,0,1)--(3.5,0,1);
\draw (-0.5,0,-1)--(3.5,0,-1);
\draw (-0.5,1,1)--(3.5,1,1);
\draw (0,0,-1.5)--(0,0,1)--(0,1.5,1);
\draw (1,0,-1.5)--(1,0,1)--(1,1.5,1);
\draw (2,0,-1.5)--(2,0,1)--(2,1.5,1);
\draw (3,0,-1.5)--(3,0,1)--(3,1.5,1);
\fill[pattern=north west lines, pattern color=red] 
(0,0,0)--(1,0,0)--(1,0,1)--(0,0,1) --(0,0,0);		
\fill[pattern=north west lines, pattern color=blue] 
(1,0,0)--(2,0,0)--(2,0,-1)--(1,0,-1);
\draw plot [mark=*, mark size=1] coordinates{(1,0,1)}; 
\node [above ] at (1,0,1) {$x$};
\node [below left] at (0,0,.7){$\Gamma_x$};
\node [below] at (1.5,0,-1){$\Gamma_y$};
\draw plot [mark=*, mark size=1] coordinates{(1,0,0)}; 
\node [above right] at (1,0,0) {$y$};
\draw[line width=1.5](0,0,-1)--(0,0,1)--(2,0,1)--(2,0,-1)--(0,0,-1); 
\draw[line width=1.5](0,0,1)--(0,1,1)--(2,1,1)--(2,0,1);
\end{tikzpicture}\caption{}\label{b43}
\end{subfigure}
%%%%%%%%%%%%%%%%%%%%%%%
\begin{subfigure}{.45\textwidth}
\begin{tikzpicture}[scale=1.5,tdplot_main_coords]
\draw (-0.5,0,0)--(3.5,0,0);
\draw (-0.5,0,1)--(3.5,0,1);
\draw (-0.5,0,-1)--(3.5,0,-1);
\draw (-0.5,1,1)--(3.5,1,1);
\draw (0,0,-1.5)--(0,0,1)--(0,1.5,1);
\draw (1,0,-1.5)--(1,0,1)--(1,1.5,1);
\draw (2,0,-1.5)--(2,0,1)--(2,1.5,1);
\draw (3,0,-1.5)--(3,0,1)--(3,1.5,1);
\fill[pattern=north west lines, pattern color=red] 
(0,0,0)--(1,0,0)--(1,0,1)--(0,0,1) --(0,0,0);		
\fill[pattern=north west lines, pattern color=blue] 
(2,0,1)--(3,0,1)--(3,1,1)--(2,1,1);
\draw plot [mark=*, mark size=1] coordinates{(1,0,1)}; 
\node [above] at (1,0,1) {$x$};
\node [below left] at (0,0,.7){$\Gamma_x$};
\node [above right] at (2.5,1,1){$\Gamma_y$};
\draw plot [mark=*, mark size=1] coordinates{(2,0,1)}; 
\node [above] at (2,0,1) {$y$};
\draw [line width=1.5](0,0,0)--(0,0,1)--(3,0,1)--(3,0,0)--(0,0,0); 
\draw[line width=1.5](0,0,1)--(0,1,1)--(3,1,1)--(3,0,1);
\end{tikzpicture}\caption{}\label{b45}
\end{subfigure}
\caption{An illustrative construction of a good cover of the surface of a box}
\label{b46}
\end{figure}
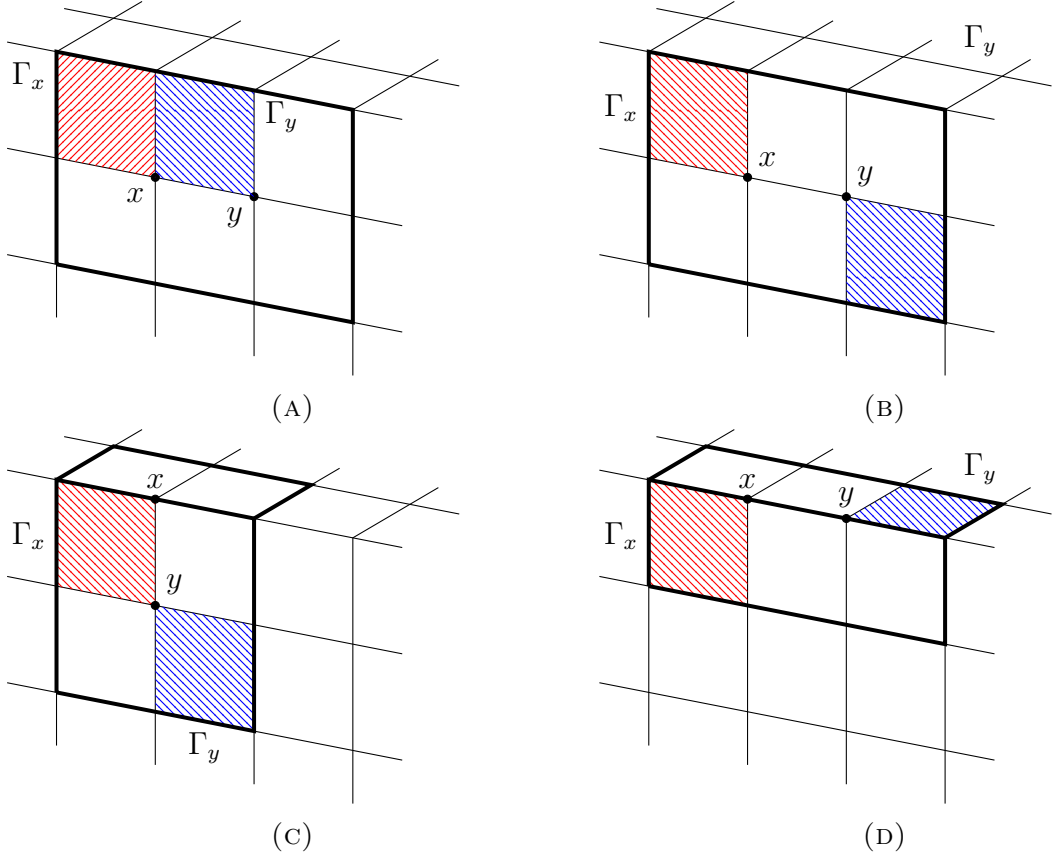
\begin{proof}[Proof of \cref{b39c}]
Since for every
$R\in\N$, $\Gamma \in \mathbf{A}_R$,
$x\in\partial D_R$
we can use an affine transformation to move $\Gamma_x$ to the reference box $[0,1]^{d-1}$, \cref{b40a} is obvious.
Next, Items \eqref{b20a} and \eqref{b20d} in \cref{b20} imply
 for all $R\in \N$, $\Gamma\in \mathbf{A}_R$, $f\colon \partial D_R\to\R$, $x\in\partial D_R$, $y\in\Gamma_x$ that
$(T_Rf)(y) = \sum_{z\in \Gamma_ x\cap\Z^d}f(z)\varphi_z^{(R)}(y)$. 
Hence, \eqref{c01} and \eqref{c03} show that for all 
$R\in \N$, $\Gamma\in \mathbf{A}_R$, $f\colon \partial D_R\to\R$, $x\in\partial D_R$ it holds that 
\begin{align}\begin{split}
(\Pi^{(\Gamma,R)} T_Rf) (x)&= 
\int_{\Gamma_x}\psi_{x}^{(\Gamma,R)} \left[\sum_{z\in \Gamma_ x\cap\Z^d}f(z)\varphi_z^{(R)}\right]\,d\sigma=\sum_{z\in \Gamma_ x\cap\Z^d}\left[f(z)
\int_{\Gamma_x}\psi_{x}^{(\Gamma,R)} \varphi_z^{(R)}\,d\sigma\right]
\\&=\sum_{z\in \Gamma_ x\cap\Z^d}f(z) \1_{\{x\}}(z)=f(x).
\end{split}\end{align}
This shows \cref{b40b}.
Next, observe that there exist $K\in \N $ and a collection of  $(d-1)$-dimensional manifolds with boundary $\{M_{x,y}^{(R)} \mid R\in\N, x,y\in \partial D_R,x\sim y\}$ such that 
for all $R\in \N$, $\Gamma\in\mathbf{A}_R$, $x,y\in \partial D_R$, $z\in \partial C_R$ 
\begin{enumerate}[i)]
\item it holds that
 $M_{x,y}^{(R)}$ is either a $(d-1)$-dimensional face or 
the union of two dimensional faces of a $d$-dimensional rectangle
whose edges are at most $3$
(informally speaking, 
$M_{x,y}^{(R)}$ is
either a rectangle or a ''roof'' of dimension $(d-1)$ with edges of size at most $3$), 
\item it holds that
 $\Gamma_x\cup\Gamma_y\subseteq M^{(R)}_{x,y}$,  and
\item it holds that
\begin{align}\label{m01}
\left|\left\{x\in \partial D_R\;\middle| z\in \Gamma_x \right\}\right|\leq K\quad\text{and}\quad \left|\left\{x,y\in\Z^d\;\middle| x\sim y, z\in M_{x,y}^{(R)}\right\}\right|\leq K.
\end{align}
\end{enumerate}
An illustrative argument for this (in the case $d=3$) is provided in \cref{b46}: 
\begin{inparaenum}[(a)]
\item 
If none of $x,y$ lies on an edge of the box $\overline{C}_R$, we choose $M^{(R)}_{x,y}$, e.g., as a $2\times 3$ rectangle, see \cref{b42,b44}; \item otherwise, we choose $M^{(R)}_{x,y}$ as a roof, see \cref{b43,b45}.
\end{inparaenum}
Note that $M_{x,y}^{(R)}$, $ R\in\N$, $x,y\in \partial D_R,x\sim y$, can only have a finite number of shapes. 
Hence,
Poincar\'e's inequality on Lipschitz manifolds implies that
there exists $\poincareConst\colon [1,\infty]\to (0,\infty) $ such that for all $R\in\N$, $x,y\in \partial D_R$, $f\in  L^p_1(M^{(R)}_{x,y})$, $p\in [1,\infty]$ it holds that $\inf_{a\in\R}\norm{f-a}{p}{M^{(R)}_{x,y}}\leq
\poincareConst(p) \normB{\nabla^{\mathrm{tan}}f }{p}{M^{(R)}_{x,y}} $.
Furthermore, \eqref{c03}, \eqref{b47}, and Jensen's inequality imply for all 
$R\in\N$, $\Gamma\in \mathbf{A}_R$,
$x\in \partial D_R$, $p\in [1,\infty)$, $f\in\mathrm{Lip}(\partial C_R)$ that 
$|(\Pi^{(R,\Gamma)} f)(x)|^p\leq B^p \normB{f}{p}{\Gamma_x}^p$ 
and 
$|(\Pi^{(R,\Gamma)} f)(x)|\leq B  \normB{f}{\infty}{\Gamma_x}$. 
Summing or taking the supremum (over $x\in \partial D_R$) and using \eqref{m01} yield that
for all 
$R\in\N$, $\Gamma\in \mathbf{A}_R$, $p\in[1,\infty] $, $f\in\mathrm{Lip}(\partial C_R)$ it holds that 
\begin{align}\label{c04}
\norm{\Pi^{(R,\Gamma)} f}{p}{\partial D_R}\leq BK^{1/p} \normB{f}{p}{\partial C_R}. 
\end{align}
Next, \cref{b40b} (with $f\defeq (\partial D_R\ni x\mapsto 1\in\R)$) implies 
for all $R\in\N $, $\Gamma\in\mathbf{A}_R$, $x\in\partial D_R$
that $\int_{\Gamma_x}\psi^{(R,\Gamma)}_x=1$. This, \eqref{c03}, the fact that $\forall\,p\in[1,\infty), a,b\in \R \colon |a+b|^p\leq 2^{p-1}(|a|^p+|b|^p) $, the triangle inequality, \eqref{b47}, and Poincar\'e's inequality  imply that for all 
 $R\in\N $, $\Gamma\in\mathbf{A}_R$, $x,y\in\partial D_R$, $p\in [1,\infty)$, $f\in\mathrm{Lip}(\partial C_R)$ it holds that
\begin{align}\begin{split}
&\left|(\Pi^{(R,\Gamma)} f)(x)-(\Pi^{(R,\Gamma)} f)(y)\right|^p=\inf_{a\in\R}\left |
\int_{\Gamma_x} \psi^{(R,\Gamma)}_x (f-a)d\sigma - \int_{\Gamma_y} \psi^{(R,\Gamma)}_y (f-a) d\sigma\right|^p\\
&\leq 2^{p-1} \inf_{a\in\R}\left[
\left |
\int_{\Gamma_x} \psi^{(R,\Gamma)}_x (f-a)d\sigma \right|^p+\left| \int_{\Gamma_y} \psi^{(R,\Gamma)}_y (f-a) d\sigma\right|^p\right]\\
&\leq 2^{p-1} B^p\left[\inf_{a\in\R}
\normB{f-a}{p}{M_{x,y}^{(R)}}^p\right]
\leq 2^{p-1}B ^p \poincareConst(p)^p\normB{\nabla^{\mathrm{tan}}f}{p}{M_{x,y}^{(R)}}^p
\end{split}\end{align}
and 
\begin{align}
\left|(\Pi^{(R,\Gamma)} f)(x)-(\Pi^{(R,\Gamma)} f)(y)\right|
\leq 2B\left[ \inf_{a\in \R}\normB{f-a}{\infty}{M_{x,y}^{(R)}}\right]
\leq 2B \poincareConst(\infty)\normB{\nabla^{\mathrm{tan}f}}{\infty}{M_{x,y}^{(R)}}
.
\end{align}
Summing or taking the supremum (over $x,y\in \partial D_R$ with $x\sim y$) and using \eqref{m01} yield  for all 
 $R\in\N $, $\Gamma\in\mathbf{A}_R$,  $p\in [1,\infty]$ that 
$
\norm{\nablaBar (\Pi^{R,\Gamma}f)}{p}{\etan{R}}\leq 
2^{(p-1)/p}K^{1/p}M \poincareConst(p) \normB{\nabla^\mathrm{tan} f}{p}{\partial{C}_R}.
$
This, \eqref{c04}, and the fact that $B<\infty$ (following from \cref{b40a}) imply \cref{b40c}. The proof of \cref{b39c} is thus completed.
\end{proof}

Now, for every bounded Lipschitz domain $\Omega$ and every $f\in\mathrm{Lip}(\partial \Omega)$
let 
$|f|_{L^p(\partial \Omega)}$ be the real number given by
$|f|_{L^p(\partial \Omega)}^p=\|f\|_{L^p(\partial \Omega)}/(\int_{\partial\Omega}d\sigma)$ and let 
$|f|_{L^\infty(\partial \Omega)}$ be the real number given by 
$|f|_{L^\infty(\partial \Omega)}=\|f\|_{L^\infty(\partial \Omega)}$.
Then \cref{b34} and a scaling argument show that
there exist 
$\epsilon_0\in (0,1)$, $C_1\colon  [1,\infty]\times[1,\infty]\to (0,\infty)$, $C_2\colon [1,\infty]\to (0,\infty)$
and there exist linear operators $\contSmooth{\epsilon}^{(R)}\colon \mathrm{Lip}(\partial C_R)\to \mathrm{Lip}(\partial C_R)$, $\epsilon\in (0,\epsilon_0)$, $R\in\N$,
such that  for all $R\in\N$, $u\in \mathrm{Lip}(\partial C_R)$, 
$r,s\in [1,\infty]$ with $1\leq  s\leq r\leq  \infty$,  $\epsilon\in (0,\epsilon_0)$
it holds that
\begin{align}
\label{b35a}
|\contSmooth{\epsilon}^{(R)}u|_{L^r(\partial C_R)}
&\leq C_1(r,s) \epsilon ^{-(d-1)(\frac1s-\frac1r)}|u|_{L^s(\partial C_R)},\\
\label{b37a}
|\nabla^\mathrm{tan} (\contSmooth{\epsilon}^{(R)}u)|_{L^r(\partial C_R)}
&\leq C_1(r,s) \epsilon^{-(d-1)(\frac1s-\frac1r)}|\nabla^\mathrm{tan}u|_{L^s(\partial C_R)},\\
\label{b36a}
|u-\contSmooth{\epsilon}^{(R)}u|_{L^s(\partial C_R)}&\leq C_2(s)\epsilon R |\nabla^\mathrm{tan}u|_{L^s(\partial C_R)}.
\end{align}
Using the notation given in \cref{b19b} 
let $\smooth{R}{\epsilon}^\Gamma\colon \R^{\partial D_R}\to\R^{\partial D_R}$, $\epsilon\in(0,\epsilon_0)$, $R\in\N$, $\Gamma\in \mathbf{A}_R$, be the linear operators which satisfy for all 
$\epsilon\in(0,\epsilon_0)$, $R\in\N$, $\Gamma\in \mathbf{A}_R$, $u\colon \partial D_R\to \R$ that
$\smooth{R}{\epsilon}^\Gamma u= \Pi^{\Gamma,R}\contSmooth{\epsilon}^{(R)} T_Ru$. The projection property (\cref{b40b} in \cref{b39c}) then shows  for all 
$\epsilon\in(0,\epsilon_0)$, $R\in\N$, $\Gamma\in \mathbf{A}_R$, $u\colon \partial D_R\to \R$ that
$u-\smooth{R}{\epsilon}^\Gamma u= 
\Pi^{\Gamma,R} (T_Ru-
 \contSmooth{\epsilon}^{(R)} T_Ru)$. Combining this, \cref{b20e} in \cref{b20}, and \cref{b40c} in \cref{b39c} we obtain \cref{a48a} below where $\Gamma$ is dropped. Note that although there are several ways to choose
$\Gamma$ and hence  several ways to define the family $(\smooth{R}{\epsilon})= (\smooth{R}{\epsilon}^\Gamma)$ of smoothing operators, the constants $C_1$ and $C_2$ in \eqref{a48}--\eqref{a50} only depend on $s,r$, and the dimension~$d$.
\begin{corollary}[Smoothing operators for discrete functions]\label{a48a}
There exist a real number $\epsilon_0\in (0,1)$, functions $C_1\colon [1,\infty]\times[1,\infty]\to (0,\infty)$, 
$C_2\colon [1,\infty]\to(0,\infty)$,
and   linear operators
$\smooth{R}{\epsilon}\colon  \R^{\partial D_R}\to  \R^{\partial D_R}$,
 $\epsilon\in (0,\epsilon_0)$, $R\in \N $, 
such that 
for all 
$s\in [1,\infty]$, $r\in [s,\infty]$, $R\in [1,\infty)\cap\N$, $u\colon \partial D_R\to\R$, $\epsilon\in (0,\epsilon_0)$ it holds that
\begin{align}
\avnorm{\smooth{R}{\epsilon}u}{r}{\partial D_R}
&\leq C_1(s,r)\epsilon ^{-(d-1)(\frac1s-\frac1r)}
\avnorm{u}{s}{\partial D_R},\label{a48} 
\\
\avnorm{\nablaBar (\smooth{R}{\epsilon}u) }{r}{\etan{R}}&\leq C_1(s,r)
\epsilon ^{-(d-1)(\frac1s-\frac1r)}
\avnorm{\nablaBar u}{s}{\etan{R}}\label{a49},
\\
\avnorm{u-\smooth{R}{\epsilon}u}{s}{\partial D_R}
&\leq C_2(s)\epsilon R \avnorm{\nablaBar u}{s}{\etan{R}}. \label{a50}\end{align}
\end{corollary}
\section{Estimates on the harmonic extension in details}\label{b47c}
\renewcommand{\smooth}[2]{\mathbf{S}}
\renewcommand{\dualsmooth}[2]{\mathbf{S}^*}
\renewcommand{\DirichletExt}[2]{\mathbf{D}}
\renewcommand{\NeumannExt}[2]{\mathbf{N}}
\renewcommand{\funcNu}[1]{{\mathbf{n}}}
%%%%%%%%%%%

In \cref{b47b} below $\smooth{R}{\epsilon}$ is a smoothing operator (cf. \cref{a48a}) which depends on $R,\epsilon$ and $\modi{R}$ is a modifying operator (cf. \cref{a04}) which depends on $R$. 
Here, we drop the dependency on $R,\epsilon$ to lighten the notation.
All inequalities on $\smooth{R}{\epsilon}$ in \eqref{d06c} are direct consequences of \eqref{a48}--\eqref{a50}.
Recall that we use \cref{b66b} and \cref{f01} and in the following 
we briefly write $\funcNu{R}=\mathbf{n}_R$ for the normal mapping.

Discrete boundary problems have been widely studied in numerical analysis, e.g., to approximate the continuum solutions (see, e.g., the classical work by Stummel~\cite{Stu67} and
G\"urlebeck and Hommel \cite{Hom98}, \cite{GH01}, \cite{GH02},
who studied Dirichlet and Neumann boundary problems 
on general two-dimensional discretized domains
using difference potentials,
and the references therein). We therefore do not discuss the definitions of $\DirichletExt{R}{\epsilon}$ and $\NeumannExt{R}{\epsilon}$ in details. The condition $\sum_{y\in\tilde \partial D_R}u(y)-(\NeumannExt{R}{\epsilon}u)(y)=0$ is reasonable since 
uniqueness of solutions to Neumann problems only holds up to constants and
we can therefore adapt the mean of a solution by adding a constant. 

Note that 
\eqref{y22} is a consequence of Sobolev's inequality in the case $d\ge3$.
In the case $d=2$, i.e., when $\partial D_R$ is one-dimensional, this inequality is just a direct consequence of the triangle inequality, i.e., 
$\avnorm{u}{\infty}{\partial D_R}\lesssim R\avnorm{\nablaBar u}{1}{ \etan{R}}$.

Finally, since for the main results there is no need to provide more detailed estimates, we solely use  a real number $c\in(1,\infty)$  for several purposes
(see e.g. \eqref{d06c}, \eqref{y01c}, and \eqref{y22a}). 

\begin{setting}\label{b47b}
Let $\epsilon\in (0,1)$, $c\in (1,\infty)$, 
$p,q\in (1,\infty]$,
$R\in\N$, 
$\omega\in\Omega$,
let $\matAhom\in \R^{d\times d}$ be a diagonal matrix
which satisfies for all $i\in[1,d]\cap\N$ that $c^{-1}\leq (\matAhom)_{ii}\leq c$, 
let $\omegaH\in \Omega$ satisfy  for all $ i\in [1,d]\cap\N, x\in\Z^d$ that $ \omegaH(\{x,x+\unit{i}\}) =(\matAhom)_{ii}$,
and assume that
$
|\partial D_R|/ |\tilde{\partial} D_R|\leq c$ and $|\tilde{\partial} D_R|/|E_R|\leq c/R$. 
Let
$\smooth{R}{\epsilon}\colon \R^{\partial D_R}\to \R^{\partial D_R}$, $ 
\modi{R}\colon  \R^{ \partial D_R}\to \R^{ \partial D_R}$
be linear operators
which satisfy for all $u\colon \partial D_R\to\R$ that
\begin{align}\begin{split}
&\avnorm{u-\smooth{R}{\epsilon}u}{\frac{2p}{p-1}}{\partial D_R}
\leq c\epsilon R \avnorm{\nablaBar u}{\frac{2p}{p-1}}{\etan{R}}
,\quad
\avnorm{u-\smooth{R}{\epsilon}u}{\frac{2q}{q-1}}{\partial D_R}
\leq c\epsilon R \avnorm{\nablaBar u}{\frac{2q}{q-1}}{\etan{R}} , \\
&
\avnorm{\smooth{R}{\epsilon}u}{\frac{2p}{p-1}}{\partial D_R}
\leq c
\avnorm{u}{\frac{2p}{p-1}}{\partial D_R} ,\quad
  \avnorm{\smooth{R}{\epsilon}u}{\frac{2p}{p-1}}{\partial D_R}
\leq c
\avnorm{u}{\frac{2p}{p-1}}{\partial D_R} , 
\\
&
\avnorm{\smooth{R}{\epsilon}u}{\frac{2p}{p-1}}{\partial D_R}\leq \epsilon^{-(d-1)\frac{p+1}{2p}}
\avnorm{u}{1}{\partial D_R}
 , \quad
\avnorm{\smooth{R}{\epsilon}u}{\infty}{\partial D_R}\leq \epsilon^{-(d-1)\frac{q+1}{2q}}
\avnorm{u}{\frac{2q}{q+1}}{\partial D_R}
 , \\
& 
\avnorm{\nablaBar (\smooth{R}{\epsilon}u) }{\frac{2q}{q+1}}{\etan{R}}\leq c
\avnorm{\nablaBar u}{\frac{2q}{q+1}}{\etan{R}}
 ,\quad
 \avnorm{\nablaBar (\modi{R}u)}{\frac{2q}{q+1}}{\etan{R}}\leq 
\avnorm{\nablaBar u}{\frac{2q}{q+1}}{\etan{R}}
 ,\\&
\avnorm{\nablaBar(\modi{R}u)}{\frac{2p}{p+1}}{\etan{R}}\leq c \avnorm{\nablaBar u}{\frac{2p}{p+1}}{\etan{R}} ,\quad
  \avnorm{\modi{R}u}{\frac{2p}{p-1}}{\partial D_R}\leq c\avnorm{u}{\frac{2p}{p-1}}{\tilde{\partial} D_R} ,\quad
 \avnorm{\modi{R}u}{1}{\partial D_{R}}\leq c \avnorm{u}{1}{\tilde \partial D_{R}} \end{split}\label{d06c}
\end{align}
and let $\dualsmooth{R}{\epsilon}\colon \R^{\tilde \partial D_R}\to 
 \R^{\tilde \partial D_R}$
be the operator which satisfies for all
 $h\colon \tilde\partial D_R\to\R$, $x\in \tilde{\partial}D_R$  that
\begin{align}\label{d06a}
 (\dualsmooth{R}{\epsilon}h)(x) =
\sum_{y\in \tilde \partial D_{R}} h(y) \left(\smooth{R}{\epsilon}\modi{R} \left(\1_{\{x\}}\upharpoonright_{\partial D_R}\right) \right)(y).
\end{align}
Let $\DirichletExt{R}{\epsilon}\colon \R^{\overline   D_R}\to \R^{\overline   D_R}$ be the solution operator of the discrete Dirichlet problem in the sense that for all  $u \colon \overline {D}_R\to \R $ it holds that
\begin{align}
\forall x\in D_R\colon \quad (\nabla^*\cdot \matAhom\nabla(\DirichletExt{R}{\epsilon}u))(x)=0
\quad\text{and}
\quad 
\forall x\in \partial D_R\colon \quad  (\DirichletExt{R}{\epsilon}u)(x)=(\smooth{R}{\epsilon}u )(x),
\label{y01b}
\end{align}
let  $ \NeumannExt{R}{\epsilon}\colon \R^{\overline{D}_R}   \to \R^{\overline   D_R}$ be
a solution operator of the discrete Neumann problem in the sense that for all $u \colon \overline {D}_R\to \R $  it holds that 
\begin{align} \begin{split}
&\forall x\in \tilde \partial D_R\colon\quad  \left(\omegaH\nablaBar (\NeumannExt{R}{\epsilon}u)\right) _{\funcNu{R}(x)}= \left[\dualsmooth{R}{\epsilon}\left(\omega_\funcNu{R} \nabla_\funcNu{R} u \right)\right](x)-\frac{1}{|\tilde\partial D_R|}\left[
\sum_{y\in \tilde \partial D_R}\left[\dualsmooth{R}{\epsilon}\left(\omega_\funcNu{R} \nabla_\funcNu{R} u\right)\right](y)\right],\\
&\forall x\in D_R\colon\quad  (\nabla^*\cdot \matAhom\nabla(\NeumannExt{R}{\epsilon}u))(x)=0,\quad\text{and}\quad\left[ \sum_{y\in\tilde \partial D_R}u(y)-(\NeumannExt{R}{\epsilon}u)(y)\right]=0,
\end{split}\label{y01}
\end{align}
with $\funcNu{R}$ the normal mapping,
 assume  
for all $u \colon \overline {D}_R\to \R $,  $s\in\{2p/(p+1),2q/(q+1)\}$ that
\begin{align}
\avnorm{\nablaBar(\NeumannExt{R}{\epsilon}u)}{s}{\etan{R}}\leq c\avnorm{\nablaBar(\NeumannExt{R}{\epsilon}u)}{s}{\enor{R}}\quad\text{and}\quad
\avnorm{\nablaBar(\DirichletExt{R}{\epsilon}u)}{s}{\enor{R}}\leq c\avnorm{\nablaBar(\DirichletExt{R}{\epsilon}u)}{s}{\etan{R}},
\label{y01c}
\end{align}
and  assume for all $u \colon \overline {D}_R\to \R $ that
\begin{align}\begin{split}
\avnorm{\nablaBar (\DirichletExt{R}{\epsilon}u)}{\max\left\{\frac{4p}{p-1},
\frac{ 4q}{q-1}\right\}}{E_R}
&\leq c\avnorm{\nablaBar (\DirichletExt{R}{\epsilon}u)}{\max\left\{\frac{4p}{p-1},
\frac{ 4q}{q-1}\right\}}{\etan{R}},
\\
\avnorm{\nablaBar (\NeumannExt{R}{\epsilon}u)}{\max\left\{\frac{4p}{p-1},
\frac{ 4q}{q-1}\right\}}{E_R}
&\leq c \avnorm{\nablaBar (\NeumannExt{R}{\epsilon}u)}{\max\left\{\frac{4p}{p-1},
\frac{ 4q}{q-1}\right\}}{\enor{R}}.
\label{y01d}
\end{split}\end{align}
Let $\alpha\colon (1,\infty)\to\R$ be the function which satisfies for all 
$r\in \{p,q\}$ that
\begin{align}
\frac{1}{\alpha(r)}=\left(\frac{r-1}{2r}+\frac{1}{d-1}\right)\1_{d\geq 3}+\1_{d=2},\label{y20}
\end{align}
assume Sobolev's inequality in the sense that for all $r\in \{p,q\}$, $u\colon \partial D_R\to\R$ it holds that
\begin{align}
\inf_{a\in \R}\avnorm{u-a}{\frac{2r}{r-1}}{\partial D_R}\leq c R
\avnorm{\nablaBar u}{\alpha(r)}{\etan{R}},\label{y22a}
\end{align}
and
let
$\theta\colon \{p,q\}\times\{p,q\}\to\R$ be the function which satisfies
for all $r,s\in \{p,q\}$ that 
\begin{align}
\theta(r,s)= \left[1-(d-1)\left(\frac{1}{2r}+\frac{1}{2s}\right)\right]\1_{d>2}+
\frac{(s-1)r}{s(r+1)}\1_{d=2}.\label{y22}
\end{align}
For every $u\colon \overline{D}_R\to \R$ let 
$\overline{\Lambda}(u)$ be the real number given by
\begin{align}
\overline{\Lambda}(u)=
\max\left\{
\avnorm{\omega\nablaBar u}{\frac{2p}{p+1}}{\etan{R}},
\avnorm{\omega\nablaBar u}{\frac{2p}{p+1}}{\enor{R}},
\avnorm{\nablaBar u}{\frac{2q}{q+1}}{\etan{R}},
\avnorm{\nablaBar u}{\frac{2q}{q+1}}{\enor{R}}
\right\}.\label{y05}
\end{align}
 For every
 $u,v\colon \overline D_R\to\R$ let $\bterm{\omega}{R}{u}{v} $ be the real number given by
\begin{align}
\bterm{\omega}{R}{u}{v}=
\frac{1}{|\tilde\partial D_R|}\left[\sum_{x\in \tilde\partial D_R} \left(u(x)-v(x)\right)
\left(\omega\nablaBar u-\omegaH \nablaBar v\right)_{\funcNu{R}(x)} \right].\label{b49}
\end{align}
\end{setting}

\subsection{Duality}
\begin{lemma}\label{d06}Assume \cref{b47b} and let
 $h\colon \tilde{\partial} D_R\to\R$, $g\colon \partial D_R\to\R$. Then it holds that
\begin{align}
\sum_{x\in \tilde \partial D_{R}} (\dualsmooth{R}{\epsilon}h)(x) g(x)=
\sum_{x\in \tilde \partial D_{R}} h(x) \left(\smooth{R}{\epsilon}\modi{R}g\right)(x).
\end{align}
\end{lemma}
\begin{proof}[Proof of \cref{d06}]
The assumption that  $\forall\,u\in \R^{\partial D_R}\colon \avnorm{\modi{R}u}{1}{\partial D_{R}}\leq c \avnorm{u}{1}{\tilde \partial D_{R}}$ in \eqref{d06c} and the assumption on linearity of  $\modi{R}$  show  for all $u,v\colon \partial D_R\to\R $ with $\forall\,x\in\tilde{\partial} D_R\colon u(x)=v(x)$  that $\modi{R}u=\modi{R}v$.
The fact that $g$ and $ \sum_{x\in \tilde \partial D_{R}}g(x){(\1_{\{x\}}\!\!\upharpoonright_{\partial D_R})} $ coincide on $\tilde{\partial}D_R$  and 
\eqref{d06a} hence imply that
\begin{align}\begin{split}
&\left[\sum_{x\in \tilde \partial D_{R}}g(x) (\dualsmooth{R}{\epsilon}h)(x) \right]=
\sum_{x\in \tilde \partial D_{R}}\left[g(x)
\left[\sum_{y\in \tilde \partial D_{R}} h(y) \left(\smooth{R}{\epsilon}\modi{R} \left(\1_{\{x\}}\upharpoonright_{\partial D_R}\right) \right)(y)\right]\right]\\
&=
\sum_{y\in \tilde \partial D_{R}} h(y) \left(\smooth{R}{\epsilon}\modi{R} \left(\sum_{x\in \tilde \partial D_{R}}g(x)(\1_{\{x\}}\upharpoonright_{\partial D_R})\right) \right)(y)=
\sum_{y\in \tilde \partial D_{R}} h(y) \left(\smooth{R}{\epsilon}\modi{R} g \right)(y).
\end{split}\end{align}
This completes the proof of \cref{d06}.
\end{proof}

\begin{lemma}\label{y27}Assume \cref{b47b} and 
let $h\colon \tilde{\partial} D_R\to\R $. Then 
$\avnorm{\dualsmooth{R}{\epsilon}h}{\frac{2p}{p+1}}{\tilde{\partial}D_R}\leq c^3 \avnorm{h}{\frac{2p}{p+1}}{\tilde{\partial}D_R}$.
\end{lemma}
\begin{proof}[Proof of \cref{y27}]
Throughout this proof let $E\colon \R^{\tilde{\partial} D_R}\to \R^{\partial D_R}$ be the trivial extension operator which satisfies for all $g\colon \tilde{\partial}D_R\to\R$, $x\in\tilde{\partial}D_R$ that 
$(Eg)(x)=g(x) $ and which satisfies for all $g\colon \tilde{\partial}D_R\to\R$, $x\in\partial D_R\setminus \tilde{\partial}D_R$ that $(Eg)(x)=0$. First, observe that
the fact that
$\tilde\partial D_R\subseteq \partial D_R$ and the assumption that $|\partial D_R|/ |\tilde{\partial} D_R|\leq c$ prove for all $f\colon \partial D_R\to\R$, $r\in[1,\infty]$ that $\avnorm{f}{r}{\tilde{\partial} D_R}\leq c\avnorm{f}{r}{\partial D_R}$.
This,
\cref{d06} (with $g\defeq E g$ for $g\colon \tilde{\partial}D_R\to\R$),  H\"older's inequality, and the assumption that  for all $u\in {\partial D_R}\to \R$ it holds that $ \avnorm{\smooth{R}{\epsilon}u}{\frac{2p}{p-1}}{\partial D_R}\leq c\avnorm{u}{\frac{2p}{p-1}}{\partial D_R}$ and 
$ \avnorm{\modi{R}u}{\frac{2p}{p-1}}{\partial D_R}\leq c\avnorm{u}{\frac{2p}{p-1}}{\tilde{\partial} D_R}$ demonstrate 
 for all $g\colon \tilde{\partial} D_R\to\R$  that
\begin{align}
\begin{aligned}
&\frac{1}{|\tilde{\partial}D_R|}
\left|
\sum_{x\in \tilde \partial D_{R}} (\dualsmooth{R}{\epsilon}h)(x) g(x)\right|=
\frac{1}{|\tilde{\partial}D_R|}
\left|
\sum_{x\in \tilde \partial D_{R}} (\dualsmooth{R}{\epsilon}h)(x)
( Eg)(x)\right|\\
&
=\frac{1}{|\tilde{\partial}D_R|}\left|
\sum_{x\in \tilde \partial D_{R}} h(x) 
\left(\smooth{R}{\epsilon} \modi{R}  E g\right)(x)\right|
\leq \avnorm{h}{\frac{2p}{p+1}}{\tilde\partial D_R}
\avnorm{(\smooth{R}{\epsilon} \modi{R}  E) g}{\frac{2p}{p-1}}{\tilde\partial D_R}\\&\leq \avnorm{h}{\frac{2p}{p+1}}{\tilde\partial D_R}c
\avnorm{\smooth{R}{\epsilon} \modi{R}  E g}{\frac{2p}{p-1}}{\partial D_R}
\leq  
\avnorm{h}{\frac{2p}{p+1}}{\tilde\partial D_R}c^2
\avnorm{\modi{R} Eg}{\frac{2p}{p-1}}{\partial D_R}
\leq  
\avnorm{h}{\frac{2p}{p+1}}{\tilde\partial D_R}c^3
\avnorm{g}{\frac{2p}{p-1}}{\tilde{\partial} D_R}.
\end{aligned}\label{d13}
\end{align}
This
and a simple duality argument complete the proof of \cref{y27}.
\end{proof}
\begin{lemma}\label{d09}Assume \cref{b47b} and let  $h\colon \tilde\partial D_R\to\R$. Then it holds that
\begin{align}
\avnorm{\dualsmooth{R}{\epsilon}h}{\infty}{\tilde{\partial}D_R}\leq
c^2\epsilon^{-(d-1)\frac{p+1}{2p}}\avnorm{h}{\frac{2p}{p+1}}{\tilde{\partial}D_R}.\end{align}
\end{lemma}
\begin{proof}[Proof of \cref{d09}]
Observe that
\eqref{d06a} and the assumption that  
\begin{align}
\forall\,u \in \R^{ \partial D_R}\colon\quad \avnorm{\smooth{R}{\epsilon}u}{\frac{2p}{p-1}}{\partial D_R}\leq c\epsilon^{-(d-1)\frac{p+1}{2p}}\avnorm{u}{1}{\partial D_R}
\quad\text{and}\quad
\avnorm{\modi{R}u}{1}{\partial D_R}\leq c \avnorm{u}{1}{\tilde{\partial} D_R} 
\end{align}
in \eqref{d06c}
 establish for all $x\in \tilde{\partial}D_R$  that
\begin{align}\begin{split}
& (\dualsmooth{R}{\epsilon}h)(x) =|\tilde{\partial}D_R|\left[\frac{1}{|\tilde{\partial}D_R|}
\sum_{y\in \tilde \partial D_{R}} h(y) \left[\smooth{R}{\epsilon}\modi{R} \left(\1_{\{x\}}\upharpoonright_{\partial D_R}\right) \right](y)\right]\\
&\leq 
|\tilde{\partial}D_R|\avnorm{h}{\frac{2p}{p+1}}{\tilde{\partial}D_R}
\avnorm{\smooth{R}{\epsilon}\modi{R} \left(\1_{\{x\}}\upharpoonright_{\partial D_R}\right) }{\frac{2p}{p-1}}{\partial D_R}\\
&\leq 
|\tilde{\partial}D_R|\avnorm{h}{\frac{2p}{p+1}}{\tilde{\partial}D_R}
c\epsilon^{-(d-1)\frac{p+1}{2p}}
\avnorm{\modi{R} \left(\1_{\{x\}}\upharpoonright_{\partial D_R}\right) }{1}{\partial D_R}\\
&\leq 
|\tilde{\partial}D_R|\avnorm{h}{\frac{2p}{p+1}}{\tilde{\partial}D_R}
c^2\epsilon^{-(d-1)\frac{p+1}{2p}}
\avnorm{\1_{\{x\}} }{1}{\tilde{\partial} D_R}
=c^2\epsilon^{-(d-1)\frac{p+1}{2p}}\avnorm{h}{\frac{2p}{p+1}}{\tilde{\partial}D_R}.
\\
\end{split}
\end{align}
This completes the proof of \cref{d09}.
\end{proof}
\begin{lemma}\label{y03}Assume \cref{b47b} and let
 $u\colon \overline D_R\to\R$. Then it holds that
\begin{align}
\begin{split}
&\bterm{\omega}{R}{u}{\NeumannExt{R}{\epsilon}u}\\
&=
\frac{1}{|\tilde\partial D_R|}\left[
\sum_{x\in \tilde\partial D_R} \left[\left(\modi{R}\left((u-\NeumannExt{R}{\epsilon}u)\upharpoonright_{\partial D_R}\right)\right)(x)
-\left(\smooth{R}{\epsilon} \modi{R}\left(u-\NeumannExt{R}{\epsilon}u\upharpoonright_{\partial D_R}\right)\right)(x)
 \right](\omega \nablaBar  u)_{\funcNu{R}(x)}\right].
\end{split}
\end{align}
\end{lemma}
\begin{proof}[Proof of \cref{y03}]
To lighten the notation let $m\in\R$ be the real number given by
\begin{align}m=
\frac{1}{|\tilde\partial D_R|}
\sum_{y\in \tilde \partial D_R}\left(\dualsmooth{R}{\epsilon}\left(\omega_\funcNu{R}  \nabla_\funcNu{R} u\right)\right)(y).
\end{align}
First, \cref{d06} (with $h\defeq(\omega \nablaBar u)_\funcNu{R}$ and $g\defeq (u-\NeumannExt{R}{\epsilon}u)\upharpoonright_{\partial D_R}$)  shows that
\begin{align} \begin{split}
\sum_{x\in \tilde\partial D_R} \left[u(x)-(\NeumannExt{R}{\epsilon}u)(x)\right]
\left[\dualsmooth{R}{\epsilon}\left(\omega_{\funcNu{R}} \nablaBar_{\funcNu{R}}  u\right)\right](x)
=
\sum_{x\in \tilde\partial D_R}
\left[\smooth{R}{\epsilon} \modi{R}\left((u-\NeumannExt{R}{\epsilon}u)\upharpoonright_{\partial D_R}\right)\right](x)
(\omega \nablaBar  u)_{\funcNu{R}(x)}.\label{y02}\end{split}
\end{align}
Next, observe that \eqref{y01} implies  for all  $x\in \tilde \partial D_R$ that $
(\omegaH\nablaBar (\NeumannExt{R}{\epsilon}u))_{\funcNu{R}(x)}= \left(\dualsmooth{R}{\epsilon}\left(\omega_{\funcNu{R}} \nablaBar_{\funcNu{R}}  u\right)\right)(x)-m$ and $  \sum_{y\in\tilde \partial D_R}\left[u(y)-(\NeumannExt{R}{\epsilon}u)(y)\right]=0
$.
This, \eqref{b49}, and \eqref{y02} imply that
\begin{align}
\begin{split}
&\bterm{\omega}{R}{u}{v}\\&=\frac{1}{|\tilde\partial D_R|}
\sum_{x\in \tilde\partial D_R} (u(x)-(\NeumannExt{R}{\epsilon}u)(x))
\left(\omega\nablaBar u-\omegaH \nablaBar \left(\NeumannExt{R}{\epsilon}u\right)\right)_{\funcNu{R}(x)}\\
&=
\frac{1}{|\tilde\partial D_R|}
\sum_{x\in \tilde\partial D_R} [u(x)-\left(\NeumannExt{R}{\epsilon}u\right)(x)]
\left[(\omega \nablaBar  u)_{\funcNu{R}(x)}- \left( \left(\dualsmooth{R}{\epsilon}\left(\omega_{\funcNu{R}} \nablaBar_{\funcNu{R}}  u\right)\right)(x)-m\right)  \right]\\
&=
\frac{1}{|\tilde\partial D_R|}
\sum_{x\in \tilde\partial D_R} (u(x)-\left(\NeumannExt{R}{\epsilon}u\right)(x))
\left[(\omega \nablaBar  u)_{\funcNu{R}(x)}-  \left(\dualsmooth{R}{\epsilon}\left(\omega_{\funcNu{R}} \nablaBar_{\funcNu{R}}  u\right)\right)(x)\right] \\
&=
\frac{1}{|\tilde\partial D_R|}
\sum_{x\in \tilde\partial D_R} \left[\left(\modi{R}\left((u-\NeumannExt{R}{\epsilon}u)\upharpoonright_{\partial D_R}\right)\right)(x)
-\left(\smooth{R}{\epsilon} \modi{R}\left((u-\NeumannExt{R}{\epsilon}u)\upharpoonright_{\partial D_R}\right)\right)(x)
 \right](\omega \nablaBar  u)_{\funcNu{R}(x)}.
\end{split}
\end{align}
This completes the proof of \cref{y03}.
\end{proof}
\subsection{Boundary estimates}
\begin{lemma}[Reducing high exponents]\label{x33}Assume \cref{b47b}. Then
\begin{enumerate}[(i)]
\item \label{y21} it holds for all $r,s\in \{p,q\}$ that
$
 \theta(r,s)\frac{r-1}{2r}+ (1-\theta(r,s))\frac{1}{\alpha(r)}=\frac{s+1}{2s}
$ and \item 
it holds
for all $u\colon \partial D_R\to\R$, $r,s\in \{p,q\}$ with $\theta(r,s)\in (0,1)$  that 
\begin{align}
\avnorm{u-\smooth{R}{\epsilon}u}{\frac{2r}{r-1}}{\partial D_{R}} 
\leq 2c^2
  \epsilon ^{\theta(r,s)}
 R\avnorm{\nabla u}{\frac{2s}{s+1}}{\etan{R}}.
\end{align}
\end{enumerate}
\end{lemma}
\begin{proof}[Proof of \cref{x33}]First, \eqref{y22} and an easy calculation imply \cref{y21}. Next, 
\eqref{d06c} implies that $\smooth{R}{\epsilon}(\1_{\partial D_R})=\1_{\partial D_R}$. This, 
the triangle inequality, the assumption that 
\begin{align}\begin{split}
\forall \,u\in  \R^{\partial D_R},& r\in \{p,q\}
\colon\\
&\avnorm{u-\smooth{R}{\epsilon}u}{\frac{2r}{r-1}}{\partial D_R}
\leq c\epsilon R \avnorm{\nablaBar u}{\frac{2r}{r-1}}{\etan{R}} 
\quad\text{and}\quad
 \avnorm{\smooth{R}{\epsilon}u}{\frac{2r}{r-1}}{\partial D_R}
\leq c
\avnorm{u}{\frac{2r}{r-1}}{\partial D_R}
\end{split}\end{align}
in \eqref{d06c},
 the Sobolev inequality in
\eqref{y22a}, and the assumption $c\geq 1$
 prove that for all $u\colon \partial D_R\to\R$, $r,s\in \{p,q\}$ it holds that
\begin{align}\begin{split}
 &\avnorm{u-\smooth{R}{\epsilon}u}{\frac{2r}{r-1}}{\partial D_R}
=\inf_{a\in\R}
 \avnorm{(u-a)-\smooth{R}{\epsilon}(u-a)}{\frac{2r}{r-1}}{\partial D_R}
\\&\leq \inf_{a\in\R}
\left[
\avnorm{u-a}{\frac{2r}{r-1}}{\partial D_R}
+\avnorm{\smooth{R}{\epsilon}(u-a)}{\frac{2r}{r-1}}{\partial D_R}\right]
\leq 
2c\left[\inf_{a\in\R}\avnorm{u-a}{\frac{2r}{r-1}}{\partial D_R}\right]\leq 2c^2R
\avnorm{\nablaBar u}{\alpha(r)}{\etan{R}}
\end{split}\end{align}
and 
\begin{align}
 \avnorm{u-\smooth{R}{\epsilon}u}{\frac{2r}{r-1}}{\partial D_R}\leq 
2c^2\epsilon R\avnorm{\nablaBar u}{\frac{2r}{r-1}}{\etan{R}}.
\end{align}
This, an interpolation argument, and \cref{y21} 
% yields that there exists $c\in (0,\infty)$ such that
%for all $R\in [1,\infty)\cap\N$, $u\colon \partial D_R\to\R$, $\epsilon\in (0,\epsilon_0)$ it holds that
%$
% \avnorm{u-\smooth{R}{\epsilon}u}{\frac{2p}{p-1}}{\partial D_R}\leq 
%c(p,q)
%\epsilon^{\theta(p,q)}R \avnorm{\nablaBar u}{\frac{2q}{q+1}}{\etan{R}}.
%$
complete the proof of \cref{x33}.
\end{proof}

\begin{lemma}[Boundary term by the Dirichlet extension]
\label{y11}
Assume \cref{b47b}, let
 $u\colon \overline D_R\to\R $, 
 and assume that $\{\theta(q,q),\theta(p,q)\}\subseteq (0,1)$. Then 
\begin{align}
\left|\bterm{\omega}{R}{u}{\DirichletExt{R}{\epsilon}u}\right|
\leq 4c^6R\max\left\{\epsilon^{\theta(q,q)},\epsilon^{\theta(p,q)}\right\}\overline{\Lambda}(\omega,R,u).
\end{align}
\end{lemma}
\begin{proof}[Proof of \cref{y11}]
First, the boundary regularity in \eqref{y01c} and  \eqref{y01b}  show that
\begin{align}
\avnorm{\nablaBar \left( \DirichletExt{R}{\epsilon}u\right)}{\frac{2q}{q+1}}{\enor{R}}
\leq c\avnorm{\nablaBar \left( \DirichletExt{R}{\epsilon}u\right)}{\frac{2q}{q+1}}{\etan{R}}
= c\avnorm{\nablaBar \left( \smooth{R}{\epsilon}u\right)}{\frac{2q}{q+1}}{\etan{R}}
\leq c^2\avnorm{\nablaBar u}{\frac{2q}{q+1}}{\etan{R}}.\label{d10}
\end{align}
Next, the assumption that $|\partial D_R|/ |\tilde{\partial} D_R|\leq c$ and the assumption that $c\geq 1$ show  for all $f\colon \partial D_R\to\R$, $s\in [1,\infty]$ that $\avnorm{f}{s}{\tilde{\partial} D_R}\leq c\avnorm{f}{s}{\partial D_R}$.
Hence, \eqref{b49} (combined with the triangle inequality and 
H\"older's inequality), \cref{x33} (combined with the assumption that $\{\theta(q,q),\theta(p,q)\}\subseteq (0,1)$ and the fact that
$\forall\, x\in \partial D_R\colon   (\DirichletExt{R}{\epsilon}u)(x)=(\smooth{R}{\epsilon}u )(x)$ in
 \eqref{y01b}), the assumption that $c^{-1}\leq \omegaH\leq c$,
\eqref{d10}, and \eqref{y05}  imply that
\begin{align}
\begin{split}
&\left|\bterm{\omega}{R}{u}{\DirichletExt{R}{\epsilon}u}\right|=\frac{1}{\left|\tilde\partial D_R\right|}\left|\sum_{x\in \tilde\partial D_R} \left[\big(u(x)-\left(\DirichletExt{R}{\epsilon}u\right)(x)\big)
\big(\omega\nablaBar u-\omegaH \nablaBar \left(\DirichletExt{R}{\epsilon}u\right)\big)\left(\funcNu{R}(x)\right)\Big.\! \right]\right|\\
&\leq \avnorm{u-\DirichletExt{R}{\epsilon}u}{\frac{2p}{p-1}}{\tilde{\partial} D_R}
\avnorm{(\omega \nablaBar  u)_{\funcNu{R}}}{\frac{2p}{p+1}}{\tilde{\partial}D_R}
+
 \avnorm{u-\DirichletExt{R}{\epsilon}u}{\frac{2q}{q-1}}{\tilde{\partial} D_R}
\avnorm{\omegaH\circ\funcNu{R}}{\infty}{\tilde{\partial}D_R}
\avnorm{\nablaBar_{\funcNu{R}}( \DirichletExt{R}{\epsilon}u)}{\frac{2q}{q+1}}{\tilde{\partial}D_R}
\\
&\leq
 c\avnorm{u-\DirichletExt{R}{\epsilon}u}{\frac{2p}{p-1}}{\partial D_R}
\avnorm{\omega \nablaBar u}{\frac{2p}{p+1}}{\enor{R}}
+c
 \avnorm{u-\DirichletExt{R}{\epsilon}u}{\frac{2q}{q-1}}{\partial D_R}
c
\avnorm{\nablaBar( \DirichletExt{R}{\epsilon}u)}{\frac{2q}{q+1}}{\enor{R}}
\\
&\leq 2c^3R\epsilon^{\theta(p,q)}\avnorm{\nablaBar u}{\frac{2q}{q+1}}{\partial D_R}
\avnorm{\omega \nablaBar u}{\frac{2p}{p+1}}{\enor{R}}
+
2c^3R\epsilon^{\theta(q,q)}\avnorm{\nablaBar u}{\frac{2q}{q+1}}{\etan{R}}c^3
\avnorm{\nablaBar u}{\frac{2q}{q+1}}{\etan{R}}\\
&\leq 4c^6R\max\{\epsilon^{\theta(q,q)},\epsilon^{\theta(p,q)}\}\overline{\Lambda}(u).
\end{split}
\end{align}
This completes the proof of \cref{y11}.
\end{proof}

\begin{lemma}[Boundary term by the Neumann extension]\label{y10}Assume \cref{b47b}, let
 $u\colon \overline D_R\to\R $,
and assume that $\{\theta(p,p),\theta(p,q)\}\subseteq (0,1)$. Then 
$\left|\bterm{\omega}{R}{u}{\NeumannExt{R}{\epsilon}u}\right|\leq 8c^8R\max\{\epsilon^{\theta(p,p)},\epsilon^{\theta(p,q)}\}.$
\end{lemma}
\begin{proof}[Proof of \cref{y10}]
\cref{x33} (with $u\defeq \modi{R}(u\upharpoonright_{\partial D_R})$ and combined with the assumption that $\theta(p,q)\in (0,1)$),  
the assumption that $\forall \,v\in\R^{\partial D_R}\colon \avnorm{\nablaBar (\modi{R}v}{\frac{2q}{q+1}}{\etan{R}}\leq c
\avnorm{\nablaBar v}{\frac{2q}{q+1}}{\etan{R}}
$ in \eqref{d06c}, and \eqref{y05}
ensure 
 that
\begin{align}
\begin{split}
\avnorm{\modi{R}(u\upharpoonright_{\partial D_R})
-\smooth{R}{\epsilon}\modi{R}(u\upharpoonright_{\partial D_R})}{\frac{2p}{p-1}}{\partial D_R}
&\leq 2c^2R\epsilon^{\theta(p,q)}\avnorm{\nablaBar (\modi{R}(u\upharpoonright_{\partial D_R}))}{\frac{2q}{q+1}}{\etan{R}}\\&\leq 2c^3R\epsilon^{\theta(p,q)}\avnorm{\nablaBar u}{\frac{2q}{q+1}}{\etan{R}}\leq 2c^3R\epsilon^{\theta(p,q)}(\overline{\Lambda}(u))^{\nicefrac{1}{2}}.\end{split}\label{y08}
\end{align}
Next, \cref{x33} (with $u\defeq \modi{R}\left( \NeumannExt{R}{\epsilon}u\upharpoonright_{\partial D_R}\right)$ and combined with the assumption that  $\theta(p,p)\in (0,1)$),
the assumption that 
$
\forall\,v\in \R^{\partial D_R}\colon  
 \avnorm{\nablaBar(\modi{R}v)}{\frac{2p}{p+1}}{\etan{R}}\leq c \avnorm{\nablaBar v}{\frac{2p}{p+1}}{\etan{R}}
$
in \eqref{d06c}, 
the boundary regularity in \eqref{y01c},
the assumption that
$c^{-1}\leq \omegaH\leq c $,
\eqref{y01} (combined with
the triangle inequality and Jensen's inequality), \cref{y27}, and
\eqref{y05}  demonstrate that
\begin{align}
\begin{split}
&\avnorm{\modi{R}\left( \NeumannExt{R}{\epsilon}u\upharpoonright_{\partial D_R}\right)
-\smooth{R}{\epsilon}\modi{R}\left( \NeumannExt{R}{\epsilon}u\upharpoonright_{\partial D_R}\right)}{\frac{2p}{p-1}}{\partial D_R}
\leq 2c^2R\epsilon^{\theta(p,p)}\avnorm{\nablaBar \left(\modi{R} \left( \NeumannExt{R}{\epsilon}u\upharpoonright_{\partial D_R}\right)\right)}{\frac{2p}{p+1}}{\etan{R}}
\\
&\leq 2c^3R\epsilon^{\theta(p,p)}\avnorm{\nablaBar  \left(\NeumannExt{R}{\epsilon}u\right)}{\frac{2p}{p+1}}{\etan{R}}
\leq 
2c^4R\epsilon^{\theta(p,p)}\avnorm{\nablaBar  \left(\NeumannExt{R}{\epsilon}u\right)}{\frac{2p}{p+1}}{\enor{R}}\\
&
\leq 2c^5R\epsilon^{\theta(p,p)}\avnorm{(\omegaH)_\funcNu{R} \nablaBar_\funcNu{R}   \left(\NeumannExt{R}{\epsilon}u\right)}{\frac{2p}{p+1}}{\tilde\partial D_R}
 \leq 4c^5R\epsilon^{\theta(p,p)}
\avnorm{
\dualsmooth{R}{\epsilon} (\omega_\funcNu{R} \nablaBar_\funcNu{R} u)}{\frac{2p}{p+1}}{\tilde \partial D_R}\\&\leq 4c^8R\epsilon^{\theta(p,p)}
\avnorm{\omega_\funcNu{R} \nablaBar_\funcNu{R}   u}{\frac{2p}{p+1}}{\tilde\partial D_R}\leq 
8c^8R\epsilon^{\theta(p,p)}
(\overline{\Lambda}(u))^{\nicefrac{1}{2}}.
\end{split}\label{y07}
\end{align}
Combining this, \eqref{y08}, and the triangle inequality yields that
\begin{align}\begin{split}
&\avnorm{\modi{R}\left((u-\NeumannExt{R}{\epsilon}u)\upharpoonright_{\partial D_R}\right)
-\smooth{R}{\epsilon} \modi{R}\left((u-\NeumannExt{R}{\epsilon}u)\upharpoonright_{\partial D_R}\right)}{\frac{2p}{p-1}}{\partial D_R}\\
&\leq \avnorm{\modi{R}(u\upharpoonright_{\partial D_R})
-\smooth{R}{\epsilon}\modi{R}(u\upharpoonright_{\partial D_R}) }{\frac{2p}{p-1}}{\partial D_R} +
\avnorm{\modi{R} \left(\NeumannExt{R}{\epsilon}u\upharpoonright_{\partial D_R}\right)
-\smooth{R}{\epsilon}\modi{R}\left(\NeumannExt{R}{\epsilon}u\upharpoonright_{\partial D_R}\right)}{\frac{2p}{p-1}}{\partial D_R}\\
&\leq 8c^8R\max\{\epsilon^{\theta(p,p)},\epsilon^{\theta(p,q)}\}
(\overline{\Lambda}(u))^{\nicefrac{1}{2}}.
\end{split}\end{align}
\cref{y03}, H\"older's inequality,    and \eqref{y05}  hence prove that
\begin{align}
\begin{split}
&\left|\bterm{\omega}{R}{u}{\NeumannExt{R}{\epsilon}u}\right|\\
&=
\frac{1}{|\tilde\partial D_R|}\left|
\sum_{x\in \tilde\partial D_R} \left[\left(\modi{R}\left((u-\NeumannExt{R}{\epsilon}u)\upharpoonright_{\partial D_R}\right)\right)(x)
-\left(\smooth{R}{\epsilon} \modi{R}\left((u-\NeumannExt{R}{\epsilon}u)\upharpoonright_{\partial D_R}\right)\right)(x)
 \right](\omega \nablaBar  u)_{\funcNu{R}(x)}\right|\\
&\leq 
\avnorm{\modi{R}\left((u-\NeumannExt{R}{\epsilon}u)\upharpoonright_{\partial D_R}\right)
-\smooth{R}{\epsilon} \modi{R}\left((u-\NeumannExt{R}{\epsilon}u)\upharpoonright_{\partial D_R}\right)}{\frac{2p}{p-1}}{\partial D_R}
 \avnorm{(\omega \nablaBar  u)_{\funcNu{R}}}{\frac{2p}{p+1}}{\tilde \partial D_R}.
\\
&\leq 
8c^9R\max\{\epsilon^{\theta(p,p)},\epsilon^{\theta(p,q)}\}
\overline{\Lambda}(u)
.
\end{split}
\end{align}
This completes the proof of \cref{y10}.
\end{proof}
\cref{d11} below is an easy calculation. However, together with \cref{y03,y10} it explains why we call $q\geq p$ the Dirichlet case and $p\geq q$ the Neumann case. 
\begin{lemma}\label{d11}
Assume \cref{b47b} and assume that 
${1}/{p}+{1}/{q}\leq {2}/{(d-1)}$. Then
\begin{enumerate}[i)]
\item it holds in the case $q\geq p $ that $0<\theta(p,q)\leq \theta(q,q)<1$ 
and $\epsilon^{\theta(q,q)}\leq\epsilon^{\theta(p,q)} $
and
\item it holds in the case $p\geq q$ that 
$0<\theta(p,q)\leq \theta(p,p)<1$
and $\epsilon^{\theta(p,p)}\leq\epsilon^ {\theta(p,q)}$.
\end{enumerate}
\end{lemma}
\begin{proof}[Proof of \cref{d11}]
The assumptions of \cref{d11} and \eqref{y22} show that
\begin{enumerate}[i)]
\item 
it holds in the case 
 $d\geq 3$ and $ q\geq p $ that
\begin{align}
\theta(p,q)-\theta(q,q)=\left[
1-(d-1)\left(\frac{1}{2p}+\frac{1}{2q}\right)\right]-\left[
1-\frac{(d-1)}{q}\right]= (d-1)\left(\frac{1}{2q}-\frac{1}{2p}\right)\leq 0
,\end{align}
\item it holds in the case 
  $d\geq 3$ and $p\geq q $ that 
\begin{align}
\theta(p,q)-\theta(p,p)=
\left[1-(d-1)\left(\frac{1}{2p}+\frac{1}{2q}\right)\right]-\left[1-\frac{(d-1)}{p}\right]= (d-1)\left(\frac{1}{2p}-\frac{1}{2q}\right)\leq 0,
\end{align}
\item  it holds in the case 
  $d=2$ and $q\geq p $ that
\begin{align}
\frac{\theta(p,q)}{\theta(q	,q)}=\frac{(q-1)p}{q(p+1)}\frac{q+1}{q-1}= \frac{1+\frac1q}{1+\frac1p}\leq 1,\end{align}
and
\item it holds in the case 
 $d=2$ and $p\geq q $ that 
\begin{align}
\frac{\theta(p,q)}{ \theta(p,p)}=\frac{(q-1)p}{q(p+1)}\frac{p+1}{p-1}= \frac{1-\frac1q}{1-\frac1p}\leq  1.\end{align}

\end{enumerate}
This completes the proof of  \cref{d11}.
\end{proof}
Combining \cref{y03,y10,d11} we obtain
\cref{y26} below.
\begin{corollary}[Boundary term]\label{y26}
Assume \cref{b47b}, let
 $u\colon \overline D_R\to\R $,
 and assume that 
${1}/{p}+{1}/{q}\leq {2}/{(d-1)}$.
Then 
\begin{enumerate}[i)]
\item it holds in the case $q\geq p$ that $\left|\bterm{\omega}{R}{u}{\DirichletExt{R}{\epsilon}u}\right|\leq8c^8R\epsilon^{\theta(p,q)}\overline{\Lambda}(u)$ and
\item it holds in the case $p\geq q$ that $\left|\bterm{\omega}{R}{u}{\NeumannExt{R}{\epsilon}u}\right|\leq 8c^8R\epsilon^{\theta(p,q)}\overline{\Lambda}(u)$.
\end{enumerate}
\end{corollary}

%\subsection{Energy of the harmonic extension}
\begin{corollary}[Energy of the harmonic extensions]\label{y15}
Assume \cref{b47b}, assume that ${1}/{p}+{1}/{q}\leq {2}/{(d-1)}$, and
let
$R\in \N$, $u\colon\partial D_R\to\R  $, $\epsilon\in (0,\epsilon_0)$,
 $\omega\in\Omega$. Then
\begin{enumerate}[i)]
\item\label{y15a} it holds in the case $q\geq p$ that 
$
\avnorm{\omegaH(\nablaBar (\DirichletExt{R}{\epsilon}u) )^2}{1}{E_R}\leq 2c^6 \overline{\Lambda}(u) $ and
\item\label{y15b} it holds in the case $p\geq q$ that 
$
\avnorm{\omegaH(\nablaBar (\NeumannExt{R}{\epsilon}u))^2}{1}{E_R}
\leq 8c^{12} \overline{\Lambda}(u)$.
\end{enumerate}
\end{corollary}
%%%%%%%%%%%%%%%%%%%%%%%%%%%
\begin{proof}[Proof of \cref{y15}]First, note that
\eqref{y20} and the assumption that $1/p+1/q\leq 2/(d-1)$ imply that
\begin{align}\label{y14b}
\frac{1}{\alpha(q)}=\left[\left(\frac{q-1}{2q}+\frac{1}{d-1}\right)\1_{d\geq 3}+\1_{d=2}\right]\geq \frac{p+1}{2p}\quad\text{and}\quad \alpha(q)\leq \frac{2p}{p+1}.
\end{align}

\cref{b66} (with $w\defeq v-a$, $g\defeq \omegaH\nabla v$, $f\defeq 0$ for $a\in\R$, $v\in \{\DirichletExt{R}{\epsilon}u, \NeumannExt{R}{\epsilon}u\}$), H\"older's inequality,
the assumption that 
$|\tilde{\partial} D_R|/|E_R|\leq c/R$,
the assumption  that $c^{-1}\leq \omegaH\leq c$, Sobolev's inequality in \eqref{y22a}, and Jensen's inequality (combined with \eqref{y14b}) imply that for all $v\in \{\DirichletExt{R}{\epsilon}u, \NeumannExt{R}{\epsilon}u\}$ it holds that
\begin{align}\begin{split}
&\avnorm{\omegaH(\nablaBar v)^2}{1}{E_R}
= 
\inf_{a\in\R}\left[
\frac{|\tilde{\partial}D_R|}{|E_R|}
\frac{2}{|\tilde{\partial}D_R|}
\left|\sum_{x\in \tilde \partial D_R} (\omegaH
 \nablaBar v)_{\funcNu{R}(x)}(v(x)-a)\right|\right]\\
&
\leq 2cR^{-1}\avnorm{\omegaH\circ\funcNu{R}}{\infty}{\tilde{\partial}D_R}
\avnorm{\nablaBar_\funcNu{R} v}{\frac{2q}{q+1}}{\tilde{\partial}D_R}
\left[
\inf_{a\in\R}\avnorm{v-a}{\frac{2q}{q-1}}{\tilde\partial D_R}\right]\\
&
\leq 2c^2R^{-1}
\avnorm{\nablaBar v}{\frac{2q}{q+1}}{\enor{R}}cR
\avnorm{\nablaBar v}{\alpha(q)}{\etan{R}}
\leq 2c^3
\avnorm{\nablaBar v}{\frac{2q}{q+1}}{\enor{R}}
\avnorm{\nablaBar v}{\frac{2p}{p+1}}{\etan{R}}.\end{split}\label{d14}
\end{align}
Furthermore, 
\eqref{y01c} and Jensen's inequality (combined with the fact that 
$(1,\infty)\ni x\mapsto 2x/(x+1) $ is non-decreasing) prove that in the case $q\geq p $ it holds that
\begin{align}
\avnorm{\nablaBar (\DirichletExt{R}{\epsilon}u)}{\frac{2q}{q+1}}{\enor{R}}
\avnorm{\nablaBar( \DirichletExt{R}{\epsilon}u)}{\frac{2p}{p+1}}{\etan{R}}
\leq c
\avnorm{\nablaBar (\DirichletExt{R}{\epsilon}u)}{\frac{2q}{q+1}}{\etan{R}}^2
\leq c^3 \avnorm{\nablaBar u}{\frac{2q}{q+1}}{\etan{R}}^2\leq c^3\overline{\Lambda}(u).
\end{align}
This and \cref{d14} imply  in the case  $q\geq p$ that
$
\avnorm{\omegaH(\nablaBar (\DirichletExt{R}{\epsilon}u) )^2}{1}{E_R}\leq 2c^3 c^3\overline{\Lambda}(u)= 2c^6R \overline{\Lambda}(u)$. This shows \cref{y15a}.
In addition, the boundary regularity in
\eqref{y01c}, Jensen's inequality (combined with the fact that 
$(1,\infty)\ni x\mapsto 2x/(x+1) $ is non-decreasing), and the assumption that $c^{-1}\leq \omegaH\leq c$ prove  in the case $p\geq q $  that
\begin{align}
\avnorm{\nablaBar (\NeumannExt{R}{\epsilon}u)}{\frac{2q}{q+1}}{\enor{R}}
\avnorm{\nablaBar( \NeumannExt{R}{\epsilon}u)}{\frac{2p}{p+1}}{\etan{R}}
\leq c
\avnorm{\nablaBar (\NeumannExt{R}{\epsilon}u)}{\frac{2p}{p+1}}{\enor{R}}^2
\leq  c^{3}
\avnorm{\omegaH\nablaBar (\NeumannExt{R}{\epsilon}u)}{\frac{2p}{p+1}}{\enor{R}}^2.\label{d15}
\end{align}
Furthermore, \eqref{y01} (combined with the triangle inequality and Jensen's inequality), \cref{y27}, and \eqref{y05} demonstrate that
\begin{align}\begin{split}
\avnorm{\omegaH\nablaBar (\NeumannExt{R}{\epsilon}u)}{\frac{2p}{p+1}}{\enor{R}}
&= \avnorm{(\omegaH\nablaBar (\NeumannExt{R}{\epsilon}u))_ \funcNu{R}}{\frac{2p}{p+1}}{\tilde{\partial}D_R}
\leq 
2\avnorm{\dualsmooth{R}{\epsilon}((\omega \nablaBar u)_ \funcNu{R} )}{\frac{2p}{p+1}}{\tilde{\partial}D_R}\\
&\leq 2c ^3
\avnorm{(\omega \nablaBar u)_ \funcNu{R} }{\frac{2p}{p+1}}{\tilde{\partial}D_R}
= 2c ^3
\avnorm{\omega \nablaBar u }{\frac{2p}{p+1}}{\tilde{\partial}D_R}\leq 2c^3(\overline{\Lambda}(u))^{\nicefrac{1}{2}}.
\end{split}\end{align}
This, \eqref{d14}, and \eqref{d15} imply in the case $p\geq q$ that
$
\avnorm{\omegaH(\nablaBar (\NeumannExt{R}{\epsilon}u))^2}{1}{E_R}
\leq 2c^3 c^3 (2c^3)^2\overline{\Lambda}(u)= 8c^{12} R \overline{\Lambda}(u)$. This shows \cref{y15b}. 
 The proof of \cref{y15} is thus completed.
\end{proof}

\subsection{The annulus term}

\begin{lemma}\label{y25}
Assume \cref{b47b}. Then 
\begin{enumerate}[i)]
\item it holds in the case $q\geq p$ that 
$\avnorm{\nablaBar (\DirichletExt{R}{\epsilon}u)}{\max\left\{\frac{4p}{p-1},
\frac{ 4q}{q-1}\right\}}{E_R}
\leq 
c \epsilon^{-(d-1)\frac{q+1}{2q} }
(\overline{\Lambda}(u))^{\nicefrac{1}{2}}$ and
\item it holds in the case $p\geq q$ that 
$
\avnorm{\nablaBar (\NeumannExt{R}{\epsilon}u)}{\max\left\{\frac{4p}{p-1},
\frac{ 4q}{q-1}\right\}}{E_R}
\leq2c^4\epsilon ^{-(d-1)\frac{p+1}{2p}} (\overline{\Lambda}(u))^{\nicefrac{1}{2}}$.
\end{enumerate}
\end{lemma}
\begin{proof}[Proof of \cref{y25}]
First,
\eqref{y01d}, the assumption that $\forall x\in \partial D_R\colon   (\DirichletExt{R}{\epsilon}u)(x)=(\smooth{R}{\epsilon}u )(x)$ in \eqref{y01b}, 
the assumption that $\avnorm{\smooth{R}{\epsilon}u}{\infty}{\partial D_R}\leq \epsilon^{-(d-1)\frac{q+1}{2q}}
\avnorm{u}{\frac{2q}{q+1}}{\partial D_R}$ in \eqref{d06c},
and \eqref{y05} imply that 
\begin{align}\begin{split}
&\avnorm{\nablaBar (\DirichletExt{R}{\epsilon}u)}{\max\left\{\frac{4p}{p-1},
\frac{ 4q}{q-1}\right\}}{E_R}
\leq c\avnorm{\nablaBar (\DirichletExt{R}{\epsilon}u)}{\max\left\{\frac{4p}{p-1},
\frac{ 4q}{q-1}\right\}}{\etan{R}}
=c\avnorm{\nablaBar (\smooth{R}{\epsilon}u)}{\max\left\{\frac{4p}{p-1},
\frac{ 4q}{q-1}\right\}}{\etan{R}}\\&
\leq c\avnorm{\nablaBar (\smooth{R}{\epsilon}u)}{\infty}{\etan{R}}\leq c \epsilon^{-(d-1)\frac{q+1}{2q} }
\avnorm{\nablaBar u}{\frac{2q}{q+1}}{\etan{R}}\leq 
c \epsilon^{-(d-1)\frac{q+1}{2q} }
(\overline{\Lambda}(u))^{\nicefrac{1}{2}}.
\label{c02a}
\end{split}
\end{align}
Next,
\eqref{y01d}, the assumption that 
$c^{-1}\leq \omegaH\leq c$,
\eqref{y01} (combined with the triangle inequality, Jensen's inequality, and \cref{d09}),  and \eqref{y05}
 show that 
\begin{align}\begin{split}
&\avnorm{\nablaBar (\NeumannExt{R}{\epsilon}u)}{\max\left\{\frac{4p}{p-1},
\frac{4q}{q-1}\right\}}{E_R}
\leq c \avnorm{\nablaBar (\NeumannExt{R}{\epsilon}u)}{\max\left\{\frac{4p}{p-1},
\frac{4q}{q-1}\right\}}{\enor{R}}\\
&
\leq 
c^2 \avnorm{\omegaH\nablaBar (\NeumannExt{R}{\epsilon}u)}{\max\left\{\frac{4p}{p-1},
\frac{ 4q}{q-1}\right\}}{\enor{R}}
\leq c^2
\avnorm{(\omegaH \nablaBar (\NeumannExt{R}{\epsilon}u)_\funcNu{R}}{\infty}{\tilde\partial D_R}\\
&\leq 2c^4\epsilon^{-(d-1)\frac{p+1}{2p}}
\avnorm{(\omega \nablaBar u)_\funcNu{R} }{\frac{2p}{p+1}}{\tilde{\partial}D_R}\leq2c^4\epsilon ^{-(d-1)\frac{p+1}{2p}} (\overline{\Lambda}(u))^{\nicefrac{1}{2}}.
\label{c02b}
\end{split}
\end{align}
The proof of \cref{y25} is thus completed.
\end{proof}
\begin{corollary}\label{d16}
Assume \cref{b47b}, 
let $\Lambda\in [0,\infty)$ be  given by 
$
\Lambda=\avnorm{\omega}{p}{E_R}+
\avnorm{\omega^{-1}}{q}{E_R}$,
let $u\colon\overline{D}_R\to\R$,
let $v\in\{\DirichletExt{R}{\epsilon}u,\NeumannExt{R}{\epsilon}u\} $ be a function which satisfies 
that in the case $q> p$ it holds that $v=\DirichletExt{R}{\epsilon}u $ and 
 in the case $p>q$ it holds that $v=\NeumannExt{R}{\epsilon}u $, and let $\rho \in [1,R/2]\cap\N$ satisfy that $|E_R\setminus E_{R-\rho}|\leq c(\rho/R)|E_R|$. Then it holds that
\begin{align}\begin{split}
&\avnorm{\omega(\nablaBar v)^2\1_{E_R\setminus E_{R-\rho}}}{1}{E_R}
+
\avnorm{\omega^{-1}(\nablaBar v)^2\1_{E_R\setminus E_{R-\rho}}}{1}{E_R}\\&\leq 
8c^9  \left(\frac{\rho}{R}\right)
^{\min\left\{\frac{p-1}{2p},
\frac{q-1}{2q}\right\}}
\epsilon^{-(d-1)\min \left\{\frac{p+1}{p}, \frac{q+1}{q}\right\}} \Lambda\overline{\Lambda}(u).
\end{split}\end{align}
\end{corollary}
\begin{proof}[Proof of \cref{d16}]\cref{y25} combined with the assumption on $v$ and a simple case distinction implies that
\begin{align}
\avnorm{(\nablaBar v)^2}{\max\left\{\frac{2p}{p-1},
\frac{ 2q}{q-1}\right\}}{E_R}=\avnorm{\nablaBar v}{\max\left\{\frac{4p}{p-1},
\frac{ 4q}{q-1}\right\}}{E_R}^2\leq 
4c^8
\epsilon^{-(d-1)\min\left\{\frac{p+1}{p}, \frac{q+1}{q}\right\}} \overline{\Lambda}(u).\end{align}
H\"older inequality, the assumption that
$|E_R\setminus E_{R-\rho}|\leq c(\rho/R)|E_R|$, 
and the assumption that $c\geq 1$ hence show that
\begin{align}\begin{split}
\avnorm{\omega(\nablaBar v)^2\1_{E_R\setminus E_{R-\rho}}}{1}{E_R}
&\leq \avnorm{\omega}{p}{E_R}
\left(\frac{|E_R\setminus E_{R-\rho}|}{|E_R|}\right)^{1-\frac{1}{p}-\min\left\{\frac{p-1}{2p},
\frac{q-1}{2q}\right\}}
\avnorm{(\nablaBar v)^2}{\max\left\{\frac{2p}{p-1},
\frac{2q}{q-1}\right\}}{E_R}\\
&\leq \Lambda c \left(\frac{\rho}{R}\right)
^{\min\left\{\frac{p-1}{2p},
\frac{q-1}{2q}\right\}}
4c^8
\epsilon^{-(d-1)\min\left\{\frac{p+1}{p}, \frac{q+1}{q}\right\}} 
\overline{\Lambda}(u)
\end{split}\end{align}
and
\begin{align}\begin{split}
\avnorm{\omega^{-1}(\nablaBar v)^2\1_{E_R\setminus E_{R-\rho}}}{1}{E_R}
&\leq \avnorm{\omega^{-1}}{q}{E_R}
\left(\frac{|E_R\setminus E_{R-\rho}|}{|E_R|}\right)^{1-\frac{1}{p}-\min\left\{\frac{p-1}{2p},\frac{q-1}{2q}\right\}}
\avnorm{(\nablaBar v)^2}{\max\left\{\frac{2p}{p-1},
\frac{2q}{q-1}\right\}}{E_R}\\&\leq \Lambda c \left(\frac{\rho}{R}\right)
^{\min\left\{\frac{p-1}{2p},
\frac{q-1}{2q}\right\}}
4c^8
\epsilon^{-(d-1)\min\left\{\frac{p+1}{p}, \frac{q+1}{q}\right\}} \overline{\Lambda}(u).
\end{split}\end{align}
This completes the proof of \cref{d16}.
\end{proof}

\small{\bibliography{literature}{}
\bibliographystyle{plain}}
%%%%%%%%%%%%%%%%%%
\end{document}